\documentclass[reqno,centertags,11pt]{amsart}

\usepackage{amssymb,amsmath,amsfonts,amssymb}

\usepackage{latexsym}
\usepackage{amsfonts}
\usepackage{amssymb}
\usepackage{amsthm}
\usepackage{verbatim}
\usepackage{enumerate}
\usepackage[unicode]{hyperref}

\newtheorem{definition}{Definition}[section]
\newtheorem{theorem}{Theorem}[section]
\newtheorem{lemma}[theorem]{Lemma}

\newtheorem{proposition}[theorem]{Proposition}

\theoremstyle{remark}
\newtheorem{remark}[theorem]{Remark}

\newcommand{\ZZ}{\mathbb{Z}}

\newcommand{\NN}{\mathbb{N}}
\newcommand{\PP}{\mathcal{P}}

\newcommand{\RR}{\mathbb{R}}

\newcommand{\dd}{\mathrm{d}}
\newcommand{\Id}{\mathrm{Id}}

\DeclareMathOperator{\Div}{div}

\allowdisplaybreaks

\numberwithin{equation}{section}

\begin{document}
\title[Global well-posedness for 2D fractional INS equations]{Global well-posedness for 2D fractional inhomogeneous Navier-Stokes equations with rough density}

\author{Yatao Li}
\address{Laboratory of Mathematics and Complex Systems (MOE), School of Mathematical Sciences, Beijing Normal University, Beijing 100875, P.R. China}
\email{liyatao\_maths@bnu.edu.cn}
\author{Qianyun Miao}
\address{School of Mathematics and Statistics, Beijing Institute of Technology, Beijing 100081, P. R. China}
\email{qianyunm@bit.edu.cn}
\author[Liutang Xue]{Liutang Xue}
\address{Laboratory of Mathematics and Complex Systems (MOE), School of Mathematical Sciences, Beijing Normal University, Beijing 100875, P.R. China}
\email{xuelt@bnu.edu.cn}

\date{}

\begin{abstract}
The paper concerns with the global well-posedness issue of the 2D incompressible inhomogeneous Navier-Stokes (INS) equations with fractional dissipation and rough density.
We first establish the $L^q_t(L^p)$-maximal regularity estimate for the generalized Stokes system with fractional dissipation,
and then we employ it to obtain the global existence of solution for the 2D fractional INS equations with large velocity field,
provided that the $L^2\cap L^\infty$-norm of density minus constant 1 is small enough.
Moreover, by additionally assuming that the density minus 1 is sufficiently small in the norm of some multiplier spaces,
we prove the uniqueness of the constructed solution by using the Lagrangian coordinates approach.
We also consider the density patch problem for the 2D fractional INS equations,
and show the global persistence of $C^{1,\gamma}$-regularity of the density patch boundary when the piecewise jump of density is small enough.
\end{abstract}
\maketitle


\section{Introduction}\label{INTR}
\setcounter{section}{1}\setcounter{equation}{0}

We consider the 2D incompressible fractional inhomogeneous Navier-Stokes (abbr. INS) equations:
\begin{equation}\label{eq.frc-INNS}
  \left\{\begin{array}{ll}
  \partial_t \rho+\Div(\rho u)=0,\\
  \rho(\partial_t u+u\cdot\nabla u)+\nu\Lambda^{2\alpha} u+\nabla \pi=0,\\
  \Div u=0,\\
  (\rho,u)|_{t=0}(x)=(\rho_0(x),u_0(x)),
\end{array}\right.
\end{equation}
where $x\in \RR^2$, $\nu>0$ is the kinematic viscosity coefficient, the scalar $\rho$ is the density, $u=(u^1,u^2)$ represents the velocity field,
and $\pi$ stands for the pressure of the fluid. The fractional Laplacian operator $\Lambda^{2\alpha}:=(-\Delta)^\alpha$ for $\alpha\in (0,1)$ is defined by the Fourier transform via
\begin{align*}
  \widehat{\Lambda^{2\alpha} f}(\xi) = |\xi|^{2\alpha} \widehat{f}(\xi),
\end{align*}
where $ \widehat{f}$ is the Fourier transform of $f$. From the stochastic process point of view, the fractional Laplacian $\Lambda^{2\alpha}$
is an infinitesimal generator of the symmetric $2\alpha$-stable L$\acute{\text e}$vy process (e.g. see \cite{Apple09}).
When $\rho\equiv 1$ and $\nu=1$, \eqref{eq.frc-INNS} reduces to the 2D incompressible fractional (homogeneous) Navier-Stokes equations
\begin{equation}\label{eq.2dNS}
\left\{\begin{array}{ll}
  \partial_t u+u\cdot\nabla u+\Lambda^{2\alpha} u +\nabla \pi=0,\\
  \Div u=0,\quad
  u|_{t=0}(x)=u_0(x),
\end{array}\right.
\end{equation}
where $x\in \RR^2$, $\alpha\in(0,1)$. The system \eqref{eq.2dNS} was used in modeling a fluid motion with internal friction interaction (\cite{MGSIG}).
Compared with \eqref{eq.2dNS}, the density-dependent system \eqref{eq.frc-INNS} can describe the dynamics of flows with variable densities.

Recently, there have been many works studying on the fractionally dissipative systems arising from many physical applications.
The fractional Laplacian operators describe various phenomena in hydrodynamics \cite{CV10},
fractional quantum mechanics \cite{Laskin00}, anomalous diffusion in semiconductor growth \cite{Wo01},
physics and chemistry \cite{MetKla00,SZF95} and so on.
We also mention a quite related fractionally dissipative model, known as the Euler-alignment system,
\begin{equation}\label{eq.EA}
  \left\{\begin{array}{ll}
  \partial_t \rho+\Div(\rho u)=0,\\
  \rho(\partial_t u + u\cdot\nabla u) + \mathcal{D}(u,\rho) + \nabla p(\rho) = 0, \\
  \mathcal{D}(u,\rho) = \rho \big( u \Lambda^{2\alpha} \rho - \Lambda^{2\alpha} (u\rho)  \big)
  =c_\alpha\, \rho \int_{\RR^d} \frac{u(x)-u(y)}{|x-y|^{d+2\alpha}} \rho(y) \dd y, \\
  (\rho,u)|_{t=0}(x)=(\rho_0(x),u_0(x)),
\end{array}\right.
\end{equation}
where $x\in \RR^d$, $d\geq 1$, $\rho$ is the density, $\alpha\in (0,1)$, $u$ is the velocity field, and $p(\rho)=\rho^\gamma$,
$\gamma\geq 1$ is the pressure. The system \eqref{eq.EA} is the hydrodynamic limit model (\cite{KMT15})
of the Cucker-Smale kinetic model which describes the flocking phenomenon for animal groups,
and one can see \cite{CDR20,CTT20} for the recent mathematical studies.
The model \eqref{eq.EA} can be viewed as a compressible Euler system with fractional dissipation, 
and thus one may formally view system \eqref{eq.frc-INNS} as an intermediate model between fractional Navier-Stokes \eqref{eq.2dNS}
and the Euler-alignment system \eqref{eq.EA}.

When $\alpha=1$ and $\nu=1$, the system \eqref{eq.frc-INNS} corresponds to the classical incompressible inhomogeneous Navier-Stokes equations:
\begin{equation}\label{eq.INNS}
  \left\{\begin{array}{ll}
  \partial_t \rho+\Div(\rho u)=0,\\
  \rho(\partial_t u+u\cdot\nabla u)-\Delta u+\nabla \pi=0,\\
  \Div u=0,\qquad  (\rho,u)|_{t=0}=(\rho_0,u_0),
\end{array}\right.
\end{equation}
where $x\in \RR^d$, $d=2,3$. The INS system $\eqref{eq.INNS}$ originates in describing the dynamics of geophysical flows which are incompressible and also have variable densities (\cite{PLions96}).
When $\rho\equiv 1$, the system $\eqref{eq.INNS}$ becomes exactly the classical incompressible Navier-Stokes equations.
System $\eqref{eq.INNS}$ has been extensively investigated in recent decades.
When the density is bounded and $(\rho_0,u_0)$ is of finite energy,
the global existence of weak solutions with finite energy for INS system \eqref{eq.INNS} was obtained in
\cite{Simon90,PLions96}. If the density is bounded and smooth enough
(at least continuous, and have some fractional derivatives in Lebesgue spaces),
the global existence and uniqueness results can be obtained for the INS system \eqref{eq.INNS}
in dimension two with large initial data,
and in dimension three under a smallness condition of the velocity (e.g. see \cite{Lady75,Danch03,CK03,HuangPaicuZhang13}).
The case of the rough density admitting piecewise constant densities is of much interest,
which can be used in modeling a mixture of two fluids.
Danchin and Mucha \cite{DanchMucha12} developed a novel Lagrangian coordinates approach to address the uniqueness in the rough density case,
and proved the global existence and uniqueness of solution to INS system \eqref{eq.INNS}
in a critical regularity framework, 
provided that a smallness condition over the initial velocity and the jumps of initial density was employed.
Huang, Paicu and Zhang \cite{HuangPZh13} in the 2D case removed the smallness condition on initial velocity in \cite{DanchMucha12}
and got the global well-poseness of solution by assuming that the jumps of initial density is sufficiently small
(depending on the size of the velocity); Danchin and Mucha \cite{DanM13} proved the local-in-time existence and uniqueness result without the smallness condition on the jumps of initial density as well as the global-in-time result similar to \cite{HuangPZh13}.
Paicu, Zhang and Zhang \cite{PaicuZhZh13} moreover showed the global well-posedness of solution to the INS system \eqref{eq.INNS}
with initial density only being bounded from above and below by some positive constants. 
For the bounded initial density admitting vacuum states, Danchin and Mucha \cite{DanchMuch19} obtained the global existence and uniqueness result.
We mention that Mucha, the third author and Zheng \cite{MXZ19}
obtained the global existence and uniqueness of strong solution to the 3D INS system \eqref{eq.INNS} associated with a type of large (2D-like) initial velocity field and
the initial density sufficiently close to constant 1 which admits discontinuous data.
Note that all the above works dealing with the rough density case essentially apply the Lagrangian coordinates method to show the uniqueness
(one can also see Constantin et al \cite{Con15,CJ19} for such a method applied to the related hydrodynamic models).

Another interesting and closely related results on the discontinuous density for INS system \eqref{eq.INNS}
is the study of the so-called \emph{density patch problem}, which was firstly raised by P.-L. Lions in \cite{PLions96}
and was concerned with the density patch $\rho_0 = 1_{\Omega_0}$ with $\Omega_0\subset \RR^d$ a smooth simple-connected domain.
Since the density solves the transport equation, it formally yields $\rho(t,x) = \rho_0(X_t^{-1}(x))$
with $X_t^{-1}$ the inverse of $X_t$ and $X_t(\cdot)$ the particle-trajectory satisfying
\begin{align}\label{eq:flow0}
  \frac{\dd}{\dd t} X_t(x) = u(t, X_t(x)),\qquad X_t(x)|_{t=0} = x,
\end{align}
thus one has that $\rho(t,x)= 1_{\Omega(t)}(x)$ with $\Omega(t):= X_t(\Omega_0)$.
The \emph{density patch problem} asks that whether or not the initial smoothness of density patch boundary can be globally persisted by the evolution?
The aforementioned works \cite{DanchMucha12,DanM13,HuangPZh13,MXZ19,PaicuZhZh13} ensured the global well-posedness of solution for INS system \eqref{eq.INNS} associated with the density patch initial data in various situations,
and showed the global persistence of either $C^1$- or $C^{1,\gamma}$-regularity of the evolutionary patch boundary.
One can also refer to \cite{DanchZhxin17,Gancedo18,LiaoZhang18,LiaoZhang19,LiaoLiu12,ChenLi20}
concerning the global persistence of higher boundary regularity of the density patch.

When $\nu=0$, the system \eqref{eq.frc-INNS} becomes the density-dependent incompressible Euler equations, which contains the classical incompressible Euler equations as a special case.
If the initial data has high enough regularity (at least the initial velocity is Lipschitz continuous and the gradient of initial density is continuous and bounded),
the local existence and uniqueness results for the density-dependent Euler equations can be obtained in many kinds of functional spaces,
and one can refer to \cite{BVV80,Dan10,DanFra11} and references therein.

For the fractional Navier-Stokes equations \eqref{eq.2dNS},
the global well-posedness results in dimension two as well as the local well-posedness results in dimension three have been obtained by Wu \cite{Wu05} in the framework of Besov spaces and by Zhang \cite{Zhang12} using the probabilistic method.
One can also see \cite{CDLM20,LMZ19} and references therein for other interesting progress.

As for the fractional INS equations \eqref{eq.frc-INNS} with $\nu>0$, the existing literatures \cite{FangZi13,WangYe18} only consider the hyper-dissipative case,
i.e. $\alpha \geq \frac{5}{4}$ in dimension three, and the global well-posedness results have been established in such a case.

Although the theoretical study of system \eqref{eq.frc-INNS} with $\nu>0$ and $0<\alpha<1$ are few so far,
it can be viewed as an interesting intermediate model between INS equations \eqref{eq.INNS} and density-dependent Euler equations,
as well as between fractional Navier-Stokes \eqref{eq.2dNS} and Euler-alignment system \eqref{eq.EA}, thus it deserves more mathematical concerns.
Our main goal in this paper is to show the global-in-time existence and uniqueness result of 2D fractional INS equations \eqref{eq.frc-INNS}
with large initial velocity field and rough initial density which admits jump discontinuity.
We restrict on the case $\frac{1}{2}<\alpha<1$, which is reasonable from the maximal regularity of fractional Laplacian operator $\Lambda^{2\alpha}$
and the needing that the velocity field should be Lipschitz continuous at least.

We also remark that it seems very hard to generalize the global results of \cite{PaicuZhZh13,DanchMuch19} to the 2D fractional INS equations \eqref{eq.frc-INNS}
with $\frac{1}{2}<\alpha <1$ and the density being merely bounded.
The partial reason lies on the internal difference between the $\alpha=1$ case and the $\alpha<1$ case\footnote{This is analogous to the difference of the 2D Navier-Stokes equations \eqref{eq.2dNS} between $\alpha=1$ case and $\alpha<1$ case;
however, the system \eqref{eq.2dNS} has additional uniform bounded quantity, that is, the vorticity $\omega= \mathrm{curl}\,u$ is uniformly bounded, which makes the $L^2$-energy supercritical $\alpha<1$ case be globally well-posed.}:
noting that the 2D fractional INS equations \eqref{eq.frc-INNS} is scale-invariant under the following transformation
\begin{align*}
  \rho(x,t)\mapsto \rho(\lambda x,\lambda^{2\alpha} t),\quad u(x,t)\mapsto \lambda^{2\alpha-1} u(\lambda x, \lambda^{2\alpha} t),\quad
  \pi(x,t) \mapsto \lambda^{4\alpha -2} \pi(\lambda x, \lambda^{2\alpha} t),
\end{align*}
for every $\lambda>0$, and in combination with the classical $L^2$-energy estimate,
one can view the $\alpha=1$ case as the energy critical case while the $\alpha <1$ case as the energy supercritical case;
thus the weighted energy estimates (with some time function as weights) used in \cite{PaicuZhZh13,DanchMuch19} just work for the critical $\alpha=1$ case,
but will not directly extend to the supercritical $\alpha<1$ case.

Inspired by \cite{HuangPZh13,Mucha01,Mucha08,MXZ19}, we here compare the solution of the 2D fractional INS equations \eqref{eq.frc-INNS}
with the large solution of the 2D system \eqref{eq.2dNS} and study the stability issue. We assume $\nu=1$ for brevity.
More precisely, let $ \bar{u}(x,t)=(\bar{u}^{1}, \bar{u}^2 )(x,t)$ be a two-dimensional vector field solving the 2D fractional Navier-Stokes equations \eqref{eq.2dNS} with initial data $u_0$,
and denote by
\begin{align}
  a := \rho-1,\quad w:=u-\bar{u},\quad p := \pi-\bar{\pi}.
\end{align}
Then we mainly investigate the following perturbed system
\begin{equation}\label{eq.INS-2dNS}
\left\{\begin{array}{ll}
  \partial_t a + u\cdot\nabla a=0, \\
  \partial_t w + u\cdot \nabla w + \Lambda^{2\alpha} w + \nabla p = F,\\
  \Div w=0,\\
   (a,w)|_{t=0}=(\rho_0-1, 0),
\end{array}\right.
\end{equation}
where
\begin{equation}\label{eq.F}
\begin{split}
  F := -a \partial_t w -a \partial_t \bar{u}  - a ( u \cdot\nabla w) - a(\bar{u}\cdot\nabla \bar{u}) - \rho ( w\cdot\nabla \bar{u}). 
\end{split}
\end{equation}

Our main result reads as follows.
\begin{theorem}\label{Main.thm}
Let $\frac 12<\alpha<1,\, p>\frac 2{2\alpha-1}$ and $u_0\in H^1\cap\dot B^{\alpha}_{p,2}(\RR^2)$, $\rho_0-1\in L^2\cap L^\infty(\RR^2)$.
There exists a generic constant $c_0\in (0,1)$ depending only on $\alpha,p$ such that if
\begin{equation}\label{con.Max.B0}
\begin{split}
  \|\rho_0-1\|_{L^2 \cap L^\infty}  \,
  \leq c_0 \exp\Big\{- c_0^{-1}\| u_0\|_{H^1\cap \dot B^\alpha_{p,2}}^2 \Big\} ,
\end{split}
\end{equation}
then the 2D fractional INS system \eqref{eq.frc-INNS} has a global-in-time strong  solution
$(\rho, u,\nabla \pi)$  fulfilling the following estimates that
\begin{equation}\label{est.Max.w3-0}
  \|\rho-1\|_{L^\infty(\RR_+;L^2\cap L^\infty(\RR^2))}\leq \|\rho_0-1\|_{L^2\cap L^\infty(\RR^2)},\quad \textrm{and}
\end{equation}
\begin{equation}\label{est.Max.w3}
  \|u\|_{L^\infty(\RR_+;L^2\cap\dot B^\alpha _{p,2})} + \|\big(\partial_t u , \Lambda^{2\alpha} u ,\nabla \pi\big)\|_{L^2(\RR_+;L^p)} + \|u\|_{L^2(\RR_+;\dot H^\alpha)}
  \leq C \big(1 + \|u_0\|_{H^1\cap \dot B^\alpha_{p,2}}^{\frac{4\alpha-1}{2\alpha-1}}\big),
\end{equation}
with $C>0$ a constant depending only on $\alpha,p$.

If we assume, in addition,  $u_0\in \dot B^{\alpha+s}_{p,2}(\RR^2)$ with $s\in (0,1)$ and
$\rho_0-1\in{ {\mathcal{M}(\dot B^s_{p,2})}\cap {\mathcal{M}(\dot B^{\frac{2}{p}+1-2\alpha}_{p,1})} }$ satisfying that for a sufficiently small generic constant $c_*>0$ (depending only on $\alpha,p,s$),
\begin{equation}\label{eq:rho-cd}
  \|\rho_0-1\|_{ {\mathcal{M}(\dot B^s_{p,2})}\cap {\mathcal{M}(\dot B^{\frac 2p+1-2\alpha}_{p,1})} } \leq c_* \exp \Big\{-c_*^{-1} \big(1+ \|u_0\|_{H^1\cap \dot B^\alpha_{p,2}}^{\frac{8\alpha-2}{2\alpha-1}} \big) \Big\},
\end{equation}
then the above constructed solution is unique, and $u$ also satisfies that
\begin{equation}\label{est.Max.w4}
\begin{split}
  \|u\|_{L^\infty(\RR_+;\dot B^{\alpha+s}_{p,2})}  + \| u\|_{L^2 (\RR_+;\dot B^{2\alpha}_{p,1}\cap \dot B^{2\alpha+s}_{p,2})} + \|u\|_{L^1(\RR_+; \dot W^{1,\infty})}
  \leq C \Big( 1+ \|u_0\|_{\dot B^{\alpha+s}_{p,2}} + \|u_0\|_{H^1\cap \dot B^\alpha_{p,2}}^{\frac{8(4\alpha-1)}{(2\alpha-1)^2}} \Big),
\end{split}
\end{equation}
with $C>0$ a constant depending only on $\alpha,p,s$.
\end{theorem}

In the above $\mathcal{M}(\dot{B}^s_{p,r})$ denotes the multiplier space defined in Definition \ref{def.mu}.

\begin{remark}\label{rem:uni-reg}
We apply the Lagrangian coordinates method in the uniqueness part of Theorem \ref{Main.thm}, 
and owing to the nonlocal effect of fractional Laplacian operator $\Lambda^{2\alpha}$,
it seems that we need a little bit more regularity of $u$ than the obtained regularity in the existence part.
More precisely, in view of \eqref{est.f24} below, one has to control $\|u\|_{L^2_T(\dot B^{2\alpha}_{p,1})}$ (from which $\|v\|_{L^2_T(\dot B^{2\alpha}_{p,1})}$ is bounded),
which is essentially stronger than the quantity $\|u\|_{L^2_T(\dot W^{2\alpha,p})}$ in \eqref{est.Max.w3}.
So we additionally assume $u_0$ is slightly more regular and
$\rho_0 -1$ is small enough in the norm of some multiplier spaces, and we build the refined estimate \eqref{est.Max.w4}.
It should be emphasized that, thanks to Lemma \ref{lem:mult-sp}, the used multiplier spaces for small $s$ contain the elements with piecewise jump discontinuity.
\end{remark}


Next we consider the density patch problem of the 2D fractional INS equations \eqref{eq.frc-INNS},
and as a direct consequence of Theorem \ref{Main.thm}, we can show the global well-posedness result and the global persistence of $C^{1,\gamma}$-patch boundary,
as long as the density jump across a $C^{1,\gamma}$-interface is small enough.
\begin{proposition}\label{cor:patch}
Let $\frac{1}{2}<\alpha <1$, $p>\frac{2}{2\alpha-1}$, $u_0 \in H^1\cap \dot B^{\alpha +s}_{p,2}(\RR^2)$ with $0< s < \frac{1}{p}$.
Assume that $\Omega_0$ is a bounded simply-connected $C^{1,\gamma}$-domain of $\RR^2$ with $0<\gamma \leq 2\alpha -1 +s - \frac{2}{p}$ and $\rho_0 = (1+\sigma)1_{\Omega_0}+1_{\Omega_0^c}$ for some small constant $\sigma\in \RR$ satisfying
\begin{equation}\label{eq:eta-cd}
  |\sigma| \leq c'\, \exp\Big(-\frac{1}{c'}  \big( 1+ \|u_0\|_{H^1\cap \dot B^\alpha_{p,2}}^{\frac{8\alpha-2}{2\alpha-1}} \big) \Big),
\end{equation}
with a generic small constant $c'>0$ depending only on $\alpha,p,s, \Omega_0$.
Then the 2D fractional INS system \eqref{eq.frc-INNS} has a unique global solution $(\rho, u)$ on $\RR^2\times\RR_+$
fulfilling the estimates \eqref{est.Max.w3}, \eqref{est.Max.w4}, and the density $\rho$ has the following expression
\begin{align}\label{rho-exp}
  \rho(t) = (1+\sigma)1_{\Omega(t)}+1_{\Omega(t)^c}\quad \textrm{with}\quad {\Omega(t)}= X_t(\Omega_0),
\end{align}
satisfying that the associated patch boundary $\partial \Omega(t)\in C^{1,\gamma}(\RR^2)$ for every $t\in \RR_+$.
\end{proposition}

\begin{proof}[Proof of Proposition \ref{cor:patch}]
Thanks to Lemma \ref{lem:mult-sp} and smallness condition \eqref{eq:eta-cd}, $\rho_0-1 = \sigma 1_{\Omega_0}$ belongs to $L^2\cap L^\infty \cap \mathcal{M}(\dot B^s_{p,2})\cap {\mathcal{M}(\dot B^{\frac 2p+1-2\alpha}_{p,1})}$
and it also fulfills \eqref{con.Max.B0}, \eqref{eq:rho-cd},
then Theorem \ref{Main.thm} guarantees that there is a unique global-in-time regular solution $(\rho,u)$ to the 2D fractional INS system \eqref{eq.frc-INNS}.
The estimates \eqref{est.Max.w3}, \eqref{est.Max.w4} and the continuous embedding  $\dot B^{2\alpha+s}_{p,2}\cap \dot H^\alpha(\RR^2) \hookrightarrow \dot C^{1,\gamma}(\RR^2)$ with $0<\gamma \leq 2\alpha-1 +s -2/p$
imply that $u\in L^1(\RR_+; \dot W^{1,\infty}(\RR^2))\cap  L^1([0,T]; C^{1,\gamma}(\RR^2))$ for any $T>0$.
By the Cauchy-Lipschitz theory, there exists a unique particle-trajectory $X_t(\cdot):\RR^2\rightarrow \RR^2$ for every $t\in \RR_+$
which solves \eqref{eq:flow0} and it is a measure-preserving bi-Lipschitzian homeomorphism with its inverse $X_t^{-1}$.
Besides, owing to \cite[Prop. 3.10]{BCD11} it is direct to see that
\begin{align*}
  \|\nabla X_t^{\pm1}\|_{\dot C^\gamma} \leq \int_0^t \|\nabla X_\tau^{\pm1}\|_{\dot C^\gamma} \|\nabla u(\tau)\|_{L^\infty} \dd \tau
  + \int_0^t \|\nabla X_\tau^{\pm1}\|_{L^\infty}^{1+\gamma} \|\nabla u(\tau)\|_{\dot C^\gamma} \dd \tau,
\end{align*}
and thus 
\begin{align}\label{es.Xt-Cgam}
  \|\nabla X_t^{\pm1}\|_{L^\infty_T(\dot C^\gamma)} \leq \|\nabla X_t^{\pm1}\|_{L^\infty_T(L^\infty)}^{1+\gamma} \|\nabla u\|_{L^1_T(\dot C^\gamma)} e^{\|\nabla u\|_{L^1_T(L^\infty)}} < \infty.
\end{align}
The method of characteristics gives that $\rho(x,t) = \rho_0(X_t^{-1}(x))$, which leads to the relation \eqref{rho-exp}.
Since the initial boundary $\partial \Omega_0\in C^{1,\gamma}$ and $X_t^{\pm1}\in L^\infty_T(C^{1,\gamma})$, we conclude that the evolutionary patch boundary
$\partial \Omega(t)\in L^\infty_T(C^{1,\gamma})$ with $0<\gamma\leq 2\alpha-1+ s -2/p$ and $T>0$ any given, as desired.

\end{proof}

Let us sketch the proof of Theorem \ref{Main.thm}.
By applying the technique of vector-valued Calder\'on-Zygmund operators in Lemari\'e-Rieusset \cite[Chapter 7]{LRie02},
we first establish the $L^q_t(L^p)$ maximal regularity estimates for the generalized Stokes system with fractional dissipation in the whole space $\RR^d$,
which may have its independent interests.
Next, by performing the $L^2$-energy estimate and the $L^2_t(L^p)$ maximal regularity estimate for the equations \eqref{eq.2dNS} and the perturbed system \eqref{eq.INS-2dNS},
we build the needing \textit{a priori} estimates for the 2D fractional INS equations \eqref{eq.frc-INNS}, and then by constructing an approximation process and using the compactness argument we show the existence proof of Theorem \ref{Main.thm}.

As for the uniqueness part, due to the hyperbolicity of density equation $\eqref{eq.frc-INNS}_1$ and the low-regularity assumption of the density, we have to adopt the Lagrangian coordinates approach originated in \cite{DanchMucha12,DanM13}.
Due to the nonlocal effect of fractional Laplacian operator $\Lambda^{2\alpha}$, the process is more complicated than that in the 2D INS equations \eqref{eq.INNS}.
We rewrite the system of the difference of two velocities in Lagrangian coordinates as the twisted fractional Stokes system \eqref{eq.delta.v},
and by making full use of the particle-trajectory technique and the finite-difference characterization of homogeneous Besov spaces,
we establish the crucial $L^2_t(L^2_x)$-maximal regularity estimate \eqref{est.sta} for the system \eqref{eq.delta.v} on a short time interval
(one can see Remark \ref{rem:uni} below for some additional explanation).
Then through carefully estimating the right-hand terms of \eqref{est.sta} and letting time $t$ being small enough,
we can show the uniqueness, and then the iteration argument implies the uniqueness on the whole $\RR_+$.
Note that the uniqueness part needs the stronger regularity of solutions (as explained in Remark \ref{rem:uni-reg}),
which is obtained in Proposition \ref{prop:uni.M.bet} by using the technique of multiplier spaces.
Finally, we remark that to the best of our knowledge, it seems to be the first result on applying the Lagrangian coordinates method to tackle with the uniqueness issue for a fractionally dissipative system.

The rest of this paper is organized as follows. In Section \ref{sec:pre}, we introduce the definitions of some functional spaces and their related estimates, and also compile several useful auxiliary lemmas.
Section \ref{sec:LpLq}  focuses on the establishment of Proposition \ref{pro.Max-Regularity} concerning the $L^q_t(L^p)$-maximal regularity estimate for the generalized Stokes system.
Sections \ref{APR} - \ref{sec:UNI} are devoted to the proof of Theorem \ref{Main.thm}, corresponding to the proof of the \textit{a priori} estimates, the existence, and the uniqueness, respectively.

\section{Preliminaries}\label{sec:pre}

In this section we include some notations, definitions and auxiliary lemmas used in the paper.

Throughout the paper, for every $p, q \in[1,\infty]$, $k\in\NN$, $s\in \RR$, the function spaces $L^p(\RR^d)$, $W^{k,p}(\RR^d)$, $\dot W^{k,p}(\RR^d)$, $\dot W^{s,p}(\RR^d)$
denote the Lebesgue space, the Sobolev space, the homogeneous Sobolev space and the fractional-order Sobolev space, respectively (e.g. see \cite{BCD11}).
Denote by $\mathcal{S}(\RR^d)$ the space of Schwartz functions, and $\mathcal{S}'(\RR^d)$ its dual, the space of tempered distributions (e.g. see \cite{Grafakos14}).
We abbreviate $L^q(0,T;X)$ as $L^q_T(X)$, with $X=X(\RR^d)$ a spacial function space.
We abbreviate $\|(f_1,\cdots,f_k)\|_X$ as $\|f_1\|_X +\cdots \|f_k\|_{X}$ for $k\in \ZZ_+$.
For two matrixes $A=(a_{ij})_{d\times d}$ and $B=(b_{ij})_{d\times d}$, denote by $A:B$ as the quantity $\sum_{1\leq i,j\leq d} a_{ij} b_{ji}$.
The notation $a\lesssim b$ means $a\leq Cb$, where the constant $C$  may be different from line to line.

\subsection{Functional spaces and related estimates}
We first recall some basic knowledge of the Littlewood-Paley theory.
One can choose two nonnegative radial functions $\chi, \varphi\in \mathcal{S}(\mathbb{R}^d)$ (see \cite{BCD11}) be
supported respectively in the ball $\{\xi\in \mathbb{R}^d:|\xi|\leq \frac{4}{3} \}$ and the annulus $\{\xi\in
\mathbb{R}^d: \frac{3}{4}\leq |\xi|\leq  \frac{8}{3} \}$ such that
\begin{align*}
  \chi(\xi)+\sum\limits_{j\ge0}\varphi(2^{-j}\xi)=1,\quad \textrm{for}\;\;\xi\in\RR^d; \quad\textrm{and}\quad
  \sum\limits_{j\in\ZZ}\varphi(2^{-j}\xi)=1,\quad \textrm{for}\;\; \xi\in\RR^d \setminus \{0\}.
\end{align*}
The homogeneous dyadic operators $\dot{\Delta}_j$ and the homogeneous low-frequency cut-off operators $ \dot{S}_j$ are defined for all $j\in\ZZ$ by
$$
\dot{\Delta}_ju=\varphi(2^{-j}D)u=2^{jd} h(2^j\cdot)*u,\quad  \dot{S}_j u= \chi(2^{-j}D)u =2^{jd} \widetilde h(2^j\cdot)*u,
$$
with $h=\mathcal{F}^{-1}{\varphi}$, $\widetilde {h}=\mathcal{F}^{-1}{\chi}$ and $\mathcal{F}^{-1}$ the Fourier inverse transform.

Then we introduce the definitions of the homogeneous Besov space and the related Chemin-Lerner's mixed spacetime space.
\begin{definition}\label{Besovdef}
\noindent (1) Let $s\in \RR$, $1\le p,r \le \infty$. Let $\mathcal S'(\RR^d)$ be the space of tempered distributions and $P(\RR^d)$ is the set of all polynomials.
The homogeneous Besov space $\dot{B}^s_{p,r}=\dot{B}^s_{p,r}(\RR^d)$ is defined as the following quotient space
\begin{align*}
  \dot{B}^{s}_{p,r}(\RR^d)=\Big\{u\in \mathcal{S}'(\RR^d)/P(\RR^d):\|u\|_{\dot{B}^{s}_{p,r}(\RR^d)}:=
  \big\| \big\{2^{js}\|\dot{\Delta}_ju\|^{r}_{L^{p}(\RR^d)}\big\}_{j\in\mathbb{Z}}\big\|_{\ell^r} <\infty\Big\}.
\end{align*}

\noindent (2)
Let $s\in \RR$, $1\le p,q,r \le \infty$, and $T\in(0,\infty]$. The Chemin-Lerner's mixed spacetime homogeneous Besov space $\widetilde{L}^q(0,T;{\dot B}^s_{p,r}(\RR^d))$,
abbreviated as $\widetilde{L}^q_T(\dot B^s_{p,r})$, is defined as the set of all tempered distributions $u$ such that
\begin{equation*}
  \|u\|_{\widetilde {L}^q_T({\dot B}^{s}_{p,r})} :=  \big\| \big\{2^{js}\|\dot\Delta_ju\|_{L^q_T(L^p)} \big\}_{j\in\mathbb{Z}} \big\|_{\ell^r}<\infty.
\end{equation*}
\end{definition}

The following product estimate in homogeneous Besov space is useful.
\begin{lemma}[\cite{BCD11}, Corollary 2.54]\label{Besovpro}
Let $s>0$, $(p,r)\in [1,\infty]^2$. Then there exists a constant $C=C(s,d)>0$ such that
\begin{align}\label{est.prod2}
  \|uv\|_{{\dot B}_{p,r}^{s}(\RR^d)}\leq C \big(\|u\|_{{\dot B}^{s}_{p,r}(\RR^d)}\|v\|_{L^\infty(\RR^d)}+\|v\|_{{\dot B}^s_{p,r}(\RR^d)}\|u\|_{L^\infty(\RR^d)}\big).
\end{align}
\end{lemma}

We have the characterization of homogeneous Besov space in terms of the fractional heat semigroup $e^{-t\Lambda^{2\alpha}}$.
\begin{lemma}[\cite{MiaoYZh08}, Proposition 2.1]\label{frc-Gausskehua}
Let $s>0$, $(p,r)\in [1,\infty]^2$.  Then, for any $\varphi\in \dot B^{-s}_{p,r}(\RR^d)$,  there exists a constant $C=C(s,p,r,d)\geq 1$ such that
\begin{align*}
  C^{-1}\|\varphi\|_{\dot B^{-s}_{p,r}(\RR^d)}\leq \big\|t^{s\over {2\alpha}}\|e^{-t \Lambda^{2\alpha}} \varphi\|_{L^p(\RR^d)}\big\|_{L^r(\RR _+;{{\mathrm dt}\over t})}\leq C \|\varphi\|_{\dot B^{-s}_{p,r}(\RR^d)}.
\end{align*}
\end{lemma}

We also use the following finite-difference characterization of homogeneous Besov space.
\begin{lemma}[\cite{BCD11}, Theorem 2.36]\label{lem:Besov-fd}
  Let $s\in (0,1)$ and $(p,r)\in [1,\infty]^2$. Then there exists a constant $C = C(s,p,r,d)\geq 1$ such that
\begin{align}\label{eq.Besv.fd}
  C^{-1} \|f\|_{\dot B^s_{p,r}(\RR^d)} \leq \Big\|\frac{\|u(y+\cdot)- u(\cdot)\|_{L^p}}{|y|^s} \Big\|_{L^r(\RR^d; \frac{\dd y}{|y|^d})}
  \leq C \|f\|_{\dot B^s_{p,r}(\RR^d)}.
\end{align}
\end{lemma}

The maximal regularity estimate in the framework of Besov space for the nonhomogeneous fractional heat equation is useful in the sequel.
\begin{lemma}[\cite{WuYuan08}, Theorem 3.2]\label{lem.es.Be}
Let $s\in\RR$, $(p,r)\in [1,\infty]^2$, $\alpha\in (0,1)$ and $1\leq \rho_1\leq\rho\leq\infty$.
Assume $f_0\in {\dot B}^s_{p,r}(\RR^d)$, $g\in {{\widetilde L}^{\rho_1}_T({\dot B}^{s-2\alpha+\frac {2\alpha}{\rho_1}}_{p,r}(\RR^d))}$ and $f$ solves the  fractional heat equation
\begin{align*}
  \partial_t f+\Lambda^{2\alpha} f=g,\quad f|_{t=0}=f_0.
\end{align*}
Then there exists a constant $C= C(d,\alpha)>0$ such that for every $T\in (0, \infty]$,
\begin{align}
  \|f\|_{{\widetilde L}^{\rho}_T({\dot B}^{s+\frac {2\alpha}\rho}_{p,r})} \leq C\Big(\|f_0\|_{{\dot B}^{s}_{p,r}}+\|g\|_{{\widetilde L}^{\rho_1}_T({\dot B}^{s-2\alpha+\frac {2\alpha}{\rho_1}}_{p,r})}\Big).
\end{align}
\end{lemma}

The (pointwise) multiplier space of homogeneous Besov space is defined as follows.
\begin{definition}\label{def.mu}
Let $1\leq p, r\leq \infty,\,\sigma\in \RR$. The multiplier space $\mathcal{M}(\dot B^\sigma_{p,r}(\RR^d))$ of $\dot B^\sigma_{p,r}(\RR^d)$, abbreviated as $\mathcal{M}(\dot B^\sigma_{p,r})$,
is the set of tempered distribution $f$ such that $f\phi\in \dot B^\sigma_{p,r}(\RR^d)$ for any $\phi\in \dot B^\sigma_{p,r}(\RR^d)$,
with the norm endowed by
\begin{align*}
  \|f\|_{\mathcal{M}(\dot B^\sigma_{p,r})} := \sup_{\|\phi\|_{ \dot B^\sigma_{p,r}(\RR^d)}\leq 1}\|f\phi\|_{\dot B^\sigma_{p,r}(\RR^d)}.
\end{align*}
\end{definition}

The following lemma states that the multiplier space $\mathcal{M}(\dot{B}^\sigma_{p,r})$ can involve the element having the patch structure.
\begin{lemma}[\cite{DanchMucha12}, Lemma A.7]\label{lem:mult-sp}
  Let $\Omega$ be the half-space $\RR^d_+$ or a bounded domain of $\RR^d$ with $C^1$-boundary.
Assume that $s\in \RR$ and $p,r\in [1,\infty]$ are such that $-1+\frac{1}{p}<s < \frac{1}{p}$.
Then the characteristic function $1_\Omega(x)$ of $\Omega$ belongs to the multiplier space $\mathcal{M}(\dot B^s_{p,r}(\RR^d))$.
\end{lemma}

We have the following regularity propagation estimates for a composite function involving the particle-trajectory map.
\begin{lemma}\label{lem:f.mu}
Let $T\in (0,\infty]$, $(q, r)\in [1,\infty]^2$, and $u\in L^1(0,T;Lip(\RR^d))$ be a divergence-free vector field.
Let $ X_t:\RR^d \rightarrow \RR^d $ be the particle-trajectory map 
defined by \eqref{eq:flow0}
with its inverse $X_t^{-1}$.
\begin{enumerate}[(1)]
\item If $f\in {L^q(0,T;\dot B^\sigma_{p,r}(\RR^d))}$, $\sigma\in (-1,1)$, then $f\circ X_t^{\pm1}\in L^q(0,T;{\dot B^{\sigma}_{p,r}}(\RR^d))$ with
\begin{align}\label{es.f.mu}
  \|f\circ X_t^{\pm1}\|_{L^q_T(\dot B^{\sigma}_{p,r})}\leq C\|f\|_{L^q_T(\dot B^{\sigma}_{p,r})}e^{C\int_0^T\|\nabla u\|_{L^\infty}\dd t}.
\end{align}
\item Assume $ a_0\in \mathcal{M}(\dot B^{\sigma}_{p,r}(\RR^d))$, $\sigma\in (-1,1)$,
and $a(t,x)$ is a smooth solution to the free transport equation $\partial_t a + u\cdot\nabla a=0$ associated with $ a|_{t=0}=a_0$.
Then $a \in L^\infty(0,T; \mathcal{M}(\dot B^{\sigma}_{p,r}))$ with
\begin{align}\label{a.m}
  \|a \|_{L^\infty_T(\mathcal{M}(\dot B^{\sigma}_{p,r}))}\leq C\|a_0\|_{\mathcal{M}(\dot B^{\sigma}_{p,r})}e^{C \int_0^T\|\nabla u\|_{L^\infty}\dd t}.
\end{align}
\end{enumerate}
\end{lemma}

\begin{proof}[Proof of Lemma \ref{lem:f.mu}]
If $q=\infty$, $r=1$ and $p\in(2,4)$, both inequalities \eqref{es.f.mu}-\eqref{a.m}
have appeared in \cite[Section 5]{HuangPZh13}, and we here sketch the proof of the slightly general cases.

(1) Taking advantage of \cite[Lemma 2.7]{BCD11} yields
\begin{align*}
  \|\dot\Delta_j( (\dot\Delta_k f)\circ X_t^{\pm1})\|_{L^p}\leq C c_k 2^{-k\sigma}\|f\|_{\dot B^{\sigma}_{p,r}}\min\{2^{j-k}, 2^{k-j}\}e^{C\int_0^t\|\nabla u\|_{L^\infty}\dd \tau},
\end{align*}
where $\{c_k\}_{k\in\ZZ}$ satisfies $\|c_k\|_{\ell^r(\ZZ)}=1$.
Together with the condition $\sigma\in(-1,1)$ we deduce that
\begin{align*}
  \|\dot\Delta_j( f\circ X_t^{\pm1})\|_{L^p}\leq& \Big(\sum_{k<j}+\sum_{k\ge j}\Big)\|\dot\Delta_j((\dot\Delta_k  f)\circ X_t^{\pm1})\|_{L^p}  \\
  \leq & C \Big(\sum_{k<j}2^{k-j}c_k2^{-k\sigma}+\sum_{k\ge j}2^{j-k}c_k2^{-k\sigma}\Big)\|f\|_{\dot B^{\sigma}_{p,r}}e^{C\int_0^T\|\nabla u\|_{L^\infty}\dd t}\\
  \leq & C c_j 2^{-j\sigma}\|f\|_{\dot B^{\sigma}_{p,r}}e^{C\int_0^T\|\nabla u\|_{L^\infty}\dd t}.
\end{align*}
Taking $\ell^r$-norm over $j\in\ZZ$ and then taking $L^q$-norm on $[0,T]$ lead to inequality \eqref{es.f.mu} as desired.

(2) Note that $a(t,x) = a_0\circ X_t^{-1}(x)$. By virtue of the definition of multiplier space $\mathcal{M}(\dot B^\sigma_{p,r})$, the measure-preserving property of $X_t^{\pm1}$ and \eqref{es.f.mu},
the inequality \eqref{a.m} can be easily deduced (e.g. see \cite[Proposition 5.1]{HuangPZh13}).
\end{proof}

\subsection{Auxiliary lemmas}


We list some useful results related to the fractional Laplacian $\Lambda^\alpha$.

\begin{lemma}[\cite{Grafakos14}, Theorem 7.6.1]\label{lem:frc-leib}
Let $1<r<\infty$ and $1<p_1,\,p_2,\,q_1,\,q_2\leq \infty$ satisfy $\frac 1r=\frac 1{p_1}+\frac 1{p_2}=\frac 1{q_1}+\frac 1{q_2}.$
Given $s> 0$, then there exists a constant $C=C(d,s,r,p_1,p_2,q_1,q_2)>0$ such that for every $f,g\in \mathcal{S}(\RR^d)$,
\begin{align}\label{eq:frc-leib}
  \|\Lambda^s(fg)\|_{L^r(\RR^d)}\leq C\big(\|\Lambda^s f\|_{L^{p_1}(\RR^d)}\|g\|_{L^{p_2}(\RR^d)}+\|f\|_{L^{q_1}(\RR^d)}\|\Lambda^s g\|_{L^{q_2}(\RR^d)} \big).
\end{align}
\end{lemma}

\begin{lemma}[\cite{MiaoYZh08}, Lemmas 2.1, 2.2]\label{lem:kernel}
  Let $K(x)$ be the kernel function of the fractional heat semigroup $e^{- \Lambda^{2\alpha}}$, $\alpha\in (0,1)$, that is,
\begin{align}\label{kernel:K}
  K(x)=\mathcal{F}^{-1} (e^{-|\xi|^{2\alpha}}) = \frac{1}{(2\pi)^d}\int_{\RR^d} e^{ix\cdot\xi} e^{- |\xi|^{2\alpha}} \dd \xi .
\end{align}
Then there exists a positive constant $C= C(d,\alpha)$ such that
\begin{align}\label{kerK-es1}
  |K(x)| \leq C (1 + |x|)^{-d -2\alpha},\quad \forall x\in \RR^d;
\end{align}
and for every $\beta>0$, there exists a positive constant $C=C(d,\alpha,\beta)$ such that
\begin{align}\label{kerK-es2}
  |\Lambda^\beta K(x)| + |\Lambda^{\beta-1} \nabla K(x)| \leq C (1+ |x|)^{- d  - \beta},\quad \forall x\in \RR^d.
\end{align}
\end{lemma}

\begin{lemma}\label{lem:fra-op-exp}
  Let $\alpha\in (0,1)$, then for every $f\in \mathcal{S}(\RR^d)$ and $i\in\{1,\cdots,d\}$, we have
\begin{equation}\label{fra-op-exp}
  \partial_{x_i} \Lambda^{2\alpha-2} f (x) = c_\alpha\, \mathrm{p.v.} \int_{\RR^d} \frac{x_i -y_i}{|x-y|^{ d + 2\alpha}} \big(f(x) - f(y)\big) \dd y,
\end{equation}
with $c_\alpha = \frac{(d+2\alpha-2)\Gamma(\frac{d}{2}-1+\alpha)}{\pi^{d/2} 2^{2-2\alpha} \Gamma(1-\alpha)}$.
\end{lemma}

\begin{proof}[Proof of Lemma \ref{lem:fra-op-exp}]
  Recalling that (e.g. see \cite[Section 5.1]{Stein70})
\begin{align*}
  \Lambda^{2\alpha-2} f(x) =  \bar{c}_\alpha \int_{\RR^d} \frac{1}{|y|^{d+2\alpha-2}} f(x-y) \dd y ,
\end{align*}
with $\bar{c}_\alpha = \frac{\Gamma(\frac{d}{2}-1+\alpha)}{\pi^{d/2} 2^{2-2\alpha} \Gamma(1-\alpha)}$, we get from the integration by parts that
\begin{align*}
  \partial_{x_i} \Lambda^{2\alpha-2}f(x) & = \bar{c}_\alpha \lim_{\epsilon\rightarrow 0}
  \int_{|y|\geq \epsilon} \frac{1}{|y|^{d+2\alpha-2}}  \partial_{x_i}f(x-y) \dd y \\
  & =  \bar{c}_\alpha \lim_{\epsilon\rightarrow 0} \int_{|y|\geq \epsilon} \frac{1}{|y|^{d+2\alpha-2}} \partial_{y_i} \big( f(x) - f(x-y) \big) \dd y \\
  & = c_\alpha \lim_{\epsilon\rightarrow 0} \int_{|y-x|\geq \epsilon} \frac{x_i-y_i}{|x-y|^{d+2\alpha}} \big(f(x) -f(y)\big) \dd y + \lim_{\epsilon\rightarrow 0} R_\epsilon,
\end{align*}
where
$R_\epsilon := \bar{c}_\alpha  \int_{|y|=\epsilon} \frac{1}{|y|^{d+2\alpha-2}} \big(-\frac{y_i}{|y|}\big) \big(f(x)-f(x-y)\big) \dd y$
satisfies that
\begin{align*}
  \lim_{\epsilon\rightarrow 0} |R_\epsilon| \leq \|\nabla f\|_{L^\infty} \lim_{\epsilon\rightarrow 0} \frac{1}{\epsilon^{d+2\alpha-2}} \int_{|y|=\epsilon} |y| \dd y
  \leq \|\nabla f\|_{L^\infty} \lim_{\epsilon\rightarrow 0} \epsilon^{2-2\alpha} =0.
\end{align*}
\end{proof}

In the proof of Proposition \ref{pro.Max-Regularity} we use the following boundedness result about Calder\'on-Zygmund operators for vector-valued singular integrals,
which can be found in \cite[Chapter 7]{LRie02}.
\begin{lemma}[Calder\'on-Zygmund operators]\label{CZ}
Let $1<p, p_1, p_2<\infty$. Let $X,\,X_1,\,X_2$ be three locally compact $\sigma$-compact metric spaces, with regular Borel measures $\mu,\,\mu_1,\,\mu_2$ on those spaces, respectively.
Define $E=L^{p_1}(X_1,\mu_1)$ and $F=L^{p_2}(X_2,\mu_2)$. Let $L(x,y;x_1,x_2)$ be a continuous function defined on $(X\times X- \Delta) \times X_1\times X_2$ (with $\Delta$ the diagonal set of $X\times X$),
then we define $L(x,y)$ as the operator from $E=L^{p_1}(X_1,\mu_1)$ to $ F=L^{p_2}(X_2,\mu_2)$ given by the integral
$$ L(x,y)f(x_2)=\int_{X_1} \mathcal{L}(x,y;x_1,x_2)f(x_1)\mathrm d\mu_1(x_1).$$
Whenever $x\notin \mathrm{supp}\,f$, we define $\mathcal{T}f(x)$ as
\begin{align*}
  \mathcal{T}f(x,x_2) & = \int_X \int_{X_1} \mathcal{L}(x,y;x_1,x_2)f(y,x_1)\mathrm d \mu(y) \mathrm d \mu_1(x_1) \\
  & = \int_X L(x,y) f(y,x_2) \dd \mu(y).
\end{align*}

Suppose that the space $X$ is equipped with a quasi-distance $d$ satisfying the quasi-triangular inequality $d(x,y) \leq C_d \big(d(x,z)+ d(z,y) \big)$ with $C_d>0$ a generic constant;
and there exists  positive numbers $n$ and $C_\mu$ such that $\mu(B(x,r))\leq C_\mu r^n$ for every $x\in X$ and $r>0$.
Assume that $\mathcal{T}$ is  bounded from $L^p(X,\mu;E)$ to  $L^p(X,\mu;F)$ such that
$$\int_{\RR^k}\|\mathcal{T}f(x)\|^p_F\mathrm dx\leq C \int_{\RR^k}\|f(x)\|^p_E\mathrm dx.$$
Assume that $\mathcal{L}(x,y)$ is continuous from $X\times X -\Delta$ to $\mathcal{L}(E,F)$ and satisfies that for some $\epsilon>0$:
\begin{align*}
  \|L(x,y)\|_{op(E\mapsto F)} & \leq C \frac{1}{d(x,y)^n};\\
  d(z,y)\leq \frac12 d(x,y)\Longrightarrow \|L(x,y)-L(x,z)\|_{op(E\mapsto F)}& \leq C\frac{d(z,y)^\epsilon}{d(x,y)^{n+\epsilon}};\\
  d(x,z)\leq \frac12  d(x,y) \Longrightarrow \|L(x,y)-L(z,y)\|_{op(E\mapsto F)} & \leq C\frac{d(x,z)^\epsilon}{d(x,y)^{n+\epsilon}}.
\end{align*}
Then the operator $\mathcal{T}$ is bounded from $L^q(X;E)$ to  $L^q(X;F)$ for any $1<q<\infty$.
\end{lemma}

The following multiplier theorem can be found in Theorem 3, Section 4.3.2 of \cite{Stein70}.
\begin{lemma}\label{lem:multiplier}
Assume that $m:\RR^d  \setminus \{0\}\longrightarrow \mathbb{C} $ is of $C^k$-class with an integer $k\geq [\frac{d}2 ]+1$ ($[\frac{d}{2}]$ is the integer part of $\frac{d}{2}$),
and it satisfies that
$|\partial_\xi^\beta m(\xi)|\leq C|\xi|^{-|\beta|}$, for every $\xi\neq 0$ and $|\beta|\leq k$.
Then for any $f\in L^p(\RR^d)$ with  $1<p<\infty$, there is a constant $C$ such that
$$
\| m(D) f\|_{L^p(\RR^d)}\leq C\|f\|_{L^p(\RR^d)}.
$$
\end{lemma}

\section{The $L^q_t(L^p)$-maximal regularity estimate for the generalized Stokes system}\label{sec:LpLq}

We mainly focus on showing the following result in this section.

\begin{proposition}\label{pro.Max-Regularity}
Let $\alpha\in (0,1)$, $1< p,\,q< \infty,$ $u_0\in \dot B^{2\alpha(1-\frac 1q)}_{p,q}(\RR^d),\,f \in L^q(0,T;L^p(\RR^d))$, $R\in \dot W^{1,q}(0,T; L^p(\RR^d))$
with $\Div R\in L^q(0,T; \dot W^{2\alpha-1,p}(\RR^d))$.
Then the generalized Stokes system
\begin{equation}\label{eq.frc-Stokes}
\left\{\begin{array}{ll}
  \partial_t u+\Lambda^{2\alpha} u+\nabla \pi=f,\\
  \Div u = \Div R,\\
  u|_{t=0}(x) = u_0(x),
\end{array}\right.
\end{equation}
has a unique solution  $(u,\nabla \pi)$. Moreover, there exists a generic positive constant $C = C(d,\alpha,p,q)$ such that for any $T\in (0,\infty]$,
\begin{equation}\label{est.Max.LpLp}
\begin{split}
  &\|u\|_{L^\infty(0,T;\dot B^{2\alpha(1-1/q)}_{p,q}(\RR^d))}+\|(\Lambda^{2\alpha}u,\partial_t u, \nabla \pi)\|_{L^q(0,T;L^p(\RR^d))}\\
  \leq & C \Big(\|u_0\|_{ \dot B^{2\alpha(1-1/q)}_{p,q}(\RR^d)} + \|(f, \partial_t R)\|_{L^q(0,T;L^p(\RR^d))}
  + \|\Div R\|_{L^q(0,T; \dot W^{2\alpha-1,p}(\RR^d) )} \Big).
\end{split}
\end{equation}

\end{proposition}

\begin{remark}
When $\alpha=1$ one can refer to \cite{DanM13} for the $L^q_t(L^p)$-maximal regularity estimate for the usual Stokes system;
and one can see \cite{GigS91,Iwas89} concerning more general domains including exterior domains.
For $\alpha\in (0,1)$, the $L^p_t(L^p)$ maximal regularity estimate for the fractional Stokes system \eqref{eq.frc-Stokes} with $R=0$
was also investigated by Giga et al \cite{Giga85,GigS91} using the abstract semigroup argument and by Cao et al \cite{CaoChLai19} using the Fourier multiplier method.
\end{remark}

The proof of Proposition \ref{pro.Max-Regularity} mainly relies on the following result, whose proof is placed below in this section.
\begin{lemma}\label{lem.A2alpha}
Set
\begin{equation}\label{Aalpha}
\begin{split}
  \mathcal{A}_{2\alpha} f(x,t) : = \int_0^t  e^{-(t-s)\Lambda^{2\alpha}}\Lambda^{2\alpha}f(\cdot,s)\mathrm d s.
\end{split}
\end{equation}
 Then, for any $T\in(0,\infty]$ and $1< p,\,q< \infty$, the operator $\mathcal{A}_{2\alpha} $ is continuously bounded from $L^q(0,T;L^p(\RR^d))$ to $L^q(0,T;L^p(\RR^d))$, and there exists a constant $C>0$ such that
\begin{align*}
  \|\mathcal{A}_{2\alpha} f\|_{L^q(0,T;L^p(\RR^d))}\leq C\|f\|_{L^q(0,T;L^p(\RR^d))}.
\end{align*}

\end{lemma}

\begin{proof}[Proof of Proposition \ref{pro.Max-Regularity}]
The existence and uniqueness of solution to the system \eqref{eq.frc-Stokes} are standard,
and can be proved as those of the usual inhomogeneous Stokes system.
Next we are devoted to proving the regularity estimate \eqref{est.Max.LpLp}.

Taking the divergence operator to the equation $\eqref{eq.frc-Stokes}_1$ leads to
\begin{align*}
  \Delta \pi = - \partial_t \Div u - \Lambda^{2\alpha} \Div u + \Div f  = - \Div \partial_t R - \Div \Lambda^{2\alpha} R + \Div f,
\end{align*}
thus denoting by $\PP := \nabla \Delta^{-1} \Div $, we see that
\begin{align}\label{nab-pi-exp}
  \nabla \pi =  - \PP \partial_t R -\PP \Lambda^{2\alpha} R + \PP f.
\end{align}
Rewrite the equation $\eqref{eq.frc-Stokes}_1$ as
\begin{align}\label{eq.u.pq}
  \partial_t u + \Lambda^{2\alpha} u = \PP \partial_t R + \PP \Lambda^{2\alpha} R + (\Id -\PP) f = : \overline{f},\qquad u|_{t=0} = u_0,
\end{align}
and then Duhamel's formula yields
\begin{equation}\label{eq.integr}
\begin{split}
  u(x,t)=e^{-t\Lambda^{2\alpha}}u_0(x)+\int_0^t e^{-(t-\tau)\Lambda^{2\alpha}} \overline{f}(x,\tau)\mathrm d\tau,
\end{split}
\end{equation}
where semigroup operator $e^{-t \Lambda^{2\alpha}}$ is given by \eqref{Salpha} below.

Applying operator $\Lambda^{2\alpha}$ to the above formula
and recalling \eqref{Aalpha}, it gives that
\begin{equation}\label{eq.integr.Max}
\begin{split}
  \Lambda^{2\alpha}u(x,t) = e^{- t \Lambda^{2\alpha}} \Lambda^{2\alpha} u_0(x) + \mathcal{A}_{2\alpha} \overline{f}(x,t).
\end{split}
\end{equation}
Then by virtue of Lemmas \ref{frc-Gausskehua} and \ref{lem.A2alpha}, one finds
\begin{align}\label{es:maxLqLq.0}
  \|\Lambda^{2\alpha} u\|_{L^q(0,T; L^p(\RR^d))} & \leq \|e^{-t\Lambda^{2\alpha}} \Lambda^{2\alpha} u_0 \|_{L^q(\RR_+; L^p(\RR^d))}
  + \|\mathcal{A}_{2\alpha} \overline{f} \|_{L^q(0,T; L^p(\RR^d))} \nonumber \\
  & \leq C \|u_0\|_{\dot B^{2\alpha(1-1/q)}_{p,q}} + C \|\overline{f}\|_{L^q(0,T; L^p)} \nonumber \\
  & \leq C \|u_0\|_{\dot B^{2\alpha(1-1/q)}_{p,q}} + C \|(\partial_t R, f)\|_{L^q(0,T; L^p)} +  C \|\Div R\|_{L^q(0,T; \dot W^{2\alpha-1,p})} ,
\end{align}
where in the last line the $L^p$ $(1<p<\infty)$ boundedness property of the singular integral operators is also used.
Thanks to \eqref{nab-pi-exp}, and using the Calder\'on-Zygmund theorem again, we get
\begin{equation}\label{eq.P.Max.1}
\begin{split}
  \|\nabla\pi\|_{L^q(0,T; L^p(\RR^d))} & \leq C \|\big(\PP \partial_t R, \PP \Lambda^{2\alpha} R, \PP f \big)\|_{L^q(0,T; L^p)} \\
  & \leq  C \|(\partial_t R, f)\|_{L^q(0,T; L^p)} +  C \|\Div R\|_{L^q(0,T; \dot W^{2\alpha-1,p})}.
\end{split}
\end{equation}
We use the equation $\eqref{eq.frc-Stokes}_1$ and gather the above estimates to infer that
\begin{equation}\label{est:par-uLqLp}
\begin{split}
  \|\partial_t u\|_{L^q(0,T; L^p(\RR^d))} & \leq C \|(\Lambda^{2\alpha} u, \nabla \pi , f ) \|_{L^q(0,T; L^p(\RR^d))} \\
  & \leq  C \|(\partial_t R, f)\|_{L^q(0,T; L^p)} +  C \|\Div R\|_{L^q(0,T; \dot W^{2\alpha-1,p})} .
\end{split}
\end{equation}
Noticing that
$\|e^{-t\Lambda^{2\alpha}} u_0\|_{\dot B^{2\alpha(1-1/q)}_{p,q}(\RR^d)} \leq C \|u_0\|_{\dot B^{2\alpha(1-1/q)}_{p,q}(\RR^d)}$ (following from \eqref{kerK-es1}) and
\begin{align*}
  \|\mathcal{A}_{2\alpha} \overline{f}(\cdot, t)\|_{\dot B^{-2\alpha/q}_{p,q}(\RR^d)}
  & \leq C \big\| \|e^{- \tau' \Lambda^{2\alpha}} \mathcal{A}_{2\alpha}\overline{f}\|_{L^p(\RR^d)} \big\|_{L^q_{\tau'}(\RR_+)} \\
  & \leq C  \Big\|\int_0^t e^{-(t+\tau'-\tau) \Lambda^{2\alpha}} \Lambda^{2\alpha}\overline{f}(x,\tau) \dd \tau \Big\|_{L^q_{\tau'} (\RR_+;L^p)}  \\
  & \leq C \Big\|\int_0^{t+\tau'} e^{-(t+\tau'-\tau) \Lambda^{2\alpha}} \Lambda^{2\alpha}\overline{f}(x,\tau) 1_{[0,t]}(\tau) \dd \tau
  \Big\|_{L^q_{\tau'}(\RR_+; L^p)} \\
  & \leq C \|\overline{f}(x,\tau) 1_{[0,t]}(\tau)\|_{L^q_\tau(\RR_+; L^p)}  \leq C \|\overline{f}\|_{L^q(0,t; L^p)},
\end{align*}
it yields from \eqref{eq.integr.Max} that
\begin{align}\label{est.Max.LpLp.1}
  \|u\|_{L^\infty(0,T;\dot B^{2\alpha(1-1/q)}_{p,q})} \leq & C \|\Lambda^{2\alpha}u\|_{L^\infty(0,T; \dot B^{-2\alpha/q}_{p,q})}  \nonumber \\
  \leq & C \|e^{- t \Lambda^{2\alpha}} \Lambda^{2\alpha} u_0\|_{L^\infty(0,T; \dot B^{-2\alpha/q}_{p,q})}
  + C \|\mathcal{A}_{2\alpha} \overline{f}\|_{L^\infty(0,T; \dot B^{-2\alpha/q}_{p,q})} \nonumber \\
  \leq & C\|u_0\|_{ \dot B^{2\alpha(1- 1/q)}_{p,q}} + C \|\overline{f}\|_{L^q(0,T;L^p)} \nonumber \\
  \leq & C\|u_0\|_{ \dot B^{2\alpha(1-1/q)}_{p,q}} + C \|(\partial_t R, f)\|_{L^q(0,T; L^p)} +  C \|\Div R\|_{L^q(0,T; \dot W^{2\alpha-1,p})}.
\end{align}

Therefore, collecting estimates \eqref{es:maxLqLq.0}--\eqref{est.Max.LpLp.1} concludes \eqref{est.Max.LpLp}, as desired.

\end{proof}

Then it remains to prove Lemma \ref{lem.A2alpha}.
\begin{proof}[Proof of Lemma \ref{lem.A2alpha}]
The idea is analogous to that of \cite[Theorem 7.3]{LRie02} given by P.G. Lemari$\mathrm{\acute{e}}$-Rieusset.
For every $\varphi\in L^p(\RR^d)$, set
\begin{equation}\label{Salpha}
\begin{split}
  e^{-t\Lambda^{2\alpha}}\varphi(x) = \mathcal{F}^{-1}(e^{-t|\xi|^{2\alpha}})(x)\ast \varphi(x) = K_t(x)*\varphi(x),
\end{split}
\end{equation}
where $K_t(x) = t^{-\frac{d}{2\alpha}}K(\frac x{t^{1/{2\alpha}}})$.
Noting that $\Lambda^{2\alpha} K_t(x) = t^{-\frac{d}{2\alpha}-1} (\Lambda^{2\alpha} K)(\frac{x}{t^{1/2\alpha}})$,
we let $\Omega(x,t) : = t^{-\frac{d}{2\alpha}}(\Lambda^{2\alpha}K)(\frac x{t^{1/{2\alpha}}})$, and then it follows that
$\frac{1}{t} \Omega(x,t) = \mathcal{F}^{-1}(|\xi|^{2\alpha}e^{-t|\xi|^{2\alpha}})$ and
\begin{equation}
\begin{split}
\mathcal{A}_{2\alpha} f(x,t) =\int_0^t\int_{\RR^d}\frac 1{t-s}\Omega(x-y,t-s)f(y,s) \,\mathrm d y\mathrm d s.
\end{split}
\end{equation}

Without loss of generality, we only need to prove the case of $T=\infty$ because the other case can be reduced to this case by zero extension in time $t$.
Moreover, we denote $\widetilde{f}(x,t)$, $\widetilde{\Omega}(x,t)$, $\mathcal{A}_{2\alpha}\widetilde{{f}}(x,t)$ by extending $f(x,t)$, $\Omega(x,t)$, $\mathcal{A}_{2\alpha}{f}(x,t)$ to negative values of $t$ through zero extension, respectively. This is harmless since $\mathcal{A}_{2\alpha} f(x,t)$ depends only from the values of  $f$ on $(0,t)\times \RR^d$.

Now we compute the Fourier transform of $\frac{1}{t} \widetilde{\Omega}(x,t)$ with respect to spacetime variable $(x,t)$:
the Fourier transform in $x$ for every $t>0$ leads to $\int_{\RR^d} e^{-ix\cdot\xi } \frac{1}{t} \widetilde{\Omega}(x,t) \dd x
= |\xi|^{2\alpha} e^{-t |\xi|^{2\alpha}}$, and then the Fourier transform in $t$ gives
\begin{align*}
  m(\xi,\tau) = \int_0^\infty e^{-i t\tau} |\xi|^{2\alpha} e^{-|\xi|^{2\alpha}t} \dd t = \frac {|\xi|^{2\alpha}}{i\tau+|\xi|^{2\alpha}}.
\end{align*}
Since $|m(\xi,\tau)| \leq 1$, we directly get $\mathcal{A}_{2\alpha}$ is bounded on $L^2(\RR\times \RR^d)$.

Next, we consider $\mathcal{A}_{2\alpha} $  as a Calder\'on-Zygmund operator on $\RR^d\times \RR$ endowed
with the Lebesgue measure $\mu$ on $\RR^{d+1}$ and with the quasi-distance $\rho((x,t),(y,s))=(|x-y|^{4\alpha}+|t-s|^2)^{\frac 1{4\alpha}}$.
We also have that for any $(x,t)\in\RR^d\times \RR$ and $r>0$,
\begin{align*}
  \mu(B((x,t),r)) = \int_{\rho((y,s),(x,t))\leq r} \dd y \dd s \leq C_0 r^{d+2\alpha}.
\end{align*}
Denote by
\begin{align*}
  L((x,t),(y,s))= \frac {1_{\{s<t\}}(s)}{t-s} \widetilde{\Omega}(x-y,t-s)
  = \frac {1_{\{s<t\}}(s)}{(t-s)^{\frac{d}{2\alpha}+1}} (\Lambda^{2\alpha}K)\Big(\frac{ x-y}{(t-s)^{1/{2\alpha}}}\Big).
\end{align*}
We claim that $L((x,t),(y,s))$ satisfies the assumptions in Lemma \ref{CZ}. Indeed, by making use of estimate \eqref{kerK-es2},
if $|x-y|^{2\alpha}\leq|t-s|,$ one has
\begin{align*}
  |L((x,t),(y,s))|\leq \frac{\|\Lambda^{2\alpha}K\|_{L^\infty}}{|t-s|^{\frac d{2\alpha}+1}}\leq \frac{C}{|t-s|^{\frac d{2\alpha}+1}};
\end{align*}
while if $|x-y|^{2\alpha}\ge|t-s|,$ one has
\begin{align*}
  |L((x,t),(y,s))|\leq \frac{\| |z|^{d+2\alpha}\Lambda^{2\alpha}K\|_{L^\infty}}{|x-y|^{d+{2\alpha}}}\leq \frac{C}{|x-y|^{d+{2\alpha}}};
\end{align*}
thus, we get
\begin{equation}\label{cndi.CZ.1}
\begin{split}
|L((x,t),(y,s))|\leq \frac{C}{{d((x,t),(y,s))}^{d+2\alpha}}.
\end{split}
\end{equation}
In a similar manner, we can also obtain
\begin{align*}
  \Big|\frac{\partial}{\partial t}L((x,t),(y,s))\Big|
  = \Big|\frac{\partial}{\partial s}L((x,t),(y,s))\Big| \leq \frac{C}{{\rho((x,t),(y,s))}^{d+4\alpha}},
\end{align*}
and
for every $ j=1,2,\cdots,d$,
\begin{align*}
  \Big|\frac{\partial}{\partial x_j}L((x,t),(y,s))\Big| = \Big|\frac{\partial}{\partial y_j}L((x,t),(y,s))\Big|\leq \frac{C}{{\rho((x,t),(y,s))}^{d+2\alpha+1}}.
\end{align*}
Gathering the above estimates with Taylor's formula, we have that for every $j\in\{1,2,\cdots,d\}$ and for any sufficiently small $(h,\theta)$,
\begin{align}\label{cndi.CZ.2-0}
  |L((x+h,t+\theta), (y,s)) & - L((x,t),(y,s))| \leq \frac{C |h|}{\rho((x,t),(y,s))^{d+2\alpha+1}} + \frac{C|\theta|}{\rho((x,t),(y,s))^{d+4\alpha}} \nonumber \\
  & \leq C \frac{ \rho((x,t),(x+h,t+\theta))}{\rho((x,t),(y,s))^{d+2\alpha+1}} + C \frac{\rho((x,t),(x+h,t+\theta))^{2\alpha}}{\rho((x,t),(y,s))^{d+4\alpha}} \nonumber \\
  & \leq  C \frac{\rho((x,t),(x+h,t+\theta))}{\rho((x,t),(y,s))^{d+2\alpha+1}},
\end{align}
and similarly,
\begin{equation}\label{cndi.CZ.2}
\begin{split}
  |L((x,t),(y+h,s+\theta))-L((x,t),(y,s))| \leq C \frac{\rho((y+h,s+\theta),(y,s))}{{\rho((x,t),(y,s))}^{d+2\alpha+1}}.
\end{split}
\end{equation}
Therefore, together with \eqref{cndi.CZ.1}--\eqref{cndi.CZ.2} and the $L^2$-boundedness of $\mathcal{A}_{2\alpha}$,
we can apply Lemma \ref{CZ} to infer that for every $1<p<\infty$,
\begin{equation}\label{Max.LpLp}
  \|\mathcal{A}_{2\alpha} \widetilde{f}\|_{L^p(\RR\times\RR^d)} \leq C \|\widetilde{f}\|_{L^p(\RR\times\RR^d)}.
\end{equation}

Next we regard $\mathcal{A}_{2\alpha} $ as a vector-valued Calder\'on-Zygmund operator on the real line $\RR$:
$$\mathcal{A}_{2\alpha} \widetilde{f}(\cdot,t)=\int_{\RR} L(t,s) \widetilde{f}(\cdot,s)\mathrm d s,$$
where $L(t,s)$ for every $\{(t,s)\in \RR^2: t\neq s\}$ is given by the integral
\begin{align*}
  L(t,s) \widetilde{f}(x) =
  \Lambda^{2\alpha} e^{-(t-s)\Lambda^{2\alpha}}\widetilde{f}(x,s)
  = \int_{\RR^d}\frac {1_{\{s<t\}}(s)}{t-s}\widetilde{\Omega}(x-y,t-s)\tilde{f}(y,s)\mathrm dy.
\end{align*}
In view of \eqref{Max.LpLp}, we have that $\mathcal{A}_{2\alpha}$ is continuously bounded from $L^p(\RR;L^p(\RR^d))$ to $L^p(\RR;L^p(\RR^d))$.
For every $0< s < t $, note that the Fourier multipliers of $L$ and $\partial_t L$ satisfy that for every $0\leq k \leq [\frac{d}2 ]+ 1$ and $\xi \neq 0$,
\begin{align*}
  \sup_{|\beta|=k} \big|\partial_\xi^\beta \big(|\xi|^{2\alpha} e^{-(t-s)|\xi|^{2\alpha}} \big)\big| \leq \frac{C |\xi|^{-k}}{t-s},\quad
  \textrm{and} \quad \sup_{|\beta|=k} \big|\partial_\xi^\beta \big(|\xi|^{4\alpha} e^{-(t-s)|\xi|^{2\alpha}} \big)\big| \leq \frac{C |\xi|^{-k}}{(t-s)^2},
\end{align*}
thus Lemma \ref{lem:multiplier} guarantees that
\begin{align*}
  \|L\|_{op(L^p\mapsto L^p)}\leq\frac C{t-s}, \quad \textrm{and}\quad
  \|\partial_t L\|_{op(L^p\mapsto L^p)}=\|\partial_s L\|_{op(L^p\mapsto L^p)}\leq \frac C{(t-s)^2}.
\end{align*}
Hence the assumptions in Lemma \ref{CZ} are all fulfilled, so that
$\mathcal{A}_{2\alpha} $ is continuously bounded from $L^q(\RR;L^p(\RR^d))$ to $L^q(\RR;L^p(\RR^d))$ for any $1<q,p<\infty$,
which completes the proof of Proposition \ref{lem.A2alpha} for $T=\infty$. 
\end{proof}

\section{A priori estimates}\label{APR}
\indent
We will derive the \textit{a priori} bounds for the vector fields $\bar{u}$, $w$ and $u$ in the successive subsections.

\subsection{A priori estimates for $\bar{u}$ solving 2D fractional Navier-Stokes system.}\label{subsec:L2}

\begin{proposition}\label{pro.u2dL2}
Let $\alpha\in (1/2,1)$, $u_0\in H^1\cap \dot B^\alpha_{p,2}(\RR^2)$ with $p>\frac2 {2\alpha-1}$.
Then the smooth solution $(\bar{u},\nabla\bar\pi)$ of 2D fractional Navier-Stokes equations \eqref{eq.2dNS} satisfies the following:
\begin{align}
  &\|\bar{u}\|_{L^\infty_t(H^1)}^2+\|\Lambda^\alpha \bar{u}\|_{L^2_t(H^1)}^2\leq C \|u_0\|_{H^1}^2, \label{est.u2d-2}\\
  & \|\bar{u}\|_{L^\infty_t(\dot B^\alpha _{p,2})} + \|(\partial_\tau \bar{u},\Lambda^{2\alpha} \bar{u},\nabla \bar{\pi})\|_{L^2_t(L^p)}
  \leq C \|u_0\|_{\dot B^{\alpha }_{p,2}} + C \| u_0\|_{H^1}^{\frac{4\alpha-1}{2\alpha -1}}.\label{est.u2d}
\end{align}
\end{proposition}

\begin{proof}[Proof of Proposition \ref{pro.u2dL2}]
The  $L^2$-energy estimate gives that for every $t\in (0,\infty]$,
\begin{equation}\label{est.u2d.L2}
\begin{split}
  \|\bar{u}\|_{L^\infty_t(L^2)}^2 + \|\Lambda^\alpha \bar{u}\|_{L^2_t(L^2)}^2 \leq \|u_0\|_{L^2}^2.
\end{split}
\end{equation}
Let $\bar\omega = \mathrm{curl}\, \bar{u} = \partial_{x_1} \bar{u}^2 - \partial_{x_2} \bar{u}^1$ be the vorticity of fluid, then it satisfies
\begin{equation}\label{eq.2dNS.w}
\left\{\begin{array}{ll}
\partial_t \bar\omega+\bar{u}\cdot\nabla \bar\omega+\Lambda^{2\alpha} \bar\omega=0,\\
\Div \bar{u}=0,\\
\bar\omega|_{t=0}=\omega_0 = \mathrm{curl}\, u_0.
\end{array}\right.
\end{equation}
The energy estimate of system \eqref{eq.2dNS.w} also leads to
\begin{equation}\label{est.w2dL2}
\begin{split}
  \|\nabla \bar{u}\|_{L^\infty_t(L^2)}^2 + \|\Lambda^\alpha \nabla\bar{u}\|_{L^2_t(L^2)}^2
  \leq C \big(\|\bar\omega\|_{L^\infty_t(L^2)}^2+\|\Lambda^\alpha \bar\omega\|_{L^2_t(L^2)}^2\big) \leq C \|\nabla u_0\|^2_{L^2},
\end{split}
\end{equation}
which combined with \eqref{est.u2d.L2} yields the desired estimate \eqref{est.u2d-2}.

Next, 
applying Proposition \ref{pro.Max-Regularity} with $q=2$ to system  \eqref{eq.2dNS}, we find that for $\delta_1 >0$ small enough,
\begin{align}\label{est.Max.u2d}
  \|\bar{u}\|_{L^\infty_t(\dot B^\alpha _{p,2})} + & \|(\partial_\tau \bar{u},\Lambda^{2\alpha} \bar{u},\nabla \bar\pi)\|_{L^2_t(L^p)}
  \leq C \|u_0\|_{\dot B^\alpha _{p,2}} + C \|\bar{u}\cdot\nabla \bar{u}\|_{L^2_t(L^p)} \nonumber \\
  \leq  & C \|u_0\|_{\dot B^\alpha _{p,2}} + C \big\|\|\bar{u}\|_{L^{\frac{p(p+\delta_1)}{\delta_1}}} \|\nabla \bar u\|_{L^{p+\delta_1}}\big\|_{L^2_t} \nonumber \\
  \leq & C \|u_0\|_{\dot B^\alpha_{p,2}} + C \big\| \|\bar{u}\|_{H^1} \|\nabla \bar{u}\|_{L^2}^{\theta_1} \|\Lambda^{2\alpha} \bar{u}\|_{L^p}^{1-\theta_1} \big\|_{L^2_t}\nonumber \\
  \leq & C \|u_0\|_{\dot B^\alpha_{p,2}} + C \|\bar{u}\|_{L^\infty_t (H^1)}
  \|\nabla \bar{u}\|_{L^2_t (L^2)}^{\theta_1} \|\Lambda^{2\alpha} \bar{u}\|_{L^2_t (L^p)}^{1-\theta_1}  \nonumber \\
  \leq & C \|u_0\|_{\dot B^\alpha _{p,2}} + C\|u_0\|_{H^1}^{1/\theta_1} \|\Lambda^\alpha \bar{u}\|_{L^2_t (L^2)}^\alpha
  \|\Lambda^{1+\alpha} \bar{u}\|_{L^2_t (L^2)}^{1-\alpha}  +   \frac{1}2  \|\Lambda^{2\alpha} \bar{u}\|_{L^2_t (L^p)} \nonumber \\
  \leq & C \|u_0\|_{\dot B^\alpha _{p,2}} + C\|u_0\|_{H^1}^{1/\theta_1 +1}
  +   \frac{1}2  \|\Lambda^{2\alpha} \bar{u}\|_{L^2_t (L^p)},
\end{align}
where $\theta_1 = \frac{(2\alpha-1)p^2 + (2\alpha-1)\delta_1 p -2\delta_1}{(p+\delta_1)(2\alpha p -2)} \in (0,1)$ (due to $p> \frac2 {2\alpha-1}$).
Noting that $\lim_{\delta_1\rightarrow 0+} \frac{1}{\theta_1} = \frac{2\alpha p -2}{(2\alpha-1)p} $,
we can choose a suitably small $\delta_1>0$ so that $\frac{1}{\theta_1} = \frac{2\alpha}{2\alpha-1}$, and
plug it into \eqref{est.Max.u2d} yields estimate \eqref{est.u2d}, as desired.
\end{proof}

\subsection{A priori estimates for $w$ solving the perturbed system \eqref{eq.INS-2dNS}.}\label{w.L2.lp}

\begin{proposition}\label{prop:w}
Let $\frac 12<\alpha<1,\, p>\frac 2{2\alpha-1}$, $u_0\in H^1\cap\dot B^{\alpha}_{p,2}(\RR^2)$,
and $\rho_0-1\in L^2\cap L^\infty(\RR^2)$. Assume that $\|\rho_0-1\|_{L^\infty}$ is small enough so that \eqref{con.Max.B0} is satisfied.
Then the smooth solution $(a, w,  p)=(\rho-1, u-\bar u, \pi-\bar\pi)$ of system \eqref{eq.INS-2dNS} satisfies the following estimates:
\begin{equation}\label{est.wL2.4}
\begin{split}
  E_2(w) := \|w\|_{L^\infty_t(L^2)} + \|\Lambda^\alpha w\|_{L^2_t(L^2)}
 \leq C \|a_0\|_{L^2\cap L^\infty} e^{C \|u_0\|_{H^1}^2},
\end{split}
\end{equation}
and
\begin{equation}\label{est.Max.w30}
\begin{split}
  E_p(w):= \|w\|_{L^\infty_t(\dot B^\alpha _{p,2})}+\|(\partial_\tau w,\Lambda^{2\alpha} w,\nabla  p)\|_{L^2_t(L^p)}\leq  C\|a_0\|_{L^2\cap L^\infty} e^{C \|u_0\|_{H^1\cap \dot B^\alpha_{p,2}}^2 },
\end{split}
\end{equation}
where $C>0$ depending only on $\alpha,p$. 
\end{proposition}

\begin{proof}[Proof of Proposition \ref{prop:w}]
It follows  easily from the equation $\eqref{eq.INS-2dNS}_1$ that
\begin{align}
  &\|a\|_{L^\infty_t(L^2\cap L^\infty)} = \|\rho-1\|_{L^\infty_t(L^2\cap L^\infty)} \leq \|\rho_0-1\|_{L^2\cap L^\infty}. \label{est.aL2}
\end{align}
Taking the inner product of equation $\eqref{eq.INS-2dNS}_2$ with  $w$ 
and using the condition $\Div u=0$ yield
\begin{equation}\label{est.wL2}
\begin{split}
  \frac 12\frac {\mathrm d}{\mathrm dt}\|w\|_{L^2}^2+\|\Lambda^\alpha  w\|_{L^2}^2
  &=-\frac{1}2 \frac {\mathrm d}{\mathrm dt}\|\sqrt{a}w\|_{L^2}^2
  -\int_{\RR^2} a \partial_t\bar{u} \cdot w\,\mathrm d x\\
  &\quad -\int_{\RR^2}a(\bar{u}\cdot\nabla \bar{u})\cdot w\,\mathrm dx-\int_{\RR^2}\rho ( w\cdot\nabla \bar{u})\cdot w\,\mathrm dx,
\end{split}
\end{equation}
where we also have used that
\begin{equation}\label{est.wL2.d}
\begin{split}
  \int_{\RR^2} \big( -a \partial_t w  - a ( u \cdot\nabla w)\big) \cdot w \dd x = & -\frac 12\int_{\RR^2}a\partial_t|w|^2\,\mathrm d x
  -\frac 12\int_{\RR^2} au\cdot\nabla |w|^2\,\mathrm d x\\
  =&-\frac 12\frac {\mathrm d}{\mathrm dt}\int_{\RR^2}a|w|^2\,\mathrm d x
  + \frac 12\int_{\RR^2}( \partial_ta+\Div(au)) |w|^2\,\mathrm d x\\
  =&-\frac 12\frac {\mathrm d}{\mathrm dt}\int_{\RR^2}a|w|^2\,\mathrm d x.
\end{split}
\end{equation}
By virtue of the following inequality
$\|\bar{u}\|_{L^{\frac2 {\alpha}}} \leq C \|\Lambda^\alpha \bar{u}\|_{L^2}^{2\alpha-1}\|\Lambda^{1+\alpha}\bar{u}\|_{L^2}^{2-2\alpha} \leq C \|\Lambda^\alpha \bar{u}\|_{H^1}$,
and using the Gagliardo-Nirenberg inequality, we obtain that
\begin{equation}\label{est.wL2.1}
\begin{split}
  \Big|\int_{\RR^2}a(\bar{u}\cdot\nabla \bar{u})w\,\mathrm dx\Big|
  \leq & C \|a\|_{L^\infty}\|\bar{u}\|_{L^{\frac 2{1-\alpha}}}\|\nabla \bar{u}\|_{L^{\frac 2\alpha }}\|w\|_{L^2}\\
  \leq & C \|a_0\|_{L^\infty} \|\Lambda^\alpha \bar{u}\|_{L^2} \|\Lambda^\alpha \bar{u}\|_{H^1} \|w\|_{L^2} \\
  \leq & C \|\Lambda^\alpha \bar{u}\|_{L^2}^2 \|w\|_{L^2}^2 + C \|a_0\|_{L^\infty}^2 \|\Lambda^\alpha \bar{u}\|_{H^1}^2 ,
\end{split}
\end{equation}
and
\begin{equation}\label{est.wL2.1-2}
\begin{split}
  \Big|\int_{\RR^2}\rho ( w\cdot\nabla \bar{u})\cdot w\,\mathrm dx\Big|
  \leq & C\|\rho\|_{L^\infty}\|w\|_{L^{\frac 2{1-\alpha}}}\|\nabla \bar{u}\|_{L^{\frac 2\alpha }}\|w\|_{L^2}\\
  \leq & C \|\rho_0\|_{L^\infty}\|\Lambda^\alpha w\|_{L^2} \|\Lambda^\alpha \bar{u}\|_{H^1} \|w\|_{L^2} \\
  \leq & C \|\rho_0\|_{L^\infty}^2 \|\Lambda^\alpha \bar{u}\|_{H^1}^2 \|w\|_{L^2}^2
  + \frac{1}{4}\|\Lambda^\alpha w\|_{L^2}^2 .
\end{split}
\end{equation}
Applying the Leray projection operator $\mathbb{P}:=\Id -\nabla \Delta^{-1} \Div$ to system \eqref{eq.2dNS} gives that
\begin{equation*}
\begin{split}
  \|\partial_t \bar{u}\|_{L^2} & \leq \|\Lambda^{2\alpha} \bar{u}\|_{L^2} + \|\mathbb{P} (\bar{u}\cdot\nabla \bar{u})\|_{L^2}
  \leq \|\Lambda^{2\alpha} \bar{u}\|_{L^2} + \|\bar{u}\cdot\nabla \bar{u}\|_{L^2},
\end{split}
\end{equation*}
which together with H\"older inequality and \eqref{est.wL2.1}--\eqref{est.wL2.1-2} leads to that
\begin{align}\label{est.wL2.2}
  \Big|\int_{\RR^2} a \partial_t \bar{u} \cdot w\,\mathrm dx\Big|
  \leq  &  \|\partial_t\bar{u}\|_{L^2} \|a\, w\|_{L^2 } \nonumber \\
  \leq & \|a\|_{L^{\frac 2\alpha }} \|\Lambda^{2\alpha}\bar{u}\|_{L^2} \|w\|_{L^{\frac 2{1-\alpha}}}
  + \|a\|_{L^{\infty}} \|\bar{u}\cdot\nabla \bar{u}\|_{L^2} \|w\|_{L^2} \nonumber \\
  \leq & C \|a_0\|_{L^2\cap L^\infty} \|\Lambda^\alpha \bar{u}\|_{H^1}  \|\Lambda^\alpha w\|_{L^2}
  + C \|a_0\|_{L^\infty} \|\Lambda^\alpha \bar{u}\|_{L^2} \|\Lambda^\alpha \bar{u}\|_{H^1} \|w\|_{L^2}  \nonumber \\
  \leq & C   \|\Lambda^\alpha \bar{u}\|_{L^2}^2 \|w\|_{L^2}^2 + C \|a_0\|_{L^2\cap L^\infty}^2  \|\Lambda^\alpha \bar{u}\|_{H^1}^2 + \frac{1}{4}\|\Lambda^\alpha w\|_{L^2}^2 .
\end{align}
Plugging \eqref{est.wL2.1} and  \eqref{est.wL2.2} into \eqref{est.wL2}, and integrating in time, we obtain that
\begin{align}\label{eat.wL2.3}
  & \|w(t)\|_{L^2}^2 + \int_0^t\|\Lambda^\alpha w\|_{L^2}^2\mathrm d\tau \nonumber\\
  \leq & - \| \sqrt{a} w(t) \|_{L^2}^2 + C \|a_0\|_{L^\infty}^2\int_0^t\|\Lambda^\alpha \bar{u}\|_{L^2}^2\mathrm d\tau +
  C \|a_0\|_{L^2\cap L^\infty}^2 \int_0^t\|\Lambda^\alpha \bar{u}\|_{H^1}^2 \mathrm d\tau \nonumber \\
  &  + C(\|\rho_0\|_{L^\infty}^2+1) \int_0^t \|\Lambda^\alpha \bar{u}\|_{H^1}^2 \|w\|_{L^2}^2\mathrm d \tau \nonumber \\
  \leq & C \|a_0\|_{L^2\cap L^\infty}^2 \|u_0\|_{H^1}^2 +  C(\|\rho_0\|_{L^\infty}^2+1) \int_0^t \|\Lambda^\alpha \bar{u}\|_{H^1}^2 \|w\|_{L^2}^2\mathrm d \tau.
\end{align}
Utilizing Gronwall's inequality together with $\eqref{est.u2d-2}$ ensures \eqref{est.wL2.4} as desired.

Next let us estimate the  $L^p$-type norm of $w$.
Applying Proposition \ref{pro.Max-Regularity} to equation $\eqref{eq.INS-2dNS}_2$ with $w|_{t=0}=0$ yields that
\begin{equation}\label{est.Max.w0}
\begin{split}
  E_p(w) \leq  C_1 \sum_{i=1}^6\|F_i\|_{L^2_t(L^p)},
\end{split}
\end{equation}
where
\begin{align*}
   &F_1 := -a\, \partial_t w, \quad F_2 := -a\, \partial_t \bar{u},\quad F_3 := - \rho ( \bar{u} \cdot\nabla w), \\
   &F_4 := - \rho ( w \cdot\nabla w),\quad F_5 := - a(\bar{u}\cdot\nabla \bar{u}),\quad F_6 := - \rho ( w\cdot\nabla \bar{u}).
\end{align*}
Let $\|a_0\|_{L^\infty} \leq \frac 1{2C_1}$, then immediately,
\begin{equation}\label{est.Max.w}
  \|F_1\|_{L^2_t(L^p)} \leq \|a\|_{L^\infty} \|\partial_\tau w\|_{L^2_t (L^p)} \leq \frac{1}{2C_1}  \|\partial_\tau w\|_{L^2_t (L^p)} .
\end{equation}
In view of \eqref{est.u2d}, we obtain
\begin{equation}\label{est.F_1F2}
\begin{split}
\|F_2\|_{L^2_t(L^p)}=\|a \partial_\tau \bar{u}\|_{L^2_t(L^p)} \leq \|a_0\|_{L^\infty} \| \partial_\tau\bar{u}\|_{L^2_t(L^p)} \leq C (1 + \|u_0\|_{H^1\cap \dot B^\alpha_{p,2}}^{\frac{4\alpha-1}{2\alpha-1}}).
\end{split}
\end{equation}
Making use of the H\"older inequality and interpolation inequality, we get that for some small $\delta_2>0$ chosen later,
\begin{align}\label{es.F3}
  \|F_3\|_{L^2_t(L^p)} \leq & \|\rho\|_{L^\infty_t(L^\infty)} \|\bar{u}\|_{L^\infty_t(L^{\frac{p(p+\delta_2)}{\delta_2}})} \|\nabla w\|_{L^2_t(L^{p+\delta_2})} \nonumber  \\
  \leq  & C \|\rho_0\|_{L^\infty} \|\bar{u}\|_{L^\infty_t (H^1)}
  \Big(\int_0^t\|\Lambda^\alpha  w\|_{L^2}^{2\theta_2}\|\Lambda^{2\alpha} w\|_{L^p}^{2(1-\theta_2)}\mathrm d\tau\Big)^{\frac 12} \nonumber \\
  \leq & C \|\rho_0\|_{L^\infty} \|u_0\|_{H^1} \|\Lambda^\alpha  w\|_{L^2_t(L^2)}^{\theta_2}
  \|\Lambda^{2\alpha} w\|_{L^2_t(L^p)}^{1-\theta_2} \nonumber \\
  \leq & C \|\rho_0\|_{L^\infty}^{1/\theta_2} \|u_0\|_{H^1}^{1/\theta_2} \|\Lambda^\alpha w\|_{L^2_t L^2}
  + \frac{1}{4 C_1} \|\Lambda^{2\alpha} w\|_{L^2_t(L^p)},
\end{align}
where $\theta_2=\frac{2\alpha-1 -(\frac2 {p}-\frac2 {p+\delta_2})}{\alpha+1-2/p}\in (0,1)$, thus from $\lim\limits_{\delta_2\rightarrow 0+ } \frac{1}{\theta_2} = \frac{\alpha+1-2/p}{2\alpha-1}$, we can choose $\delta_2>0$ so that $\frac{1}{\theta_2} = \frac{\alpha+1}{2\alpha-1}$ and
\begin{align}\label{est.F_3}
  \|F_3\|_{L^2_t(L^p)} &  \leq  C \|\rho_0\|_{L^\infty}^{\frac{\alpha+1}{2\alpha-1}} \|u_0\|_{H^1}^{\frac{\alpha+1}{2\alpha-1}}
  \|\Lambda^\alpha w\|_{L^2_t L^2}
  + \frac{1}{4 C_1} \|\Lambda^{2\alpha} w\|_{L^2_t(L^p)} \nonumber \\
  & \leq C \|\rho_0\|_{L^\infty}^{\frac{\alpha+1}{2\alpha-1}} \|u_0\|_{H^1}^{\frac{\alpha+1}{2\alpha-1}}
  \|a_0\|_{L^2\cap L^\infty} e^{C \|u_0\|_{H^1}^2}
  + \frac{1}{4 C_1} \|\Lambda^{2\alpha} w\|_{L^2_t(L^p)} \nonumber \\
  & \leq C \|a_0\|_{L^2\cap L^\infty} e^{C \|u_0\|_{H^1}^2}
  + \frac{1}{4 C_1} \|\Lambda^{2\alpha} w\|_{L^2_t(L^p)},
\end{align}
where in the above we have used \eqref{est.wL2.4} and the fact $\|\rho_0\|\leq 1+ \|a_0\|_{L^\infty} \leq 2$.
Similarly, by using the interpolation inequality, Young inequality and estimate \eqref{est.wL2.4} again, it follows that
\begin{align}\label{est.F_4}
  \|F_4\|_{L^2_t(L^p)}
  \leq &\|\rho\|_{L^\infty_t(L^\infty)}\|w\|_{L^\infty_t(L^\infty)}\|\nabla w\|_{L^2_t(L^p)} \nonumber \\
  \leq  & C \|\rho_0\|_{L^\infty} \|w\|_{L^\infty_t(L^2)}^{\theta_3} \|w\|_{L^\infty_t(\dot B^\alpha _{p,2})}^{1-\theta_3}
  \|\Lambda^\alpha  w\|_{L^2_t(L^2)}^{\theta_4} \|\Lambda^{2\alpha} w\|_{L^2_t(L^p)}^{1-\theta_4} \nonumber \\
  \leq & C \|\rho_0\|_{L^\infty} \big(\|w\|_{L^\infty_t(L^2)} + \|\Lambda^\alpha  w\|_{L^2_t(L^2)} \big)^{\theta_3 + \theta_4}
  \big( E_p(w) \big)^{2- (\theta_3+\theta_4)} \nonumber \\
  \leq & C \|\rho_0\|_{L^\infty}^{\frac2 {\theta_3+\theta_4}} \big(\|w\|_{L^\infty_t(L^2)} + \|\Lambda^\alpha  w\|_{L^2_t(L^2)} \big)^2
  + \frac{1}{2C_1} \big(E_p(w)\big)^2 \nonumber \\
  \leq & C \|a_0\|_{L^2\cap L^\infty}^2 e^{C \|u_0\|_{H^1}^2} +  \frac{1}{2C_1}\big( E_p(w)\big)^2 ,
\end{align}
where $\theta_3 = \frac{\alpha-2/p}{1+\alpha -2/p}\in (0,1)$, $\theta_4 = \frac{2\alpha-1}{1+\alpha-2/p} \in (0,1)$.
By arguing as \eqref{est.Max.u2d}, we have
\begin{align}\label{est.F_5}
  \|F_5\|_{L^2_t(L^p)} 
  \leq & \|a\|_{L^\infty_t(L^\infty)} \|\bar{u}\cdot \nabla \bar{u}\|_{L^2_t(L^p)} \nonumber \\
  \leq & C \|a\|_{L^\infty_t(L^\infty)} \|\bar{u}\|_{L^\infty_t (H^1)} \|\nabla \bar{u}\|_{L^2_t (L^2)}^{\frac{2\alpha-1}{2\alpha}}
  \|\Lambda^{2\alpha} \bar{u}\|_{L^2_t (L^2)}^{\frac{1}{2\alpha}} \nonumber \\
  \leq  & C \|a_0\|_{L^\infty} \|u_0\|_{H^1} \big( \|u_0\|_{H^1\cap \dot B^\alpha_{p,2}} + \|u_0\|_{H^1}^{\frac{4\alpha-1}{2\alpha-1}} \big).
\end{align}
The term $F_6$ can be estimated in the same manner as $F_3$ and $F_4$:
\begin{align}\label{est.F_6}
  \|F_6\|_{L^2_t(L^p)} 
  \leq &\|\rho_0\|_{L^\infty}\|w\|_{L^\infty_t(L^\infty)}\|\nabla \bar{u}\|_{L^2_t(L^p)} \nonumber \\
  \leq & C \|\rho_0\|_{L^\infty} \|w\|_{L^\infty_t(L^2)}^{\theta_3} \|w\|_{L^\infty_t(\dot B^\alpha _{p,2})}^{1-\theta_3}
  \big(\|\nabla \bar{u}\|_{L^2_t (L^2)} + \|\Lambda^{2\alpha} \bar{u}\|_{L^2_t(L^p)}\big) \nonumber \\
  \leq & C \|\rho_0\|_{L^\infty}^{\frac{1}{\theta_3}} \big(\|u_0\|_{H^1\cap \dot B^\alpha_{p,2}}
  + \|u_0\|_{H^1}^{\frac{4\alpha-1}{2\alpha-1}} \big)^{\frac{1+\alpha -2/p}{\alpha-2/p}} \|w\|_{L^\infty_t(L^2)} + \frac{1}{4C_1} \|w\|_{L^\infty_t(\dot B^\alpha _{p,2})} \nonumber\\
  \leq & C \|a_0\|_{L^2\cap L^\infty} e^{C \|u_0\|_{H^1\cap \dot B^\alpha_{p,2}}^2 } + \frac{1}{4C_1} \|w\|_{L^\infty_t(\dot B^\alpha _{p,2})}.
\end{align}
Collecting the above estimates on $F_i$, $i=1,...,6$, we deduce that
\begin{equation}\label{est.Max.w'}
\begin{split}
  E_p(w) \leq   C_2 \|a_0\|_{L^2\cap L^\infty} e^{C_2 \|u_0\|_{H^1\cap \dot B^\alpha_{p,2}}^2 }
  + \big( E_p(w) \big)^2 .
\end{split}
\end{equation}
By letting $c_0$ in \eqref{con.Max.B0} small enough so that $c_0\leq (4C_2)^{-1}$,
we conclude the estimate \eqref{est.Max.w30} by an elementary computation.
\end{proof}

\subsection{A priori estimates for $u$ solving the 2D fractional INS system \eqref{eq.frc-INNS}.}\label{es.u}

\indent\\ \indent
Based on the above estimates for $\bar{u}$ and $w$, we directly have the following bounds for $u$.
\begin{proposition}\label{pro.propri}
Let $\frac 12<\alpha<1,\, p>\frac 2{2\alpha-1}$ and $u_0\in H^1\cap\dot B^{\alpha}_{p,2}(\RR^2)$.
Let $\rho_0-1\in L^2\cap L^\infty(\RR^2)$ be satisfying the condition \eqref{con.Max.B0} with $c_0=c_0(\alpha,p)>0$
a sufficiently small constant.
Then there exists a constant
$C= C(\alpha,p)>0$ such that
\begin{equation}\label{est.priori}
  \|u\|_{L^\infty(\RR_+; L^2 \cap \dot B^\alpha _{p,2})}
  + \|(\Lambda^{2\alpha}u,\partial_t u, \nabla \pi)\|_{L^2(\RR_+; L^p)} + \|u\|_{L^2(\RR_+;\dot H^\alpha)}
  \leq C \big(1 + \|u_0\|_{H^1\cap\dot B^\alpha_{p,2}}^{\frac{4\alpha-1}{2\alpha-1}}\big),
\end{equation}
and for any $T>0$,
\begin{equation}\label{est.priori-2}
  \|\nabla u\|_{L^1_T(L^\infty)} \leq C \big(1 + \|u_0\|_{H^1\cap\dot B^\alpha_{p,2}}^{\frac{4\alpha-1}{2\alpha-1}}\big) \sqrt{T}.
\end{equation}
\end{proposition}

\begin{proof}[Proof of Proposition \ref{pro.propri}]
The proof of \eqref{est.priori} is obvious. As for the proof of \eqref{est.priori-2}, we use the Sobolev embedding to deduce
\begin{align}\label{est.nabla-u}
  \|\nabla u\|_{L^1_T(L^\infty)} \leq C \sqrt{T} \|u\|_{L^2_T(\dot H^{\alpha} \cap \dot W^{2\alpha,p})}
  \leq C \big(1 + \|u_0\|_{H^1\cap\dot B^\alpha_{p,2}}^{\frac{4\alpha-1}{2\alpha-1}}\big) \sqrt{T} .
\end{align}
\end{proof}

Under the additional stronger assumptions on $(\rho_0,u_0)$, we can show some more refined \textit{a priori} estimates, which are of use in the uniqueness part.
\begin{proposition}\label{prop:uni.M.bet}
Let $s\in (0,1)$, $u_0\in H^1\cap \dot B^{\alpha+s}_{p,2}(\RR^2)$, $\rho_0-1\in L^2\cap L^\infty\cap {\mathcal{M}(\dot B^s_{p,2})}\cap {\mathcal{M}(\dot B^{\frac 2p+1-2\alpha}_{p,1})}(\RR^2))$.
Then there exists a constant $c_*>0$ depending only on $\alpha,p,s$ so that, if the condition \eqref{eq:rho-cd} with this $c_*$
is satisfied,
we have
\begin{align}\label{est.nabla-u'}
\|\nabla u\|_{L^1(\RR_+;L^\infty)} \leq C  \Big(1 + \| u_0\|_{H^1\cap \dot B^\alpha_{p,2}}^{\frac{8\alpha-2}{2\alpha -1}} \Big) ,
\end{align}
and
\begin{equation}\label{est.M.bet}
\begin{split}
  & \quad \|u\|_{L^\infty(\RR_+;\dot B^{\alpha+s} _{p,2})} + \|(\Lambda^{2\alpha}u, \partial_t u, \nabla\pi)\|_{L^2(\RR_+;\dot B^{s}_{p,2})}
  + \| u\|_{L^2(\RR_+;\dot B^{2\alpha}_{p,1})} \\
  & \leq C
  \Big( 1+ \|u_0\|_{\dot B^{\alpha+s}_{p,2}} + \|u_0\|_{H^1\cap \dot B^\alpha_{p,2}}^{\frac{8(4\alpha-1)}{(2\alpha-1)^2}} \Big).
\end{split}
\end{equation}
\end{proposition}

\begin{proof}[Proof of Proposition \ref{prop:uni.M.bet}]

We note that the proof of \eqref{est.M.bet} below uses the sufficient smallness of $ \|\rho-1\|_{L^\infty_T(\mathcal{M}(\dot B^{s}_{p,2}))}$,
which according to \eqref{a.m} it needs to get the uniform estimate of $\|u\|_{L^1(\RR_+; \dot W^{1,\infty})}$ (estimate \eqref{est.priori-2} in Proposition
\ref{pro.propri} is insufficient for $T=\infty$).
Thus the whole proof of Proposition \ref{prop:uni.M.bet} is divided into two parts.

(1) First we prove \eqref{est.nabla-u'}.
Applying the Leray operator $\mathbb{P} :=\Id- \nabla \Delta^{-1}\Div$ to the equation $\eqref{eq.INS-2dNS}_2$ yields
\begin{align}\label{eq.w2}
  \partial_t w + \Lambda^{2\alpha} w  = \mathbb{P} F - \mathbb{P}(u\cdot\nabla w),\quad w|_{t=0} = w_0,
\end{align}
with
\begin{align*}
  F = -a \partial_t w -a \partial_t \bar{u}  - a ( u \cdot\nabla w) - a(\bar{u}\cdot\nabla \bar{u}) - \rho ( w\cdot\nabla \bar{u}),
\end{align*}
and then it follows from Lemma \ref{lem.es.Be} that
\begin{align}\label{est1.w.Lip}
  \|w\|_{\widetilde{L}^\infty_t (\dot B^{\frac 2p+1-2\alpha}_{p,1})} + \|w\|_{L^1_t (\dot B^{\frac 2p+1}_{p,1})}
  \leq & C\Big(\|\mathbb{P} F\|_{L^1_t (\dot B^{\frac 2p+1-2\alpha}_{p,1})} +
  \|\mathbb{P}(u\cdot\nabla w)\|_{L^1_t (\dot B^{\frac 2p+1-2\alpha}_{p,1})}\Big)\nonumber\\
  \leq & C\Big(  \| F\|_{L^1_t (\dot B^{\frac 2p+1-2\alpha}_{p,1})}
  + \|(u\cdot\nabla w)\|_{L^1_t (\dot B^{\frac 2p+1-2\alpha}_{p,1})} \Big) .
\end{align}
Noting that $\nabla p = \nabla \Delta^{-1}\Div (F-u\cdot\nabla w)$, we use the equation \eqref{eq.w2} to get
\begin{align}\label{est2.w.Lip}
  \|(\partial_\tau w,\nabla p)\|_{L^1_t (\dot B^{\frac 2p+1-2\alpha}_{p,1})}
  \leq C \Big(  \|F\|_{L^1_t (\dot B^{\frac 2p+1-2\alpha}_{p,1})} +
  \|u\cdot\nabla w\|_{L^1_t (\dot B^{\frac 2p+1-2\alpha}_{p,1})}\Big) .
\end{align}
By virtue of Definition \ref{def.mu} concerning the multiplier space, we find
\begin{align}\label{est3.w.Lip}
  \|F\|_{L^1_t (\dot B^{\frac 2p+1-2\alpha}_{p,1})}
  \leq & C \|a\|_{L^\infty_t \big(\mathcal{M}(\dot B^{\frac 2p+1-2\alpha}_{p,1})\big)}  \big\|(\partial_\tau w, \partial_\tau \bar{u},\bar u\cdot\nabla \bar u, u\cdot\nabla w, w\cdot\nabla \bar u )\big\|_{L^1_t (\dot B^{\frac 2p+1-2\alpha}_{p,1})} \nonumber\\
  & + C \|w\cdot\nabla \bar u  \|_{L^1_t (\dot B^{\frac 2p+1-2\alpha}_{p,1})} .
\end{align}
Utilizing the divergence-free condition of $u$ and interpolation inequality gives that
\begin{align*}
 & \|u\cdot\nabla w\|_{L^1_t (\dot B^{\frac 2p+1-2\alpha}_{p,1})}
 \leq \|u\otimes w\|_{L^1_t (\dot B^{\frac 2p+2-2\alpha}_{p,1})}\nonumber\\
 \leq & C\int^t_0\| u  \|_{L^\infty}\|w \|_{\dot B^{\frac 2p+2-2\alpha}_{p,1}}+\|u \|_{\dot B^{\frac 2p+2-2\alpha}_{p,1}}\| w  \|_{L^\infty} \dd \tau\nonumber\\
  \leq & C \big( \| u  \|_{L^2_t(L^\infty)} + \|u \|_{L^2_t (\dot B^{2/p+2-2\alpha}_{p,1})} \big)
  \cdot \big( \| w  \|_{L^2_t (L^\infty)} + \|w \|_{L^2_t (\dot B^{2/p+2-2\alpha}_{p,1})} \big).
\end{align*}
The interpolation inequality and Sobolev/Besov embedding ensure that
\begin{align*}
  \|u\|_{L^\infty}\leq C  \|u\|_{L^{\frac2{1-\alpha}}}^{\frac {2\alpha- 2/p}{1+\alpha-2/p}}
  \|u\|_{\dot W^{2\alpha,p}}^{\frac {1-\alpha}{1+\alpha-2/p}}
  \leq C\big( \|u\|_{\dot H^\alpha}+ \|u\|_{\dot W^{2\alpha,p}} \big),
\end{align*}
\begin{align*}
  \|u\|_{\dot B^{\frac 2p+2-2\alpha}_{p,1}} \leq
  C \|u\|_{\dot B^{2/p+\alpha-1}_{p,\infty}}^{\frac{2(2\alpha-1)- 2/p}{1+\alpha-2/p}}
  \|u\|_{\dot B^{ 2\alpha}_{p,\infty}}^{\frac{3(1-\alpha)}{1+\alpha-2/p}}
  \leq C \big( \|u\|_{\dot H^\alpha}+ \|u\|_{\dot W^{2\alpha,p}}\big),
\end{align*}
thus we arrive at
\begin{align}\label{est4.w.Lip'}
 \|u\cdot\nabla w\|_{L^1_t (\dot B^{\frac 2p+1-2\alpha}_{p,1})}
    \leq & C \|u\|_{L^2_t(\dot H^\alpha\cap \dot W^{2\alpha,p})}
     \|w\|_{L^2_t(\dot H^\alpha \cap \dot W^{2\alpha,p})}.
\end{align}
Making the above procedure analogously for $w\cdot\nabla\bar u$ and $\bar u\cdot\nabla \bar u$ yields
\begin{align}
  \|w\cdot\nabla \bar u\|_{L^1_t (\dot B^{\frac 2p+1-2\alpha}_{p,1})}
  \leq & C\|\bar u\|_{L^2_t(\dot H^\alpha\cap \dot W^{2\alpha,p})} \|w\|_{L^2_t(\dot H^\alpha\cap \dot W^{2\alpha,p})} ,\label{est5.w.Lip}\\
  \|\bar u\cdot\nabla \bar u\|_{L^1_t (\dot B^{\frac 2p+1-2\alpha}_{p,1})}
  \leq & C \|\bar u\|^2_{L^2_t(\dot H^\alpha\cap \dot W^{2\alpha,p})} .\label{est6.w.Lip}
\end{align}
For the estimate of $\bar{u}$ solving equation \eqref{eq.2dNS},
noting that by the interpolation and Besov embedding,
\begin{align*}
  \| u_0\|_{\dot B^{\frac 2p+1-2\alpha}_{p,1}}\leq C \| u_0\|_{\dot B^{2/p-1}_{p,\infty}}^{\frac {3\alpha-1- 2/p}{1+\alpha-2/p}}
  \| u_0\|_{\dot B^{\alpha}_{p,\infty}}^{\frac{2-2\alpha}{1+\alpha-2/p}}\leq C \|u_0\|_{L^2\cap \dot B^\alpha_{p,2}},
\end{align*}
along the similar lines as deducing \eqref{est1.w.Lip}-\eqref{est2.w.Lip}, we obtain
\begin{align}\label{est1.bu.Lip}
  &\|\bar u\|_{\widetilde{L}^\infty_t (\dot B^{\frac 2p+1-2\alpha}_{p,1})}+ \|\bar u\|_{L^1_t (\dot B^{\frac 2p+1}_{p,1})}
  + \|(\partial_\tau \bar u, \nabla\bar\pi)\|_{L^1_t (\dot B^{\frac 2p+1-2\alpha}_{p,1})}\nonumber\\
  \leq & C\big(  \| u_0\|_{\dot B^{\frac 2p+1-2\alpha}_{p,1}}+\|\bar u\cdot\nabla\bar u\|_{L^1_t (\dot B^{\frac 2p+1-2\alpha}_{p,1})}\big) \nonumber\\
  \leq & C \big(\| u_0\|_{L^2\cap\dot B^{\alpha}_{p,2}}+\|\bar u\|^2_{L^2_t(\dot H^\alpha\cap \dot W^{2\alpha,p})} \big),
\end{align}
Gathering estimates \eqref{est1.w.Lip}--\eqref{est1.bu.Lip} together leads to that
\begin{align}\label{est.w.Lip''}
  &\|w\|_{\widetilde{L}^\infty_t (\dot B^{\frac 2p+1-2\alpha}_{p,1})}+ \|w\|_{L^1_t (\dot B^{\frac 2p+1}_{p,1})}
  +\|(\partial_\tau w, \nabla p)\|_{L^1_t (\dot B^{\frac 2p+1-2\alpha}_{p,1})}\nonumber\\
  \leq & C\|a\|_{L^\infty_t \big(\mathcal{M}(\dot B^{\frac 2p+1-2\alpha}_{p,1})\big)}
  \Big(\|\partial_\tau w \|_{L^1_t (\dot B^{\frac 2p+1-2\alpha}_{p,1})} + \| u_0\|_{L^2\cap\dot B^{\alpha}_{p,2}}
  + \|(\bar u,w)\|^2_{L^2_t(\dot H^\alpha\cap \dot W^{2\alpha,p})} \Big) \nonumber\\
  &+C \|(\bar u, w)\|_{L^2_t(\dot H^\alpha\cap \dot W^{2\alpha,p})}
  \|w\|_{L^2_t(\dot H^\alpha\cap \dot W^{2\alpha,p})} .
\end{align}
Observe that owing to Lemma \ref{lem:f.mu} and inequality \eqref{est1.bu.Lip},
\begin{align}\label{est.a.M1}
  & \|a\|_{L^\infty_t \big(\mathcal{M}(\dot B^{\frac 2p+1-2\alpha}_{p,1})\big)}
  \leq C \|a_0\|_{\mathcal{M}(\dot B^{\frac 2p+1-2\alpha}_{p,1})} e^{ C \int_0^t\|\nabla u\|_{L^\infty}\dd\tau}\nonumber\\
  \leq & C \|a_0\|_{\mathcal{M}(\dot B^{\frac 2p+1-2\alpha}_{p,1})}e^{C \|w\|_{L^1_t (\dot B^{2/p+1}_{p,1})}+ C\|\bar u\|_{L^1_t (\dot B^{2/p+1}_{p,1})} }\nonumber\\
  \leq & C \|a_0\|_{\mathcal{M}(\dot B^{\frac 2p+1-2\alpha}_{p,1})} e^{C \|w\|_{L^1_t (\dot B^{2/p+1}_{p,1})}} \exp\Big\{C\big(\| u_0\|_{L^2\cap\dot B^{\alpha}_{p,2}}+\|\bar u\|^2_{L^2_t( \dot H^\alpha\cap \dot W^{2\alpha,p})}\big)\Big\} .
\end{align}
Recalling that Propositions \ref{pro.u2dL2} and \ref{prop:w} guarantee that
\begin{align*}
  \|\bar{u}\|_{L^2_t (\dot H^\alpha\cap \dot W^{2\alpha,p})} \leq C \big( 1+ \|u_0\|_{H^1 \cap \dot B^\alpha_{p,2}}^{\frac{4\alpha-1}{2\alpha -1}}\big), \quad
  \|w\|_{L^2_t (\dot H^\alpha \cap \dot W^{2\alpha,p})} \leq C\|a_0\|_{L^2\cap L^\infty} e^{C \|u_0\|_{H^1\cap \dot B^\alpha_{p,2}}^2 },
\end{align*}
we insert the above estimates and \eqref{est.a.M1} into \eqref{est.w.Lip''} to get
\begin{align}\label{est'.w.Lip}
  &\|w\|_{\widetilde{L}^\infty_t (\dot B^{\frac 2p+1-2\alpha}_{p,1})}+ \|w\|_{L^1_t (\dot B^{\frac 2p+1}_{p,1})}+\|(\partial_\tau w, \nabla p)\|_{L^1_t (\dot B^{\frac 2p+1-2\alpha}_{p,1})}\nonumber\\
  \leq & \overline{C} \|a_0\|_{\mathcal{M}(\dot B^{\frac 2p+1-2\alpha}_{p,1})}e^{\overline{C}\big(1+\| u_0\|^2_{H^1\cap \dot B^{\alpha}_{p,2}} + \| u_0\|_{H^1}^{\frac{8\alpha-2}{2\alpha -1}}\big)} \cdot e^{\overline{C}\|w\|_{L^1_t (\dot B^{2/p+1}_{p,1})}}
  \cdot\Big(\|\partial_\tau w \|_{L^1_t (\dot B^{\frac 2p+1-2\alpha}_{p,1})}+1\Big) \nonumber \\
  & + \overline{C} \|a_0\|_{L^2\cap L^\infty}  e^{\overline{C} \|u_0\|_{H^1\cap \dot B^\alpha_{p,2}}^2 },
\end{align}
where $\overline{C}>0$ is a constant depending only on $\alpha,p$.

According to Proposition \ref{prop:w} and the continuous embedding $\dot W^{2\alpha,p}\cap \dot H^\alpha \hookrightarrow \dot B^{2/p+1}_{p,1}$ and $\dot B^0_{2,\infty}\cap \dot B^\alpha_{p,\infty}\hookrightarrow \dot B^{2/p+1-2\alpha}_{p,1}$,
we know that $w\in \widetilde{L}^\infty_T(\dot B^{2/p+1-2\alpha}_{p,1}) \cap L^1_T (\dot B^{2/p+1}_{p,1})$ for any $T>0$.
Since $w\in \widetilde{L}^\infty_T(\dot B^{2/p+1-2\alpha}_{p,1})$, by a high-low frequency decomposition argument,
one easily deduce that $w\in C([0,T]; \dot B^{2/p+1-2\alpha}_{p,1})$.
Let $T^*>0$ be the maximal existence time such that $w\in C([0,T^*);\dot B^{2/p+1-2\alpha}_{p,1})\cap L^1 ([0,T^*);\dot B^{2/p+1}_{p,1})$.
Denote by $T'$ as
\begin{align}\label{T'.max}
  T'\,{\stackrel{\mathrm{def}}{=}}\sup \Big\{t<T^*: \|w\|_{\widetilde{L}^\infty_t (\dot B^{\frac{2}{p}+1-2\alpha}_{p,1})} + \|w\|_{L^1_t (\dot B^{\frac{2}{p}+1}_{p,1})}\leq 1\Big\}.
\end{align}
We claim $T'=\infty$.  Indeed, by setting $\|a_0\|$ small enough so that
\begin{align}\label{small.con1}
  \overline{C} e^{\overline{C}} \|a_0\|_{\mathcal{M}(\dot B^{\frac 2p+1-2\alpha}_{p,1})} \exp\Big\{\overline{C}(1+\| u_0\|^2_{H^1\cap \dot B^{\alpha}_{p,2}}+ \| u_0\|_{H^1}^{\frac{8\alpha-2}{2\alpha -1}}) \Big\} \leq \frac{1}{4},
\end{align}
and
\begin{align}\label{small.con2}
  \overline{C} \|a_0\|_{L^2\cap L^\infty} e^{\overline{C} \|u_0\|_{H^1\cap \dot B^\alpha_{p,2}}^2 }\leq \frac{1}{4} ,
\end{align}
we infer from \eqref{est'.w.Lip} that
\begin{align}\label{est''.w.Lip}
  \|w\|_{\widetilde{L}^\infty_{T'} (\dot B^{\frac 2p+1-2\alpha}_{p,1})} + \|w\|_{L^1_{T'} (\dot B^{\frac 2p+1}_{p,1})}
  + \frac{1}{2} \|\partial_t w \|_{L^1_{T'} (\dot B^{\frac 2p+1-2\alpha}_{p,1})} + \|\nabla p\|_{L^1_{T'} (\dot B^{\frac 2p+1-2\alpha}_{p,1})}\leq \frac{1}{2}.
\end{align}
If $T'<\infty$, then it can be proceeded beyond and this contradicts with the maximality of $T'$ defined by \eqref{T'.max}, hence we conclude that $T'=T^*=\infty$.

Hence, by combining the estimates \eqref{est1.bu.Lip} and \eqref{est''.w.Lip}, we get the desired estimate \eqref{est.nabla-u'}:
\begin{align*}
  \|\nabla u\|_{L^1(\RR_+; L^\infty)} \leq C \|( \bar{u}, w)\|_{L^1(\RR_+; \dot B^{\frac{2}{p}+1}_{p,1})}
  \leq C \Big(1 + \|u_0\|_{H^1 \cap \dot B^\alpha_{p,2}}^{\frac{8\alpha-2}{2\alpha -1}} \Big).
\end{align*}

(2) Next we show \eqref{est.M.bet}.
Applying fractional Laplacian operator $\Lambda^s$ to equation $\eqref{eq.frc-INNS}_2$ yields that
\begin{align*}
 \partial_t\Lambda^s u+\Lambda^{2\alpha+s} u+\nabla\Lambda^s \pi= -\Lambda^s\big(a\,\partial_t u\big) - \Lambda^s\big(\rho u\cdot\nabla u) =: F_s,
\end{align*}
with $a=\rho-1$ and $\Div u=0$.
By using the Leray operator $\PP := \Id - \nabla \Delta^{-1}\Div$, we get
\begin{align*}
  \partial_t (\Lambda^s u) + \Lambda^{2\alpha} (\Lambda^s u) = \PP F_s, \quad (\Lambda^s u)|_{t=0}= \Lambda^s u_0 .
\end{align*}
Similarly as deriving \eqref{est1.w.Lip}-\eqref{est2.w.Lip}, we obtain that for every $T\in (0,\infty]$,
\begin{align}\label{est.bet}
  \|\Lambda^s u\|_{\widetilde{L}^\infty_T(\dot B^\alpha_{p,2})} + \|\Lambda^s u\|_{L^2_T(\dot B^{2\alpha}_{p,2})} + \|(\nabla \pi, \partial_t u)\|_{L^2_T (\dot B^s_{p,2})}
  & \leq  C \big( \|\Lambda^s u_0\|_{ \dot B^\alpha_{p,2}} + \|\PP F_s\|_{L^2_T(\dot B^0_{p,2})}\big) \nonumber \\
  & \leq  C \big( \|u_0\|_{ \dot B^{\alpha+s}_{p,2}} + \|F_s\|_{L^2_T(\dot B^0_{p,2})}\big).
\end{align}
By using Definition \ref{def.mu}, we infer that
\begin{align}\label{est.bet1}
  \| F_s\|_{L^2_T(\dot B^0_{p,2})} & \leq  \| a \,\partial_t u\|_{L^2_T(\dot B^s_{p,2})}
  + \|(1+a) (u\cdot\nabla u)\|_{L^2_T (\dot B^s_{p,2})} \nonumber \\
  & \leq C \|a\|_{L^\infty_T( \mathcal{M}(\dot B^{s}_{p,2}))} \|\partial_t u\|_{L^2_T(\dot B^{s}_{p,2})}
  + C (1+\|a\|_{L^\infty_T( \mathcal{M}(\dot B^{s}_{p,2}))}) \|u\cdot\nabla u\|_{L^2_T(\dot B^{s}_{p,2})}.
\end{align}
Utilizing the interpolation inequality and \eqref{est.priori} leads to
\begin{align}\label{est.bet2}
  \|u\cdot\nabla u\|_{L^2_T(\dot B^{s}_{p,2})} 
  \leq & C \|u\|_{L^\infty_T(L^\infty)}\|u\|_{ L^2_T(\dot B^{1+s}_{p,2})}\nonumber\\
  \leq & C \|u\|_{L^\infty_T(L^2\cap \dot B^{\alpha}_{p,2})}\|u\|_{L^2_T(\dot H^{\alpha})}^{\theta_5}\|u\|_{L^2_T(\dot B^{2\alpha+s}_{p.2})}^{1-\theta_5}\nonumber\\
  \leq &  C_\varepsilon \|u\|_{L^\infty_T(L^2\cap \dot B^{\alpha}_{p,2})}^{1/\theta_5}\|u\|_{L^2_T(\dot H^{\alpha})}+ \varepsilon \|u\|_{L^2_T(\dot B^{2\alpha+s}_{p.2})} \nonumber \\
  \leq & C_\varepsilon \Big( 1 + \|u_0\|_{H^1\cap \dot B^\alpha_{p,2}}^{\frac{8(4\alpha-1)}{(2\alpha-1)^2}} \Big) + \varepsilon \|u\|_{L^2_T(\dot B^{2\alpha+s}_{p.2})},
\end{align}
where $\theta_5=\frac {2\alpha-1}{\alpha+1+s- 2/p}$ and $\varepsilon>0$ is a constant chosen later.
By virtue of \eqref{a.m} and \eqref{est.nabla-u'}, we get
\begin{align}\label{est.bet3}
  \|a\|_{L^\infty_T( \mathcal{M}(\dot B^{s}_{p,2}))} \leq \|a_0\|_{ \mathcal{M}(\dot B^{s}_{p,2})}e^{C\|\nabla u\|_{L^1_T(L^\infty)}}
  \leq C \|a_0\|_{ \mathcal{M}(\dot B^{s}_{p,2})} \exp\Big\{C  \big(1 + \|u_0\|_{H^1\cap \dot B^\alpha_{p,2}}^{\frac{8\alpha-2}{2\alpha-1}}\big)  \Big\}.
\end{align}
Collecting estimates \eqref{est.bet}--\eqref{est.bet3}  and using  the interpolation inequality $\dot B^{2\alpha+s}_{p,2}\cap \dot H^\alpha\hookrightarrow \dot B^{2\alpha}_{p,1}$ together with \eqref{est.priori},
we see that for every $T\in (0,\infty]$,
\begin{align}\label{est.u-imp-reg}
  & \quad \|u\|_{\widetilde{L}^\infty_T (\dot B^{\alpha+s}_{p,2})} + \|u\|_{L^2_T(\dot B^{2\alpha+s}_{p,2})} + \|(\nabla \pi, \partial_t u)\|_{L^2_T (\dot B^s_{p,2})} + \|u\|_{L^2_T(\dot B^{2\alpha}_{p,1})} \nonumber \\
  & \leq C \|u_0\|_{\dot B^{\alpha+s}_{p,2}} +  C_\varepsilon \Big( 1 + \|u_0\|_{H^1\cap \dot B^\alpha_{p,2}}^{\frac{8(4\alpha-1)}{(2\alpha-1)^2}} \Big) \nonumber \\
  & \quad + C \|a_0\|_{\mathcal{M}(\dot B^s_{p,2})} \exp\Big\{C  \big(1 + \|u_0\|_{H^1\cap\dot B^\alpha_{p,2}}^{\frac{8\alpha-2}{2\alpha-1}}\big) \Big\} \|\partial_t u\|_{L^2_T(\dot B^s_{p,2})} \nonumber \\
  & \quad +  C \varepsilon \Big(1 + \|a_0\|_{\mathcal{M}(\dot B^s_{p,2})} \exp\Big\{C  \big(1 + \|u_0\|_{ H^1\cap\dot B^\alpha_{p,2}}^{\frac{8\alpha-2}{2\alpha-1}}\big)  \Big\} \Big ) \|u\|_{L^2_T (\dot B^{2\alpha+s}_{p,2})} ,
\end{align}
so that by letting $c_*>0$ in \eqref{eq:rho-cd} small enough and then $\varepsilon>0$ small enough, we conclude the desired estimate \eqref{est.M.bet}.
\end{proof}

\section{Proof of Theorem \ref{Main.thm}: the existence part}\label{existence}
In Subsection \ref{subsec:exi1}, as a first step we show the global well-posedness of strong solution to system \eqref{eq.frc-INNS} with additional regularity assumption $\nabla\rho_0\in L^{\frac{2}{\alpha}}(\RR^2)$,
and then by the compactness argument we prove the global existence of solution to system \eqref{eq.frc-INNS} with rough density in Subsection \ref{subsec:exi2}.

\subsection{Global well-posedness result for 2D fractional INS system with regular density.}\label{subsec:exi1}
Our main result in this subsection is the following proposition.
\begin{proposition}\label{Main.pro}
Let $\frac{1}{2}<\alpha<1,\, p>\frac{2}{2\alpha-1}$ and $u_0\in H^1\cap\dot B^{\alpha}_{p,2}(\RR^2)$, $\rho_0-1\in L^2\cap L^\infty(\RR^2)$
be with the smallness condition \eqref{con.Max.B0}.
In addition, assume $\nabla \rho_0 \in L^{\frac{2}{\alpha}}(\RR^2)$.
Then there exists  a unique global-in-time solution $(\rho, u,\nabla \pi)$ to the 2D fractional INS system \eqref{eq.frc-INNS}
which fulfills estimates \eqref{est.Max.w3-0}-\eqref{est.Max.w3}. Besides, it holds that for any $T>0$,
\begin{align}\label{est.nab.a}
  \|\nabla \rho\|_{L^\infty_T(L^{\frac{2}{\alpha}})}
  \leq \|\nabla \rho_0\|_{L^{\frac{2}{\alpha}}} \exp\Big\{ C\big(1+ \|u_0\|_{H^1\cap \dot B^\alpha_{p,2}}^{\frac{4\alpha -1}{2\alpha-1}}\big) \sqrt{T} \Big\}.
\end{align}
\end{proposition}

\begin{proof}[Proof of Proposition \ref{Main.pro}]

Since $u_0\in H^1\cap \dot B^\alpha_{p,2}(\RR^2)$ with $\alpha\in (\frac{1}{2},1)$ and $p>\frac{2}{2\alpha-1}$,
according to Proposition \ref{pro.u2dL2} and the standard compactness theory, there exists a unique global-in-time strong solution
$\bar{u}\in C(\RR_+;H^1\cap \dot B^\alpha_{p,2})\cap L^2(\RR_+; H^{1+\alpha})$ to the 2D fractional Navier-Stokes system \eqref{eq.2dNS} which satisfies estimates \eqref{est.u2d-2}-\eqref{est.u2d}.
Thus it suffices to treat the global existence issue of the perturbed system \eqref{eq.INS-2dNS}.
We divide the proof into the following several steps.

\textit{Step 1: construction of approximate solutions.}

We consider
$(w^{n+1}, a^{n+1} ) \;(n\in\NN)$ as the solutions to the following approximate perturbed system
\begin{equation}\label{apr.INS-2dNS}
\left\{\begin{array}{ll}
  \partial_t a^{n+1}+u^n\cdot\nabla a^{n+1}=0,\\
  \partial_t w^{n+1}+u^n\cdot\nabla w^{n+1}+\Lambda^{2\alpha} w^{n+1} +\nabla p^{n+1} = \sum\limits_{i=1}^5 F^n_i,\\
  \Div w^{n+1}=0,\\
  (a^{n+1}, w^{n+1})|_{t=0}=(a_0,0),
\end{array}\right.
\end{equation}
where 
$u^n=w^n + \bar{u}$ and
\begin{equation}\label{eq.Fn}
\begin{split}
  &F^n_1=-a^n \partial_t w^{n+1},\quad  F^n_2=- a^n (\partial_t\bar{u}),\quad F_3^n = -a^n (u^{n-1}\cdot\nabla w^{n+1}),\\ 
  &F_4^n=-a^n\big( \bar{u}\cdot\nabla \bar{u}\big),\quad F_5^n=-(1+a^n) \big( w^{n+1}\cdot \nabla  \bar{u} \big).
\end{split}
\end{equation}
We also set $u^{-1}(t,x)\equiv 0,\,  w^0(t,x)\equiv 0,\,  a^0(t,x)\equiv a_0(x),\,  u^0(t,x)\equiv u_0(x)$.
It can admit a unique solution of system \eqref{apr.INS-2dNS} by Proposition \ref{pro.Max-Regularity}
and the Banach fixed point theorem.
We treat the nonlinearity $u^n\cdot \nabla w^{n+1}$ as a perturbation and find a solution via a contraction map for small time intervals.
Solvability of the transport equation follows directly from the method of characteristics.
The extension to the global-in-time solutions can use the following uniform estimates. We omit the details of this part.

\textit{Step 2: uniform-in-$n$ estimates for approximate solutions.}

We shall derive the  uniform-in-$n$ estimates for approximate solutions by the induction method.

First, $(w^1, a^1 )$ satisfy that
\begin{equation}\label{apr.n1}
\left\{\begin{array}{ll}
  \partial_t a^{1}+u_0\cdot\nabla a^{1}=0,\\
  \partial_t w^{1}+u_0\cdot\nabla w^{1}+\Lambda^{2\alpha} w^{1} +\nabla p^{1} = \sum\limits_{i=1}^5 F^0_i,\\
  \Div w^{1}=0,\\
  (a^{1}, w^{1})|_{t=0}=(a_0,0),
\end{array}\right.
\end{equation}
where 
$u^1=w^1 + \bar{u}$ and
\begin{equation}\label{eq.F0i}
\begin{split}
  &F^0_1=-a_0 \partial_t w^{1},\quad  F^0_2=- a_0 (\partial_t\bar{u}),\quad F_3^0 =0,\\ 
  &F_4^0=-a_0\big( \bar{u}\cdot\nabla \bar{u}\big),\quad F_5^0=-(1+a_0) \big( w^{1}\cdot \nabla  \bar{u} \big) .
\end{split}
\end{equation}
We have that $ \|a^1(t,\cdot)\|_{L^2\cap L^\infty} \leq \|a_0\|_{L^2\cap L^\infty}$ for any $t>0$.
Along the same line as deriving \eqref{eat.wL2.3}, we find
\begin{align}\label{eaw1.L2}
  & \|w^1(t)\|_{L^2}^2 + \int_0^t\|\Lambda^\alpha w^1\|_{L^2}^2\mathrm d\tau \nonumber\\
  \leq & - \| \sqrt{a} w^1(t) \|_{L^2}^2 + C \|a_0\|_{L^\infty}^2\int_0^t\|\Lambda^\alpha \bar{u}\|_{L^2}^2\mathrm d\tau +
  C \|a_0\|_{L^2\cap L^\infty}^2 \int_0^t\|\Lambda^\alpha \bar{u}\|_{H^1}^2 \mathrm d\tau \nonumber \\
  &  + C(\|\rho_0\|_{L^\infty}^2+1) \int_0^t \|\Lambda^\alpha \bar{u}\|_{H^1}^2 \|w^1\|_{L^2}^2\mathrm d \tau \nonumber \\
  \leq & C \|a_0\|_{L^2\cap L^\infty}^2 \|u_0\|_{H^1}^2 +  C(\|\rho_0\|_{L^\infty}^2+1) \int_0^t \|\Lambda^\alpha \bar{u}\|_{H^1}^2 \|w^1\|_{L^2}^2\mathrm d \tau ,
\end{align}
which combined with the Gronwall inequality yields that
\begin{equation*}\label{esw1.L2'}
\begin{split}
  \|w^1\|_{L^\infty_t(L^2)} + \|\Lambda^\alpha w^1\|_{L^2_t(L^2)}
 \leq C \|a_0\|_{L^2\cap L^\infty} e^{C \|u_0\|_{H^1}^2} .
\end{split}
\end{equation*}

Applying Proposition \ref{pro.Max-Regularity} to equation $\eqref{apr.n1}_2$ gives
\begin{equation*}\label{esw1.M}
\begin{split}
  \|w^1\|_{L^\infty_t(\dot B^\alpha _{p,2})}+\|(\partial_\tau w^1,\Lambda^{2\alpha} w^1,\nabla  p^1)\|_{L^2_t(L^p)} \leq  C \sum_{i=1}^5\|F^0_i\|_{L^2_t(L^p)} + C \|u_0\cdot\nabla w^1\|_{L^2_t (L^p)},
\end{split}
\end{equation*}
where $F^0_i, (i=1,2,3,4,5)$ is given by \eqref{eq.F0i}.
By arguing as \eqref{es.F3} and \eqref{est.F_3}, we get
\begin{align}\label{es.F06}
  \|u_0 \cdot\nabla w^1\|_{L^2_t(L^p)} & \leq   \|u_0\|_{L^{\frac{p(p+\delta_2)}{\delta_2}}} \|\nabla w^1\|_{L^2_t(L^{p+\delta_2})}  \nonumber \\
  & \leq C \|u_0\|_{H^1}^{\frac{\alpha+1}{2\alpha-1}}
  \|a_0\|_{L^2\cap L^\infty} e^{C \|u_0\|_{H^1}^2}
  + \frac{1}{4C } \|\Lambda^{2\alpha} w^1\|_{L^2_t(L^p)} \nonumber \\
  & \leq C \|a_0\|_{L^2\cap L^\infty} e^{C \|u_0\|_{H^1}^2}
  + \frac{1}{4C} \|\Lambda^{2\alpha} w^1\|_{L^2_t(L^p)} .
\end{align}
For $F^0_i$, by making the similar procedure as deducing \eqref{est.F_5}, \eqref{est.F_6}, we see that
\begin{align}
  \sum_{i=1}^5 \|F_i^0\|_{L^2_t (L^p)} & \leq \|a_0\|_{L^\infty} \big(\|\partial_\tau w^1\|_{L^2_t (L^p)} +  \|\partial_\tau \bar{u}\|_{L^2_t (L^p)} + \|\bar{u}\cdot \nabla \bar{u}\|_{L^2_t(L^p)} \big) \nonumber \\
  &\quad + \|\rho_0\|_{L^\infty} \|w^1\|_{L^\infty_t(L^\infty)} \|\nabla \bar{u}\|_{L^2_t(L^p)} \nonumber \\
  & \leq C \|a_0\|_{L^2\cap L^\infty} e^{C \|u_0\|_{H^1\cap \dot B^\alpha_{p,2}}^2} + \frac{1}{4C} \|w^1\|_{L^\infty_t (\dot B^\alpha_{p,2})} +  \|a_0\|_{L^\infty} \|\partial_\tau w^1\|_{L^2_t (L^p)} .
\end{align}
By letting $\|a_0\|_{L^\infty} \leq \frac{1}{4C}$ and using the assumption \eqref{con.Max.B0}, we gather the above estimates to obtain
\begin{equation}\label{esw1.M'}
\begin{split}
  \|w^1\|_{L^\infty_t(\dot B^\alpha _{p,2})}+\|(\partial_\tau w^1,\Lambda^{2\alpha} w^1,\nabla  p^1)\|_{L^2_t(L^p)}  \leq   C \|a_0\|_{L^2\cap L^\infty} e^{C \|u_0\|_{H^1\cap \dot B^\alpha_{p,2}}^2} \leq 1.
\end{split}
\end{equation}

Now under the smallness condition \eqref{con.Max.B0}, we suppose that, for any $n\in \mathbb{Z}_+$ and $k\leq n$, there hold that for any $t>0$,
\begin{equation}\label{eswk.L2}
\begin{split}
  E_2(w^k) := \|w^k\|_{L^\infty_t(L^2)} + \|\Lambda^\alpha w^k\|_{L^2_t(L^2)}
 \leq C \|a_0\|_{L^2\cap L^\infty} e^{C \|u_0\|_{H^1}^2},
\end{split}
\end{equation}
and
\begin{equation}\label{eswk.M}
\begin{split}
  E_p(w^k):=\|w^k\|_{L^\infty_t(\dot B^\alpha _{p,2})}+\|(\partial_\tau w^k,\Lambda^{2\alpha} w^k,\nabla  p^k)\|_{L^2_t(L^p)}
  \leq \widetilde{C} \|a_0\|_{L^2\cap L^\infty} e^{\widetilde{C} \|u_0\|_{H^1\cap \dot B^\alpha_{p,2}}^2 },
\end{split}
\end{equation}
with $C,\widetilde{C}>0$ depending only on $\alpha,p$, then we shall show the same uniform estimates for $w^{n+1}$.

Since $u^n = w^n + \bar{u}$, by using the induction assumption and estimating as \eqref{est.nabla-u},
we have $u^n\in L^1(0,t;\dot W^{1,\infty})$ for any $t>0$, thus it follows from the flow property of transport equation that
$ \|a^{n+1}(t,\cdot)\|_{L^2\cap L^\infty} \leq \|a_0\|_{L^2\cap L^\infty}$ for any $t>0$. For the equation $\eqref{apr.INS-2dNS}_2$, noting that (similarly as \eqref{est.wL2.d})
\begin{equation*}\label{eq.wL2}
\begin{split}
 \int_0^t \int_{\RR^2} \big( -a^n \partial_{\tau} w^{n+1}  - a^n ( u^{n-1} \cdot\nabla w^{n+1})\big) \cdot w^{n+1} \dd x\dd\tau = -\frac 12\int_{\RR^2}a^n|w^{n+1}|^2\,\dd x
\end{split}
\end{equation*}
and along similar lines as driving \eqref{eat.wL2.3} or \eqref{eaw1.L2}, we infer that
\begin{align}\label{eawn.L2}
  & \|w^{n+1}(t)\|_{L^2}^2 + \int_0^t\|\Lambda^\alpha w^{n+1}\|_{L^2}^2\mathrm d\tau \nonumber\\
  \leq & C \|a_0\|_{L^2\cap L^\infty}^2 \|u_0\|_{H^1}^2 +  C(\|\rho_0\|_{L^\infty}^2+1) \int_0^t \|\Lambda^\alpha \bar{u}\|_{H^1}^2 \|w^{n+1}\|_{L^2}^2\mathrm d \tau ,
\end{align}
from which and the Gronwall inequality we concludes that \eqref{eswk.L2} holds for $k=n+1$.

Applying Proposition \ref{pro.Max-Regularity} to equation $\eqref{apr.INS-2dNS}_2$ leads to that
\begin{equation*}\label{eswn.M}
\begin{split}
  E_p(w^{n+1}) & = \|w^{n+1}\|_{L^\infty_t(\dot B^\alpha _{p,2})} + \|(\partial_\tau w^{n+1},\Lambda^{2\alpha} w^{n+1},\nabla  p^{n+1})\|_{L^2_t(L^p)} \\
  & \leq  C \sum_{i=1}^5\|F^n_i\|_{L^2_t(L^p)}  + C \| u^n\cdot\nabla w^{n+1}\|_{L^2_t(L^p)},
\end{split}
\end{equation*}
where  $F^n_i, i=1,...,5$ are defined by \eqref{eq.Fn}.
By using the induction assumptions \eqref{eswk.L2}-\eqref{eswk.M} and estimating as  \eqref{es.F3}--\eqref{est.F_4},
 the terms $F^{n,2}_3$ and $F^{n,2}_6$ can be by
\begin{align}
   & \|F^n_3 \|_{L^2_t (L^p)} + \| u^n\cdot\nabla w^{n+1} \|_{L^2_t(L^p)} \nonumber \\
  \leq &(1+\|a^n\|_{L^\infty_t(L^\infty)}) \big( \|(w^{n-1},w^{n})\|_{L^\infty_t(L^\infty)}\|\nabla w^{n+1}\|_{L^2_t(L^p)} +  \|\bar{u}\|_{L^\infty_t(L^{\frac{p(p+\delta_2)}{\delta_2}})} \|\nabla w^{n+1}\|_{L^2_t(L^{p+\delta_2})}\big) \nonumber \\
  \leq  & C  \|(w^{n-1},w^{n})\|_{L^\infty_t(L^2)}^{\theta_3} \|(w^{n-1},w^{n})\|_{L^\infty_t(\dot B^\alpha _{p,2})}^{1-\theta_3}
  \|\Lambda^\alpha  w^{n+1}\|_{L^2_t(L^2)}^{\theta_4} \|\Lambda^{2\alpha} w^{n+1}\|_{L^2_t(L^p)}^{1-\theta_4} \nonumber \\
  & + C \|\bar{u}\|_{L^\infty_t (H^1)} \|\Lambda^\alpha  w^{n+1}\|_{L^2_t(L^2)}^{\theta_2}  \|\Lambda^{2\alpha} w^{n+1}\|_{L^2_t(L^p)}^{1-\theta_2}  \nonumber \\
  \leq & C \|(w^{n-1},w^{n})\|_{L^\infty_t(L^2)}^{\frac{\alpha-2/p}{2\alpha-1}} \|(w^{n-1},w^{n})\|_{L^\infty_t(\dot B^\alpha _{p,2})}^{\frac{1}{2\alpha-1}}
  \|\Lambda^\alpha  w^{n+1}\|_{L^2_t(L^2)} + \frac{1}{4C} \|\Lambda^{2\alpha} w^{n+1}\|_{L^2_t(L^p)} \nonumber \\
  & + C \|u_0\|_{H^1}^{\frac{\alpha+1}{2\alpha-1}} \|\Lambda^\alpha w^{n+1}\|_{L^2_t(L^2)}  \nonumber  \\
  \leq  &C \widetilde{C}^{1+\frac{1}{\theta_4}} \|a_0\|_{L^2\cap L^\infty}^{1+\frac{1}{\theta_4}}  e^{(1+\frac{1}{\theta_4}) \widetilde{C} \|u_0\|_{H^1\cap \dot B^\alpha_{p,2}}^2 }
  + C \|a_0\|_{L^2\cap L^\infty} e^{C \|u_0\|_{H^1}^2} + \frac{1}{4C} \|\Lambda^{2\alpha} w^{n+1}\|_{L^2_t(L^p)} , \nonumber
\end{align}
where $\theta_2=\frac{2\alpha-1}{\alpha+1}$, $\theta_3 = \frac{\alpha-2/p}{1+\alpha -2/p}$ and $\theta_4 = \frac{2\alpha-1}{1+\alpha-2/p}$.
For the remaining terms $F^n_i (i=1,2,4,5)$, similarly as deriving \eqref{est.F_5} and \eqref{est.F_6}, we get that
\begin{align*}
  \|(F^n_1,F^n_2,F^n_4,F^n_5)\|_{L^2_t (L^p)} & \leq \|a^n\|_{L^\infty} \big(\|\partial_\tau w^{n+1}\|_{L^2_t (L^p)} +  \|\partial_\tau \bar{u}\|_{L^2_t (L^p)} + \|\bar{u}\cdot \nabla \bar{u}\|_{L^2_t(L^p)}  \big) \\
  & \quad + (1+ \|a^n\|_{L^\infty_t (L^\infty)})   \|w^{n+1}\|_{L^\infty_t(L^\infty)} \|\nabla  \bar{u}\|_{L^2_t(L^p)} \\
  & \leq C \|a_0\|_{L^2\cap L^\infty} e^{C \|u_0\|_{H^1}^2} +  \frac{1}{4C} \big(\| w^{n+1}\|_{L^\infty_t(\dot B^\alpha_{p,2})} + \|\partial_\tau w^{n+1}\|_{L^2_t (L^p)} \big).
\end{align*}
Collecting the above estimates yields that
\begin{align*}
  E_p(w^{n+1}) \leq C \widetilde{C}^{1+\frac{1}{\theta_4}} \|a_0\|_{L^2\cap L^\infty}^{1+\frac{1}{\theta_4}}  e^{(1+\frac{1}{\theta_4}) \widetilde{C} \|u_0\|_{H^1\cap \dot B^\alpha_{p,2}}^2 }
  + C \|a_0\|_{L^2\cap L^\infty} e^{C \|u_0\|_{H^1}^2},
\end{align*}
then by letting constant $c_0$ in \eqref{con.Max.B0} small enough, it guarantees \eqref{eswk.M} for $k=n+1$.

Hence, the induction method ensures that the uniform estimates of \eqref{eswk.L2} and \eqref{eswk.M} indeed hold for any $k\in\NN$ under the smallness condition  \eqref{con.Max.B0}.

In addition, taking the gradient operator $\nabla$ on both sides of equation $\eqref{apr.INS-2dNS}_1$ leads to
\begin{equation}\label{eq.naban}
\begin{split}
  \partial_t (\nabla a^{n+1}) + u^n \cdot\nabla (\nabla a^{n+1}) + (\nabla u^n)\cdot \nabla a^{n+1} = 0,\quad \nabla a^{n+1}|_{t=0} = \nabla a_0,
\end{split}
\end{equation}
and by making the scalar product of equation \eqref{eq.naban} with $|\nabla a^{n+1}|^{\frac{2}{\alpha}-1} \nabla a^{n+1}$ we obtain
\begin{align*}
   \frac{\dd}{\dd t} \|\nabla a^{n+1}\|_{L^{\frac{2}{\alpha}}} \leq \|\nabla u^n\|_{L^\infty} \|\nabla a^{n+1}\|_{L^{\frac{2}{\alpha}}}  ,
\end{align*}
which combined with the Gronwall inequality implies that for every $n\in\NN$,
\begin{equation}\label{naban}
\begin{split}
  \|\nabla a^{n+1}\|_{L^\infty_t(L^{\frac 2\alpha})} \leq
  \|\nabla a_0\|_{L^{\frac 2\alpha }} \exp\Big\{ \int_0^t \|\nabla u^n \|_{L^\infty} \mathrm d\tau\Big\}
  \leq \|\nabla a_0\|_{L^{\frac 2\alpha }} e^{C(u_0)\sqrt{t}} .
\end{split}
\end{equation}

\textit{Step 3: convergence of approximate solutions.}

We show that $\{(a^n , w^n)\}_{n\in\NN}$ is a Cauchy sequence in the $L^2$-energy space on a small interval $[0,t_*]$ with $t_* >0.$
Denote by
\begin{align*}
  \delta a^n := a^n-a^{n-1},\quad \delta w^n := w^n-w^{n-1},\quad \delta p^n := p^n-p^{n-1},\quad n\in\mathbb{N},
\end{align*}
with the convention
$a^{-1}=w^{-1}=p^{-1}=u^{-1}=0$. Then $(\delta a^{n+1}, \delta w^{n+1})$ fulfills
\begin{equation}\label{apr.del}
\left\{\begin{array}{ll}
  \partial_t \delta a^{n+1}+u^n\cdot\nabla \delta a^{n+1}=-\delta w^n\cdot\nabla  a^n,\\
  \partial_t \delta w^{n+1}+u^n\cdot\nabla \delta w^{n+1}+\Lambda^{2\alpha}\delta w^{n+1} +\nabla \delta  p^{n+1} =\sum\limits_{i=1}^9 \delta F^n_i,\\
  \Div  \delta w^{n+1}=0,\\
  (\delta a^{n+1},\delta w^{n+1})|_{t=0}=(0, 0),
\end{array}\right.
\end{equation}
where $u^n=w^n+\bar{u}$ and
\begin{equation*}
\begin{split}
  &\delta  F^n_1 := -\delta a^n (\partial_t w^{n+1}),\quad \delta F^n_2: = -\delta  a^n (\partial_t \bar{u}),\quad \delta F_3^n :=-a^{n-1} (\partial_t \delta w^{n+1}),\\
  &\delta  F_4^n := -\delta a^n(u^{n-1}\cdot\nabla w^{n+1}),\quad \delta  F_5^n := - a^{n-1}( \delta w^{n-1}\cdot\nabla w^{n+1}),\\
  &\delta  F_6^n := - a^{n-1}(  u^{n-2}\cdot\nabla \delta w^{n+1}),\quad \delta  F_7^n := -\delta w^n\cdot\nabla w^{n+1},\\
  &\delta  F_8^n := - \delta a^n( u^{n+1}\cdot\nabla \bar{u}),\quad \delta  F_9^n :=-(1+a^{n-1})\delta w^{n+1}\cdot\nabla\bar{ u}.
\end{split}
\end{equation*}
First, we investigate the $L^2$-estimate of $\delta a^{n+1}$. Multiplying both sides of equation $\eqref{apr.del}_1$ by $\delta a^{n+1}$
and integrating on the spatial variable yield that
\begin{align*}
  \frac{1}{2}\frac{\mathrm d}{\mathrm d t}\|\delta a^{n+1}(t)\|_{L^2}^2 \leq & \int_{\RR^2}|\delta w^n\cdot\nabla a^n(x,t)||\delta a^{n+1}(x,t)|\mathrm dx\nonumber\\
  \leq &\|\delta w^n(t)\|_{L^{\frac 2{1-\alpha}}} \|\nabla a^n(t)\|_{L^{\frac 2\alpha }} \|\delta a^{n+1}(t)\|_{L^2}\nonumber\\
  \leq & C \|\delta w^n(t)\|_{\dot H^\alpha} \|\nabla a^n(t)\|_{L^{\frac 2\alpha }} \|\delta a^{n+1}(t)\|_{L^2},
\end{align*}
which implies
\begin{align}\label{est.an.del}
  \frac{\mathrm d}{\mathrm d t}\|\delta a^{n+1}(t)\|_{L^2} \leq C \|\delta w^n(t)\|_{\dot H^\alpha} \|\nabla a^n(t)\|_{L^{\frac 2\alpha }} .
\end{align}
By integrating on the time interval $[0,t]$ and using \eqref{naban}, we obtain
\begin{equation}\label{est.an.del1}
\begin{split}
  \|\delta a^{n+1}\|_{L^\infty_t(L^2)} \leq C \|\nabla a^n\|_{L^2_t(L^{\frac{2}{\alpha}})} \|\delta w^n\|_{L^2_t(\dot H^\alpha)}
  \leq C \sqrt{t} e^{C \sqrt{t}} \|\nabla a_0\|_{L^{\frac{2}{\alpha} }}
  \|\delta w^n\|_{L^2_t(\dot H^\alpha)}.
\end{split}
\end{equation}

Now we intend to bound the $L^2$-norm of $\delta w^{n+1}$. Taking the inner product of equation $\eqref{apr.del}_2$ with $\delta w^{n+1}$ gives that
\begin{equation}\label{est.wn.del}
\begin{split}
  \frac 12\frac{\mathrm d}{\mathrm d t}\|\delta w^{n+1}(t)\|_{L^2}^2 + \|\delta w^{n+1}(t)\|_{\dot H^\alpha}^2
  = \sum_{i=1}^9 \int_{\RR^2}\delta F^n_i(x,t)\cdot\delta w^{n+1}(x,t)\mathrm d x.
\end{split}
\end{equation}
Utilizing H\"older inequality,  interpolation inequality and Young inequality, we get
\begin{align}\label{del.F1n}
  \Big|\int_{\RR^2}\delta F^n_1\cdot\delta w^{n+1}\mathrm dx\Big|
  \leq & \|\delta a^n\|_{L^2} \| \partial_t w^{n+1}\|_{L^p} \|\delta w^{n+1}\|_{L^{\frac {2p}{p-2}}}\nonumber\\
  \leq & C  \|\delta a^n\|_{L^2}\| \partial_t w^{n+1}\|_{L^p} \|\delta w^{n+1}\|_{L^2}^{\frac{\alpha p-2}{\alpha p}} \| \delta w^{n+1}\|_{\dot H^\alpha}^{\frac 2{\alpha p}}\nonumber\\
  \leq & C\|\delta a^n\|_{L^2}^{\frac{\alpha p}{\alpha p-1}} \|\partial_t w^{n+1}\|_{L^p}^{\frac{\alpha p}{\alpha p-1}}
  \|\delta w^{n+1}\|_{L^2}^{\frac{\alpha p-2}{\alpha p-1}} + \frac{1}{8}\|\delta w^{n+1}\|_{\dot H^\alpha}^2 \nonumber \\
  \leq &C \|\partial_t w^{n+1}\|_{L^p} \|\delta a^n\|_{L^2}^2 + \|\partial_t w^{n+1}\|_{L^p} \|\delta w^{n+1}\|_{L^2}^2 + \frac{1}{8} \|\delta w^{n+1}\|_{\dot H^\alpha}^2,
\end{align}
and similarly,
\begin{align}\label{del.F2n}
  \Big|\int_{\RR^2}\delta F^n_2\cdot\delta w^{n+1}\mathrm dx\Big| & \leq \|\delta a^n\|_{L^2} \|\partial_t \bar{u}\|_{L^p} \|\delta w^{n+1}\|_{L^{\frac{2p}{p-2}}} \nonumber \\
  & \leq C \| \partial_t\bar{u}\|_{L^p} \|\delta a^n\|_{L^2}^2 + \|\partial_t\bar{u}\|_{L^p} \|\delta w^{n+1}\|_{L^2}^2 + \frac{1}{8}\|\Lambda^\alpha \delta w^{n+1}\|_{L^2}^2.
\end{align}
Making use of the relation $\partial_ta^{n-1}+\Div(a^{n-1}u^{n-2})=0$, and along the same lines as deducing \eqref{est.wL2.d}, we see that
\begin{equation}\label{del.F3nFn6}
\begin{split}
  \int_{\RR^2}\big(\delta F^n_3 +\delta F^n_6\big)\cdot\delta w^{n+1}\mathrm dx
  =-\frac 12\frac {\mathrm d}{\mathrm dt}\int_{\RR^2}a^{n-1}|\delta w^{n+1}|^2\,\mathrm d x.
\end{split}
\end{equation}
For the remaining terms, one can easily find that
\begin{align}
  \Big|\int_{\RR^2}\delta F^n_4\cdot\delta w^{n+1}\mathrm dx\Big| 
  & \leq \|u^{n-1}\|_{L^\infty} \|\nabla w^{n-1}\|_{L^\infty} \big( \|\delta a^n\|_{L^2}^2  + \|\delta w^{n+1}\|_{L^2}^2 \big), \label{del.F4n} \\
  \Big|\int_{\RR^2}\delta F^n_5\cdot\delta w^{n+1}\mathrm dx\Big| &\leq \|a^{n-1}\|_{L^\infty} \|\nabla w^{n+1}\|_{L^\infty} \big(\|\delta w^{n-1}\|_{L^2}^2 +  \|\delta w^{n+1}\|_{L^2}^2 \big), \label{del.F5n} \\
  \Big|\int_{\RR^2}\delta F^n_7\cdot\delta w^{n+1}\mathrm dx\Big| & \leq \|\nabla w^{n+1}\|_{L^\infty} \big( \|\delta w^n\|_{L^2}^2 + \|\delta w^{n+1}\|_{L^2}^2 \big),\label{del.F7n} \\
  \Big|\int_{\RR^2}\delta F^n_8\cdot\delta w^{n+1}\mathrm dx\Big| & \leq \|u^{n+1}\|_{L^\infty} \|\nabla \bar{u}\|_{L^\infty} \big( \|\delta a^n\|_{L^2}^2 + \|\delta w^{n+1}\|_{L^2}^2 \big),\label{del.F8n}\\
  \Big|\int_{\RR^2}\delta F^n_9\cdot\delta w^{n+1} \mathrm dx\Big| & \leq (1+\|a^{n-1}\|_{L^\infty}) \|\nabla \bar{u}\|_{L^\infty} \|\delta w^{n+1}\|_{L^2}^2. \label{del.F9n}
\end{align}
Plugging estimates $\eqref{del.F1n}$--$\eqref{del.F9n}$ into \eqref{est.wn.del}, integrating over $[0,t]$ and adding up $\eqref{est.an.del1}$, we obtain
\begin{align}\label{est.wn.del'}
  &\frac 12\|\delta w^{n+1}\|_{L^\infty_t(L^2)}^2 + \frac{1}{2} \|\delta w^{n+1}\|_{L^2_t(\dot H^\alpha)}^2 +  \|\delta a^{n+1}\|_{L^\infty_t(L^2)}^2 \nonumber \\
  \leq & \frac{1}{2}\|a_0\|_{L^\infty} \|\delta w^{n+1}\|_{L^\infty_t(L^2)}^2  + C \sqrt{t} \| (\partial_\tau w^{n+1},\partial_\tau\bar{u})\|_{L^2_t(L^p)}
  \| (\delta a^n, \delta w^{n+1}) \|_{L^\infty_t(L^2)}^2 \nonumber \\
  & + C \sqrt{t} \big(1 + \|(u^{n-1},u^{n+1})\|_{L^\infty_t(L^\infty)} \big) \| (\nabla w^{n+1},\nabla\bar {u})\|_{L^2_t(L^\infty)}
  \|(\delta a^n,\delta w^{n+1})\|_{L^\infty_t(L^2)}^2  \nonumber \\
  & + C \sqrt{t} \|\nabla w^{n+1}\|_{L^2_t(L^\infty)}  \|(\delta w^n, \delta w^{n-1})\|_{L^\infty_t(L^2)}^2 + C t e^{C \sqrt{t}} \|\nabla a_0\|_{L^{\frac{2}{\alpha} }}^2 \|\delta w^n\|_{L^2_t(\dot H^\alpha)}^2,
\end{align}
where we have used the uniform bound that $\|a^n\|_{L^\infty} \leq \|a_0\|_{L^\infty_t(L^\infty)}\leq \frac{1}{2}$.
Denote by
\begin{align*}
  I_n(t) \stackrel{\mathrm{def}}{=}\|\delta a^n \|_{L^\infty_t(L^2)}^2 + \|\delta w^n\|_{L^\infty_t(L^2)}^2 + \|\Lambda^\alpha \delta w^n\|_{L^2_t(L^2)}^2.
\end{align*}
Since $\|(\partial_t w^n,\partial_t \bar{u},\nabla w^n\|_{L^2_T(L^p)}$, $\|\nabla\bar {u})\|_{L^2_T(L^p)}$ and $\|u^n\|_{L^\infty_T(L^\infty)}$ are uniformly-in-$n$ bounded from the above step,
by letting $\|a_0\|_{L^\infty}\leq \frac{1}{2}$ and $t$ small enough,
there exists a generic constant $t_*>0$ depending only on $\alpha,p$ and the initial data $(\rho_0,u_0)$ such that
\begin{equation}\label{est.In}
\begin{split}
  I_{n+1}(t_*)\leq \frac{1}{4} I_n(t_*) + \frac{1}{8} I_{n-1}(t_*).
\end{split}
\end{equation}
By iteration it follows from \eqref{est.In} that
\begin{equation}\label{est.In'}
\begin{split}
  I_{n+1}(t_*)+\frac 14I_{n}(t_*)
  \leq \frac 12\big(I_{n}(t_*)+\frac 14I_{n-1}(t_*)\big)
  \leq \cdot\cdot\cdot\leq \frac 1{2^n}\big(I_{1}(t_*)+\frac 14I_{0}(t_*)\big)\leq \frac{C}{2^n}.
\end{split}
\end{equation}

From \eqref{est.In'}, we know that $\{(a^n, w^n)\}_{n\in\NN}$ is a Cauchy sequence in  $L^\infty(0, t_*; L^2(\RR^2))$.
Hence, there exists a limit function $(a, w)$ such that
\begin{align*}
  (a_n,w_n)\rightarrow (a,w)\,\, \text{in}\,\, L^\infty(0, t_*; L^2(\RR^2))\times \big(L^\infty(0, t_*; L^2(\RR^2))\cap L^2(0, t_*; \dot H^\alpha (\RR^2))\big).
\end{align*}
By virtue of the uniform estimates \eqref{eswk.L2}, \eqref{eswk.M} and the interpolation, one can deduce that $\{ w^n\}_{n\in\NN}$ strongly converges to $w$
in the space $L^\infty(0, t_*; L^2\cap \dot B^{s_1}_{p.2}(\RR^2))$ with any $s_1<\alpha$ and $\{ a^n\}_{n\in\NN}$ converges strongly to $a$ in the space
$L^\infty(0, t_*; W^{s_2,\frac 2\alpha }(\RR^2))$ with any $s_2<1$. Owing to Fatou's lemma, we also get $a\in L^\infty(0, t_*;  W^{1,\frac 2\alpha }(\RR^2))$ and $w\in L^\infty(0,t_*; \dot B^\alpha _{p,2}(\RR^2))$.
Therefore, we can pass the limit $n\rightarrow\infty$ in the system \eqref{eq.INS-2dNS} to deduce that $(a, w)$ satisfies the perturbed
system \eqref{eq.INS-2dNS} in the distributional sense.

\textit{Step 4: uniqueness.}

Let $(a_i, w_i), i=1,2$ be two solutions of the perturbed system \eqref{eq.INS-2dNS} associated with the same initial data $(a_0,0)$ and the smallness condition \eqref{con.Max.B0}, which satisfy that for any $T\in(0,\infty]$ and $ i=1,2$,
\begin{align*}
  a_i\in L^\infty_T(L^2\cap L^\infty\cap \dot W^{1,\frac 2{\alpha}}),\quad w_i\in L^\infty_T(L^2\cap \dot B^\alpha_{p,2})\times L^2_T(\dot H^\alpha\cap \dot W^{2\alpha,p}), \quad \partial_t w_i \in L^2_T(L^p).
\end{align*}
Denote by
$\delta a := a_1-a_2$, $\delta w := w_1-w_2$, $\delta p := p_1-p_2$,
then $(\delta a, \delta w)$ fulfills
\begin{equation}\label{eq.d}
\left\{\begin{array}{ll}
  \partial_t \delta a+u_1\cdot\nabla \delta a=-\delta w\cdot\nabla  a_2 , \\
  \partial_t \delta w+u_1\cdot\nabla \delta w+\Lambda^{2\alpha}\delta w +\nabla \delta  p =\sum\limits_{i=1}^6 \delta F_i,\\
  \Div  \delta w=0,\\
  (\delta a,\delta w)|_{t=0}=(0, 0),
\end{array}\right.
\end{equation}
where  $u_i=w_i+\bar{u}, i=1,2$ and
\begin{equation*}
\begin{split}
  &\delta  F_1 := -\delta a (\partial_t w_1),\quad \delta F_2: = -\delta  a(\partial_t \bar{u}+\bar u\cdot\nabla\bar u),\quad \delta F_3 :=-a_2 (\partial_t \delta w),\\
  &\delta  F_4 := -\delta a(w_1\cdot\nabla \bar u),\quad \delta  F_5 := -(1+ a_2)( \delta w\cdot\nabla \bar u),\quad
  \delta  F_6 := - a_2 (u_2\cdot\nabla \delta w), \\
  & \delta F_7 := - a_2\delta w\cdot\nabla w_2, \quad \delta F_8 := -\delta a(u_1\cdot \nabla w_1).
\end{split}
\end{equation*}
By performing the $L^2$-energy estimates and arguing as obtaining \eqref{est.an.del1} and \eqref{est.wn.del'}, we infer that
\begin{equation*}
\begin{split}
  \|\delta a\|_{L^\infty_t(L^2)}  \leq C \sqrt{t} \|\nabla a_2\|_{L^\infty_T (L^{\frac{2}{\alpha}} )}
  \|\delta w\|_{L^2_t(\dot H^\alpha)}
\end{split}
\end{equation*}
and
\begin{align}
  &\frac 12\|\delta w\|_{L^\infty_t(L^2)}^2 + \|\delta w\|_{L^2_t(\dot H^\alpha)}^2 \nonumber\\
  \leq &  \frac{1}{2} \|a_0\|_{L^\infty} \|\delta w\|_{L^\infty_t(L^2)}^2 + \frac{1}{2} \|\delta w\|_{L^2_t (\dot H^\alpha)}^2 \nonumber \\
  & + C \sqrt{t} \Big(\| (\partial_t w_2,\partial_t\bar{u})\|_{L^2_T(L^p)}+ \big( 1+ \| (w_2,\bar u)\|_{L^\infty_T(L^\infty)}\big) \| (\nabla\bar u,\nabla w_2)\|_{L^2_T( L^\infty)}\Big)
  \|(\delta w,\delta a)\|_{L^\infty_t(L^2)}^2. \nonumber
\end{align}
Using the uniform bounds for $\bar u, w_i, a_i$, we deduce that there exists a small time $T_1>0$ so that
\begin{align}\label{es.dw'}
 \|\delta w\|_{L^\infty_{T_1}(L^2)}^2 + \|\delta w\|_{L^2_{T_1}(\dot H^\alpha)}^2 + \|\delta a\|_{L^\infty_{T_1}(L^2)}^2 \leq 0,
\end{align}
which implies $(\delta w, \delta a)\equiv 0$ on $\RR^2\times [0,T_1]$. Repeating the above procedure on the time intervals $[T_1,2T_1], [2T_1,3T_1],... $ concludes $(\delta w, \delta a)\equiv 0$
on $\RR^2\times [0,T]$ and it completes the uniqueness of the constructed solution $(\rho,u)$.

\textit{Step 5: the maximal existence time $T_*=\infty$.}

Let $T_*>0$ be the maximal existence time of strong solution $(a, w)$ solving the perturbed system \eqref{eq.INS-2dNS} and fulfilling that $a\in L^\infty([0,T_*); L^\infty\cap L^2\cap \dot W^{1,\frac 2{\alpha}})$ and $w\in L^\infty([0,T_*);  L^2\cap \dot B^{\alpha}_{p,2})\times L^2([0,T_*);\dot H^\alpha\cap \dot W^{2\alpha,p})$, $\partial_t w\in L^2([0,T_*); L^p)$.
Under the smallness condition \eqref{con.Max.B0}, we will use a bootstrapping argument to show $T_*=\infty$.

Suppose that $T_*<\infty$. Since under the smallness condition \ref{con.Max.B0} $(a,w)$ is now regular enough on $(0,T_*)$
to ensure the \textit{a priori} estimates in Propositions \ref{prop:w} and \ref{pro.propri},
we infer that
\begin{equation}\label{es.wT}
\begin{split}
  \sup_{t\in[0,T_*)}\|w(t)\|_{L^2\cap\dot B^\alpha_{p,2}} + \|w\|_{L^2([0,T_*); \dot H^\alpha \cap \dot W^{2\alpha,p})}
  \leq C\|a_0\|_{L^2\cap L^\infty} e^{C \|u_0\|_{H^1\cap \dot B^\alpha_{p,2}}^2 } \leq 1,
\end{split}
\end{equation}
and
\begin{align}\label{es.aT}
  \sup_{t\in[0,T_*)}\|a(t)\|_{L^2\cap L^\infty}\leq \|a_0\|_{L^2\cap L^\infty},\quad
  \|\nabla a\|_{L^\infty([0,T_*);L^{\frac{2}{\alpha}})}
  \leq \|\nabla a_0\|_{L^{\frac{2}{\alpha}}}e^{ C(1+ \|u_0\|_{H^1\cap \dot B^\alpha_{p,2}}^{\frac{4\alpha-1}{2\alpha-1}}) \sqrt{T_*} },
\end{align}
with $C =C(\alpha,p)>0$.
Hence we can repeat the above procedure in Steps 1 - 3 starting from some time $t_0<T_*$ that can be close to $T_*$ arbitrarily,
and due to that the proceeding time $t_*$ is depending only on $\alpha, p, T_*,$\,$ \sup_{t\in[0,T_*)}\|w(t)\|_{L^2\cap {\dot B^\alpha _{p,2}}},$
$ \sup_{t\in[0,T_*)}\|a(t)\|_{L^2\cap L^\infty\cap \dot W^{1,2/\alpha}}$, which by \eqref{es.wT}-\eqref{es.aT} in turn implies that $t_*$ is depending only on $\alpha, p, T_*$
and the norms of $( a_0, u_0)$, we conclude that the maximal time $T_*$ can be exceeded. This is a contradiction,
thus we get $T_*=\infty$ and Proposition \ref{Main.pro} is completed.
\end{proof}

\subsection{Global existence of solution for 2D fractional INS system with rough density.}\label{subsec:exi2}
Owing to the low regularity of $a_0=\rho_0-1$ (now we do not assume $\nabla a_0\in L^{\frac 2\alpha}$ any more),
we can not prove the convergence of the approximate sequences in the $L^2$-topology as in the previous subsection.
Thus we shall use compactness arguments instead.
For completeness, we outline the proof as follows.

For every $\epsilon>0$, let $\chi_\epsilon(\cdot)= \epsilon^{-2} \chi(\frac{\cdot}\epsilon)$ and $\chi\in C^\infty_c(\RR^2)$ be the standard mollifier.
Let $\rho^\epsilon_0=\chi_\epsilon*\rho_0$, then it satisfies $\nabla \rho^\epsilon_0\in L^{\frac 2\alpha }(\RR^2)$.
According to Proposition \ref{Main.pro}, the perturbed system \eqref{eq.frc-INNS} with initial data $(\rho^\epsilon_0, u_0)$ admits a unique global-in-time strong solution $(\rho^\epsilon,\,u^\epsilon)$
satisfying the uniform-in-$\epsilon$ bounds \eqref{est.Max.w3-0}-\eqref{est.Max.w3}. 
Thus we are allowed to pick a subsequence $\epsilon_k$
($\epsilon_k\rightarrow 0$ as $k\rightarrow \infty$) such that
\begin{equation}\label{lim.a}
\begin{split}
  \rho^{\epsilon_k}-1\rightharpoonup^*\rho-1\quad\textrm{in}\,\,\, L^\infty(\RR_+; L^2\cap L^\infty(\RR^2)),\quad\quad\quad\quad
\end{split}
\end{equation}
and
\begin{equation}\label{lim.u}
\begin{split}
  &u^{\epsilon_k} - \bar{u}\rightharpoonup^* u-\bar{u}\quad \textrm{in}\,\,\,L^\infty(\RR_+; L^2\cap \dot B^\alpha _{p.2}(\RR^2))\cap L^2(\RR_+;\dot W^{2\alpha,p}(\RR^2)),\\
  &\partial_t u^{\epsilon_k}- \partial_t \bar{u}\rightharpoonup \partial_t u- \partial_t \bar{u}\quad\textrm{in}\,\,\, L^2(\RR_+;L^p(\RR^2)).
\end{split}
\end{equation}
In addition, utilizing the diagonal argument together with the Rellich type theorems applied for
the compact (space-time) subsets of $\RR^2\times\RR_+ $, we conclude
\begin{equation}\label{limit.u}
\begin{split}
  u^{\epsilon_k} \rightarrow u\quad\textrm{a.e.\,\,pointwisely\,\,in}\,\,\; \RR_+ \times\RR^2.
\end{split}
\end{equation}
In view of
\begin{align*}
  \rho^{\epsilon_k} \partial_t u^{\epsilon_k} + \rho^{\epsilon_k} (u^{\epsilon_k} \cdot\nabla u^{\epsilon_k}) =
  \partial_t(\rho^{\epsilon_k} u^{\epsilon_k})+\Div(\rho^{\epsilon_k} u^{\epsilon_k} \otimes u^{\epsilon_k}),
\end{align*}
the above convergence is sufficient to pass to the limit in \eqref{eq.frc-INNS} in the distributional sense and hence
$(\rho, u)$ is indeed a distributional solution to the system \eqref{eq.frc-INNS}. By Fatou's lemma, the solution $(\rho,u)$ is also regular enough and satisfies estimates \eqref{est.Max.w3-0}-\eqref{est.Max.w3}.
Therefore, the existence part of Theorem \ref{Main.thm} is proved.

\section{Proof of Theorem \ref{Main.thm}: the uniqueness part}\label{sec:UNI}

This section is devoted to proving the uniqueness of constructed solutions in Theorem \ref{Main.thm}.

Because of the hyperbolic nature of the coupled system \eqref{eq.frc-INNS} and the low-regularity of density,
the Eulerian framework used in the uniqueness proof of Proposition \ref{Main.pro} seems not effective,
and we shall employ the Lagrangian approach as in \cite{DanchMucha12,DanM13,PaicuZhZh13} to tackle with the uniqueness issue.
Inspired by \cite{DanM13}, we intend to show the uniqueness by establishing the $L^\infty_T(\dot H^\alpha)\cap L^2_T(\dot H^{2\alpha})$-estimates of $\delta v$,
which is the difference of two velocity fields in the Lagrangian coordinates
(it seems almost impossible to prove the uniqueness in the usual $L^\infty_T(L^2)\cap L^2_T(\dot H^\alpha)$ framework due to that one can not control the term $\nabla \delta v$ on the right-hand side).
We write the system of $\delta v$ as the twisted fractional Stokes system \eqref{eq.delta.v} and we derive the crucial $L^2_T(L^2)$ maximal regularity estimate \eqref{est.sta} on a small time interval.
Meanwhile, some right-hand terms in \eqref{est.sta} arising from the nonlocal dissipation seem hard to be controlled using the (natural) quantity $\|v_i\|_{L^2_T(\dot W^{2\alpha,p})}$,
instead we have to adopt $\|v_i\|_{L^2_T(\dot B^{2\alpha}_{p,1})}$ as the bound, which in turn need the stronger regularity $u_i \in L^2_T(\dot B^{2\alpha}_{p,1})$
obtained in Proposition \ref{prop:uni.M.bet}.

In order to derive the 2D fractional INS system \eqref{eq.frc-INNS} in the Lagrangian coordinates,
we firstly introduce some basic knowledge related to it.
The particle trajectory $X_t(\cdot)$ associated with the velocity $u$ is defined by the ordinary differential equation
\begin{equation}\label{eq.flow1}
  \frac{\dd X_t(y)}{\dd t} = u (t, X_t(y)),\quad X_t(y)|_{t=0}= y,
\end{equation}
that is,
\begin{equation}\label{eq.flow}
  X_t(y)=y+\int_0^t\,u(\tau,X_\tau(y))\,\mathrm d\tau,\quad y\in \RR^2,
\end{equation}
which maps the Lagrangian coordinate $y$ to the Eulerian coordinate $x = X_t(y)$.
According to \eqref{est.nabla-u'}, we know that $u\in L^1(\RR_+; W^{1,\infty}(\RR^2))$, and equation \eqref{eq.flow1}
admits a unique solution $X_t(\cdot) : \RR^2 \rightarrow \RR^2$ on $[0,\infty)$ which is a measure-preserving bi-Lipschitzian homeomorphism satisfying
$ X_t^{\pm 1}\in L^\infty(\RR_+; W^{1,\infty}(\RR^2))$. Note that $X^{-1}_t(\cdot):\RR^2\rightarrow \RR^2$ solves that
\begin{align}\label{eq:X-invers}
  X_t^{-1}(x) = x - \int_0^t u(t-\tau, X_\tau^{-1}(x)) \dd \tau.
\end{align}
By letting $0 < T_1 <1$ be small enough, we can have
\begin{equation}\label{Lip.inverse}
\begin{split}
  \int_0^{T_1} \|\nabla_x u\|_{L^{\infty}}\mathrm d t \leq \frac{1}{2} .
\end{split}
\end{equation}

Denote by
\begin{align}\label{eq.L.vari}
  \eta(t,y) := \rho (t,X_t(y)),\quad  v(t,y) := u(t,X_t(y)),\quad \Pi(t,y) := \pi(t,X_t(y)) .
\end{align}
In terms of the above notations, system \eqref{eq.frc-INNS} can be expressed as follows
\begin{equation}\label{eq.Lagr.v}
\left\{\begin{array}{ll}
  \partial_t \eta=0,\\
  \eta\partial_t v+\Lambda^{2\alpha}_v v+\nabla_v \Pi=0,\\
  \Div_v v=0,\\
  (\eta,v)|_{t=0}(y) = (\rho_{0}(y),u_{0}(y)),
\end{array}\right.
\end{equation}
where $\Lambda^{2\alpha}_v v (t,y):= \Lambda^{2\alpha} u(t,x) = (\Lambda^{2\alpha}{u})(t,X_t(y))$, $\nabla_v \Pi(t,y):= \nabla_x \pi(t,x) = (\nabla_x \pi)(t, X_t(y))$, $\Div_v v(t,y) := \Div_x u(t,x) = (\Div_x u)(t, X_t(y))  $.
We set
\begin{align}\label{eqA}
  A (t,y):=(\nabla_y X_t )^{-1}(y) = \nabla_x X_t^{-1}(x),
\end{align}
and $A^t$ the transpose matrix of $A $, then by the chain rule, some elementary calculation gives that (e.g. see \cite[Eq. (35)]{DanM13})
\begin{equation}\label{est.notation-Lagr}
\begin{split}
  \nabla_v \Pi = A^t \nabla_y \Pi,\qquad \Div_v v= \Div_y(A  v) = A^t : \nabla v.\quad 
\end{split}
\end{equation}

Since $X_t(\cdot)$ is a measure-preserving mapping, according to Lemma \ref{lem:fra-op-exp}, we find that
\begin{align}\label{Lam-2alp-lag}
 ( \Lambda^{2\alpha} u )(t, X_t(y)) & =- \frac{\nabla\cdot}{\Lambda^{2-2\alpha}} \nabla u(t,X_t(y)) \nonumber  \\
  & = - c_\alpha \, \int_{\RR^2} \frac{(X_t(y)-z) \cdot (\nabla u(t, X_t(y)) - \nabla u(t,z)) )}{|X_t(y) - z|^{2+ 2\alpha}} \dd z \nonumber \\
  & = - c_\alpha \,  \mathrm{p.v.}\int_{\RR^2} \frac{(X_t(y)- X_t(\tilde{z})) \cdot (\nabla u(t, X_t(y)) - \nabla u(t,X_t(\tilde{z})) }{|X_t(y) - X_t(\tilde{z})|^{2+ 2\alpha}} \dd \tilde{z} \nonumber \\
  & = - c_\alpha \, \mathrm{p.v.}\int_{\RR^2} \frac{(X_t(y)- X_t({z}))  \cdot \big(A^t(t,y) \nabla v(t,y) - A^t(t,z) \nabla v(t,z)\big) }{|X_t(y) - X_t(z)|^{2+ 2\alpha}} \dd z \nonumber \\
  & =: \Lambda^{2\alpha}_v v(t,y) ,
\end{align}
where  
$c_\alpha = \frac{\alpha 4^\alpha \Gamma(\alpha)}{2 \pi \Gamma(1-\alpha)}$ and $\Gamma(s)$ is the Gamma function.

Now let $(\rho_i,u_i, \pi_i), i=1,2$ be two solutions of system \eqref{eq.frc-INNS} with the same initial data $(\rho_0, u_0)$.
Denote by
\begin{align}\label{eq:eta-v}
  \eta_i(t,y) := \rho_i (t,X_{i,t}(y)),\quad  v_i(t,y) := u_i(t,X_{i,t}(y)),\quad \Pi_i(t,y) := \pi_i(t,X_{i,t}(y)),
\end{align}
where $X_{i,t}(y)$ is the particle trajectory generated by velocity $u_i$:
\begin{align}\label{eq:Xit}
  X_{i,t}(y) = y + \int_0^t\,u_i(\tau,X_{i,\tau}(y))\,\mathrm d\tau = y + \int_0^t v_i(\tau,y)\dd \tau,\quad i=1,2.
\end{align}
Thanks to Propositions \ref{pro.propri} and \ref{prop:uni.M.bet}, we have the following estimates for solutions in Lagrangian coordinates.
\begin{proposition}\label{pro.propri.Lag}
 Let $\frac 12<\alpha<1,\, p>\frac 2{2\alpha-1}$, $u_0\in H^1\cap\dot B^{\alpha+s}_{p,2}(\RR^2)$, $s\in (0,1)$, and $\rho_0 -1 \in L^2\cap L^\infty \cap {{\mathcal{M}(\dot B^s_{p,2})}\cap  {\mathcal{M}(\dot B^{\frac 2p+1-2\alpha}_{p,1})}} $
satisfying conditions \eqref{con.Max.B0} and \eqref{eq:rho-cd}. Then for $i=1,2$ we have
\begin{equation}\label{est.0L}
  \|v_i\|_{L^\infty(\RR_+;L^2 \cap \dot B^\alpha _{p,2})} + \|(\partial_t v_i, \nabla \Pi_i)\|_{L^2(\RR_+;L^p)} + \|\Lambda^\alpha v_i\|_{L^2(\RR_+;L^2)} \leq C
\end{equation}
and there exists a constant $T_1\in (0,1]$ small enough such that
\begin{equation}\label{est.uv-lip}
  \int_0^{T_1} \|\nabla u_i(t)\|_{L^\infty} \dd t \leq \frac{1}{2}, \quad \textrm{and}\quad \int_0^{T_1} \|\nabla v_i(t)\|_{L^\infty} \dd t \leq  \frac{1}{2},
\end{equation}
and \begin{equation}\label{est.0L-2}
  \|v_i\|_{L^2(0,T_1;\dot B^{2\alpha}_{p,1})} \leq C.
\end{equation}
\end{proposition}
\begin{proof}[Proof of Proposition \ref{pro.propri.Lag}]
From \eqref{est.nabla-u'} we have $\|\nabla u_i\|_{L^1(\RR_+; L^\infty)} \leq C$, thus the estimates of $\|v_i\|_{L^\infty(\RR_+;\dot B^\alpha _{p,2})}$ and $\|\Lambda^\alpha v_i\|_{L^2(\RR_+;L^2)}$ in \eqref{est.0L}
follow immediately from Lemma \ref{lem:f.mu}.
By virtue of \eqref{eq:Xit}, we get
\begin{equation*}\label{eq.flow.na}
  \|\nabla X_{i,t}\|_{L^\infty(\RR_+; L^\infty)} \leq e^{\|\nabla u_i\|_{L^1(\RR_+; L^\infty)}} \leq C,
\end{equation*}
then Proposition \ref{pro.propri} and the direct calculation lead to
\begin{align}
  & \|\partial_t v_i\|_{L^2(\RR_+; L^p)} \leq \|\partial_t u_i\|_{L^2(\RR_+; L^p)} + \|\nabla u_i\|_{L^2(\RR_+; L^p)} \|u_{i}\|_{L^\infty (\RR_+;L^\infty)} \leq C, \nonumber\\
  & \|\nabla \Pi_i\|_{L^2(\RR_+; L^p)} \leq \|\nabla \pi_i\|_{L^2(\RR_+; L^p)} \|\nabla X_{i,t}\|_{L^\infty(\RR_+; L^\infty)} \leq C, \nonumber \\
  & \|\nabla v_i\|_{L^1 (\RR_+; L^\infty)} \leq \|\nabla u_i\|_{L^1(\RR_+; L^\infty)} \|\nabla X_{i,t}\|_{L^\infty(\RR_+;L^\infty)} \leq C. \label{est.v-lip}
\end{align}
By letting $T_1>0$ small enough, estimate \eqref{est.uv-lip} stems from \eqref{est.nabla-u'} and $\eqref{est.v-lip}$.

Next we prove estimate \eqref{est.0L-2} from \eqref{est.M.bet}.
Noticing that
\begin{align*}
  \nabla v_i(t,y) = (\nabla X_{i,t}^t(y)) \nabla u(t,X_{i,t}(y)) = (\nabla X_{i,t}^t(y)-\mathrm {Id}) \nabla u_i(t,X_{i,t}(y)) + \nabla u_i(t,X_{i,t}(y)),
\end{align*}
and $\nabla X_{i,t}(y)= \mathrm {Id} + \int_0^t \nabla v_i(\tau,y) \dd \tau$,
we find that by virtue of \eqref{est.prod2} and Propositions \ref{pro.propri},
\begin{align*}
\|\nabla v_i\|_{L^2_{T_1}(\dot B^{2\alpha-1}_{p,1})} & \leq \|(\nabla X_{i,t}^t -\mathrm {Id})\nabla u_i\circ X_{i,t}\|_{L^2_{T_1}(\dot B^{2\alpha-1}_{p,1})} + \|\nabla u_i\circ X_{i,t}\|_{L^2_{T_1}(\dot B^{2\alpha-1}_{p,1})} \\
  &\leq C \|\nabla X_{i,t}-\mathrm{Id}\|_{L^\infty_{T_1}(\dot B^{2\alpha-1}_{p,1})} \|\nabla u_i\circ X_{i,t}\|_{L^2_{T_1}(L^\infty)} \\
  & \quad + C\big(\|\nabla X_{i,t}- \mathrm {Id}\|_{L^\infty_{T_1}(L^\infty)} +1\big)\|\nabla u_i \circ X_{i,t}\|_{L^2_{T_1}(\dot B^{2\alpha-1}_{p,1})} \\
  &\leq C \sqrt{T_1} \|\nabla u_i\|_{L^2_{T_1}(L^\infty)} \|\nabla v_i\|_{L^2_{T_1}(\dot B^{2\alpha-1}_{p,1})} \\
  &\quad + C\big( \|\nabla v_i \|_{L^1_{T_1}(L^\infty)}+1\big) \|\nabla u_i\circ X_{i,t}\|_{L^2_{T_1}(\dot B^{2\alpha-1}_{p,1})}.
\end{align*}
Letting $T_1>0$ small enough, and utilizing Lemma \ref{lem:f.mu} and \eqref{est.M.bet}, we obtain
\begin{align*}
  \|\nabla v_i\|_{L^2_{T_1}(\dot B^{2\alpha-1}_{p,1})}
  \leq C \|\nabla u_i\circ X_{i,t}\|_{L^2_{T_1}(\dot B^{2\alpha-1}_{p,1})} \leq C\|\nabla u_i\|_{L^2_{T_1}(\dot B^{2\alpha-1}_{p,1})}  \leq C,
\end{align*}
which immediately leads to \eqref{est.0L-2}.
\end{proof}

In terms of notations \eqref{eq:eta-v}, the system \eqref{eq.Lagr.v} is satisfied with $(\eta_i, v_i,\Pi_i)$ in place of $(\eta,v,\Pi)$ and with the same initial data $(\rho_0,u_0)$.
The equation $\eqref{eq.Lagr.v}_1$ gives that
\begin{align}
  \eta_i(t,y) \equiv \rho_0(y),\quad \textrm{for  }i=1,2.
\end{align}
Set
\begin{equation}\label{Lag.chav12}
\begin{split}
  &\delta v := v_1 - v_2,\quad  \delta \Pi := \Pi_1 - \Pi_2,\quad \delta A := A_1-A_2,
\end{split}
\end{equation}
with $A_i(t,y) = (\nabla_y X_{i,t})^{-1}(y)$. Then, we arrive at
\begin{equation}\label{eq.delta.v}
\left\{\begin{array}{ll}
  \partial_t \delta v + \Lambda^{2\alpha}_{v_1} \delta v + \nabla \delta\Pi = (1-\rho_0)\partial_t \delta v + \delta f_1 + \delta f_2 , \\
  \Div \delta v = \Div \delta g,\\
  \delta v|_{t=0} = 0,
\end{array}\right.
\end{equation}
where
\begin{align}\label{Lam-v1-v}
  \Lambda_{v_1}^{2\alpha} \delta v(t,y) : = c_\alpha\,\mathrm{p.v.}\int_{\RR^2} \frac{(X_{1,t}(y)- X_{1,t}(z))  \cdot (A_1^t(t,y) \nabla \delta v(t,y) - A_1^t(t,z) \nabla \delta v(t,z)) }{|X_{1,t}(y)-X_{1,t}(z)|^{2+2\alpha}} \dd z,
\end{align}
and
\begin{align}
  \delta f_1 & := (\nabla-\nabla_{v_1})\Pi_1-(\nabla-\nabla_{v_2})\Pi_2 = (\Id-A^t_1)\nabla \delta\Pi - (\delta A^t)\,\nabla \Pi_2, \label{eq:f1}\\
  \delta f_2 & := c_\alpha\, \mathrm{p.v.} \int_{\RR^2}\bigg( \frac{(X_{1,t}(y)-X_{1,t}(z))\cdot \big(A_1^t(t,y) \nabla v_2(t,y) - A_1^t(t,z)\nabla v_2(t,z)\big)}{|X_{1,t}(y)-X_{1,t}(z)|^{2+2\alpha}} \label{eq:f2}\\
  \quad & \qquad \qquad\qquad - \frac{(X_{2,t}(y)-X_{2,t}(z))\cdot \big(A_2^t(t,y) \nabla v_2(t,y) - A_2^t (t,z) \nabla v_2(t,z)\big) }{|X_{2,t}(y)-X_{2,t}(z)|^{2+2\alpha}}   \bigg) \dd z,\nonumber \\
  \delta g & := (\Id-A_1)v_1 - (\Id-A_2)v_2  = (\Id - A_1)\delta v - \delta A\, v_2.\label{eq:g}
\end{align}

Concerning the twisted fractional Stokes system \eqref{eq.delta.v}, we have the following $L^2_{T_1}(L^2)$ maximal regularity result on a short time interval.
\begin{proposition}\label{pro-uni-sta}
Let $(\delta v, \delta \Pi)$ be the solution to the system \eqref{eq.delta.v}, then there exists a sufficiently small constant $T_1>0$ depending on $\alpha, p,s,\|u_0\|_{H^1\cap\dot B^{\alpha+s}_{p,2}}$
such that
\begin{align}\label{est.sta}
  \delta E(T_1) & := \|\delta v \|_{L^{\infty}_{T_1}(\dot H^\alpha)} + \|(\partial_t \delta v ,\Lambda^{2\alpha} \delta v , \nabla \delta \Pi)\|_{L^2_{T_1} (L^2)}\nonumber\\
  &\;\leq C \|\Div \delta g \|_{L^2_{T_1}(\dot H^{2\alpha-1})} + C \|(\delta f_1,\delta f_2,\partial_t \delta g)\|_{L^2_{T_1}(L^2)},
\end{align}
where $C$ depends only on $\alpha$ and the upper bounds in Propositions \ref{prop:uni.M.bet} and \ref{pro.propri.Lag}.
\end{proposition}

\begin{remark}\label{rem:uni}
Following the ideas of \cite{DanM13,PaicuZhZh13}, it is more convenient to write the system $\eqref{eq.delta.v}$ as
\begin{equation}\label{eq.dv.r}
\left\{\begin{array}{ll}
  \partial_t \delta v + \Lambda^{2\alpha} \delta v + \nabla \delta\Pi = \delta h_1  +\,\textrm{the\,\,other\,\,terms}, \\
  \Div \delta v = \Div \delta g,\quad  \delta v|_{t=0} = 0,
\end{array}\right.
\end{equation}
where
\begin{equation*}\label{re.H}
\begin{split}
  \delta h_1 : = c_\alpha\,\mathrm{p.v.}\int_{\RR^2}
  \Big(\frac{X_{1,t}(y) -X_{1,t}(z)}{|X_{1,t}(y)- X_{1,t}(z)|^{2+2\alpha}} - \frac{y-z}{|y-z|^{2+2\alpha}}\Big)\cdot (\nabla\delta v(t,y)- \nabla\delta v(t,z)\big)\mathrm dz.
\end{split}
\end{equation*}
According to Proposition \ref{pro.Max-Regularity}, one can easily build the estimate \eqref{est.sta} analogously, and it suffices to bound the $L^2_{T_1}(L^2)$-norm of $\delta h_1$ and the other terms.
However, by letting $T_1>0$ small enough and using Lemma \ref{lem:Besov-fd}, one can have the following estimate
\begin{equation*}
\begin{split}
  \|\delta h_1\|_{L^2_{T_1}(L^2)} & \leq \varepsilon \Big\|\int_{\RR^2}
  \frac{|\nabla\delta v(t,y)-\nabla\delta v(t,z)|}{|y-z|^{1+2\alpha}}\mathrm dz\Big\|_{L^2_{T_1}(L^2_y)} \\
  & \leq \varepsilon \int_{\RR^2}  \frac{\|\nabla \delta v(t,y) - \nabla \delta v(t,y+z)\|_{L^2_y}}{|z|^{1+2\alpha}} \dd z \leq \varepsilon \| \nabla \delta v\|_{L^2_{T_1}(\dot B^{2\alpha-1}_{2,1})}
\end{split}
\end{equation*}
with $\varepsilon>0$ sufficiently small, but it seems difficult to control $\|\delta h_1\|_{L^2_{T_1}(L^2)}$ with the needing upper bound $\varepsilon \|\nabla \delta v\|_{L^2_{T_1}(\dot H^{2\alpha})}$.
Hence, we instead treat the system \eqref{eq.delta.v} directly to derive the key estimate \eqref{est.sta}, and the proof is more complicated than that in the 2D INS system \eqref{eq.INNS}.
\end{remark}

\begin{proof}[Proof of Proposition \ref{pro-uni-sta}]
Taking the inner product of $\eqref{eq.delta.v}_1$ with $\Lambda^{2\alpha}_{v_1}\delta v(t,y)$, we find
\begin{align}\label{eq.sta.L}
  &\quad \int_{\RR^2} \partial_t \delta v(t,y)\, \Lambda^{2\alpha}_{v_1}\delta v(t,y) \dd y + \int_{\RR^2} |\Lambda^{2\alpha}_{v_1} \delta v(t,y)|^2 \dd y\nonumber\\
  & = \int_{\RR^2} \big(- \nabla\delta \Pi + (1-\rho_0) \partial_t \delta v +  \delta f_1 + \delta f_2 \big)\, \Lambda^{2\alpha}_{v_1}\delta v(t,y) \dd y\nonumber\\
  & \leq \varepsilon \|\Lambda_{v_1}^{2\alpha}\delta v\|_{L^2}^2 + \frac{3}{4\varepsilon} \Big(\|\nabla\delta \Pi\|_{L^2}^2 + \|a_0\|_{L^\infty}^2 \|\partial_t \delta v\|_{L^2}^2 + \|(\delta f_1 ,\delta f_2)\|_{L^2}^2 \Big),
\end{align}
where $\varepsilon>0$ is a small constant chosen later.
Denoting by $\PP := \nabla\Delta^{-1}\Div$, we see that
\begin{align*}
  \nabla \delta \Pi = -\PP \partial_t \delta v  - \PP \Lambda^{2\alpha}_{v_1}\delta v + \PP (-a_0 \,\partial_t \delta v) + \PP (\delta f_1) + \PP (\delta f_2),
\end{align*}
which leads to that
\begin{align}\label{est0.Pi}
  \|\nabla\delta \Pi\|_{L^2}^2 \leq & \big(\|\PP\partial_t \delta v \|_{L^2} + \| \PP\Lambda^{2\alpha}_{v_1}\delta v \|_{L^2}
  + \| \PP(a_0 \partial_t \delta v)\|_{L^2} + \|\PP(\delta f_1,\delta f_2)\|_{L^2}\big) ^2 \nonumber\\
  \leq & 4 \Big(\|\partial_t \delta g\|_{L^2}^2 + \|\PP(\Lambda^{2\alpha}_{v_1}\delta v)\|_{L^2}^2
  + \|a_0\|_{L^\infty}^2 \|\partial_t \delta v\|_{L^2}^2 + \|(\delta f_1,\delta f_2)\|_{L^2}^2\Big).
\end{align}
Utilizing the equation $\eqref{eq.delta.v}_1$ and \eqref{est0.Pi} gives
\begin{align*}
  &\, \|\partial_t\delta v\|_{L^2}^2 + \frac{1}{4}\|\nabla\delta \Pi\|_{L^2}^2 \nonumber \\
  \leq & \|\Lambda^{2\alpha}_{v_1}\delta v\|_{L^2}^2 + \|a_0\|_{L^\infty}^2\|\partial_t \delta v\|_{L^2}^2 + \frac{5}{4}\|\nabla\delta \Pi\|_{L^2}^2 + \|(\delta f_1,\delta f_2)\|_{L^2}^2\nonumber\\
  \leq & \|\Lambda^{2\alpha}_{v_1}\delta v\|_{L^2}^2 + 5 \|\partial_t \delta g\|_{L^2}^2 + 5\|\PP(\Lambda^{2\alpha}_{v_1}\delta v)\|_{L^2}^2 + 6\|a_0\|_{L^\infty}^2\|\partial_t \delta v\|_{L^2}^2 + 6\|(\delta f_1,\delta f_2)\|_{L^2}^2 ,
\end{align*}
then by assuming $\|a_0\|_{L^\infty}\leq \frac{1}{4}$ without loss of generality, we infer that
\begin{align}\label{est0.ptv}
  \|\partial_t\delta v\|_{L^2}^2 + \frac{\|\nabla\delta \Pi\|_{L^2}^2 }{2}
  \leq 2 \|\Lambda^{2\alpha}_{v_1}\delta v\|_{L^2}^2 + 10\|\partial_t \delta g\|_{L^2}^2
  + 10 \|\PP(\Lambda^{2\alpha}_{v_1}\delta v)\|_{L^2}^2 + 12 \|(\delta f_1,\delta f_2)\|_{L^2}^2 .
\end{align}
Letting $\varepsilon_1>0$ be a small constant chosen later, we insert \eqref{est0.Pi} into \eqref{eq.sta.L} and then combine it with $\eqref{est0.ptv}\times \varepsilon_1$ to obtain
\begin{align*}
  &\int_{\RR^2} \partial_t \delta v\, \Lambda^{2\alpha}_{v_1}\delta v(t,y) \dd y+ \big(1-(\varepsilon+ 2 \varepsilon_1)\big) \|\Lambda^{2\alpha}_{v_1}\delta v\|_{L^2}^2
  + \Big(\varepsilon_1 - \frac{15\|a_0\|_{L^\infty}^2}{4\varepsilon}\Big) \|\partial_t\delta v\|_{L^2}^2 + \frac{\varepsilon_1}{2} \|\delta \Pi\|_{L^2}^2\nonumber\\
  &\leq \Big(\frac{3}{\varepsilon} + 10\varepsilon_1 \Big) \big( \|\partial_t \delta g\|_{L^2}^2
  +\|\PP(\Lambda^{2\alpha}_{v_1}\delta v)\|_{L^2}^2\big) + \Big(\frac{15}{4\varepsilon} + 12\varepsilon_1 \Big)\|(\delta f_1,\delta f_2)\|_{L^2}^2 ,
\end{align*}
so that by setting $\varepsilon = \varepsilon_1 = \frac{1}{4}$, and assuming $\|a_0\|_{L^\infty}\leq \frac{1}{9}$ without loss of generality, it leads to
\begin{align}\label{est.sta.L-2}
  &\quad \int_{\RR^2} \partial_t \delta v(t,y)\, \Lambda^{2\alpha}_{v_1}\delta v(t,y) \dd y+ \frac{1}{4} \|\Lambda^{2\alpha}_{v_1}\delta v(t)\|_{L^2}^2
  + \frac{1}{16} \|\partial_t\delta v(t)\|_{L^2}^2 + \frac{1}{8}\|\nabla\delta \Pi(t)\|_{L^2}^2\nonumber\\
  &\leq 15  \|\partial_t \delta g(t)\|_{L^2}^2  + 15 \|\PP(\Lambda^{2\alpha}_{v_1}\delta v)(t)\|_{L^2}^2 + 18 \|(\delta f_1,\delta f_2)(t)\|_{L^2}^2 .
\end{align}
Integrating on the time variable leads to that for every $t\in [0,T_1]$,
\begin{align}\label{est.sta.L-3}
  & \quad \int_0^t \int_{\RR^2} \partial_\tau \delta v \, \Lambda^{2\alpha}_{v_1}\delta v(\tau,y) \dd y \dd \tau
  + \frac{1}{4} \|\Lambda^{2\alpha}_{v_1}\delta v\|_{L^2_t(L^2)}^2
  + \frac{1}{16} \|\partial_\tau\delta v\|_{L^2_t(L^2)}^2 + \frac{1}{8} \|\nabla\delta \Pi\|_{L^2_t(L^2)}^2 \nonumber\\
  &\leq 15  \|\partial_t \delta g\|_{L^2_t (L^2)}^2  + 15 \|\PP(\Lambda^{2\alpha}_{v_1}\delta v)\|_{L^2_t (L^2)}^2 + 18 \|(\delta f_1,\delta f_2)\|_{L^2_t(L^2)}^2 .
\end{align}

 Next,  observing that by equality \eqref{Lam-2alp-lag}, formula \eqref{eqA}  and the change of variables
\begin{align}\label{eq.diss-exp}
  & \quad \|\Lambda_{v_1}^{2\alpha}\delta v(t,y)\|_{L^2_y}^2 \nonumber \\
  & = c_\alpha^2 \int_{\RR^2} \Big|\int_{\RR^2} \frac{(X_{1,t}(y)- X_{1,t}(z))  \cdot (A_1^t(t,y) \nabla \delta v(t,y) - A_1^t(t,z) \nabla \delta v(t,z)) }{|X_{1,t}(y)-X_{1,t}(z)|^{2+2\alpha}} \dd z \Big|^2 \dd y \nonumber \\
  & = c_\alpha^2 \int_{\RR^2} \Big| \int_{\RR^2} \frac{(x - {z} ) \cdot \big(\nabla \big(\delta v(t, X_{1,t}^{-1}(x)) \big)
  - \nabla \big(\delta v(t, X_{1,t}^{-1}(z)) \big)\big)}{|x-{z}|^{2+2\alpha}} \dd {z} \Big|^2 \dd x \nonumber \\
  & = \|\Lambda^{2\alpha-2}\nabla \cdot \nabla\big( \delta v(t,X_{1,t}^{-1}(x))\big)\|_{L^2_x}^2 = \|\Lambda^{2\alpha}\big( \delta v(t,X_{1,t}^{-1}(x))\big)\|_{L^2_x}^2 ,
\end{align}
we get
\begin{align*}
  \|\Lambda_{v_1}^{2\alpha}\delta v(t,y)\|_{L^2_y}^2  & \geq  \|\Lambda^{2\alpha-1} \nabla_x \big(\delta v(t, X_{1,t}^{-1}(x))\big)\|_{L^2_x}^2 \\
  & = \|\Lambda^{2\alpha-1} \big(\nabla X_{1,t}^{-1}(x)\,(\nabla \delta v)(t, X_{1,t}^{-1}(x)) \big)\|_{L^2_x}^2  \\
  & \geq \frac{1}{2} \|\nabla \delta v(t, X_{1,t}^{-1}(x)) \|_{\dot H^{2\alpha-1}_x}^2 - \|(\Id - \nabla X_{1,t}^{-1})\,\nabla \delta v(t, X_{1,t}^{-1}(x))\|_{\dot H^{2\alpha-1}_x}^2 \\
  & = : N_1 + N_2 ,
\end{align*}
where in the third line we have used the simple inequality $(a-b)^2 \geq \frac{1}{2} a^2 - b^2$ for $a,b>0$.
Notice that (owing to Lemma \ref{lem:Besov-fd})
\begin{align}\label{eq.N1}
  N_1 & \geq  \frac{1}{2C_1}\int_{\RR^2} \int_{\RR^2} \frac{|\nabla \delta v(t,X_{1,t}^{-1}(x)) -\nabla \delta v(t,X_{1,t}^{-1}(z))|^2}{|x-z|^{4\alpha}} \dd x \dd z \nonumber \\
  & = \frac{1}{2C_1}\int_{\RR^2} \int_{\RR^2} \frac{|\nabla \delta v(t,y) -\nabla \delta v(t,z)|^2}{|X_{1,t}(y)- X_{1,t}(z)|^{4\alpha}} \dd y \dd z .
\end{align}
Making use of \eqref{eq.flow} and the mean value theorem yields that for $i=1,2,$
\begin{align*}\label{delta.varepsilon}
  |X_{i,t}(y)-X_{i,t}(z)|\leq &|y-z|+\int_0^t\big|u_i(\tau,X_{i,\tau}(y))-u_i(\tau,X_{i,\tau}(z))\big|\mathrm d\tau \nonumber \\
  \leq &|y-z|+\int_0^t\|\nabla u_i\|_{L^\infty}\big|X_{i,\tau}(y) - X_{i,\tau}(z)\big|\mathrm d\tau ,
\end{align*}
and
\begin{align*}
  |y-z|\leq &\big|X_{i,t}(y)-X_{i,t}(z)\big|+\int_0^t\big|v_i(\tau,y)-v_i(\tau,z)\big|\mathrm d\tau\nonumber \\
  \leq &\big|X_{i,t}(y)-X_{i,t}(z)\big|+\int_0^t\|\nabla v_i\|_{L^\infty}|y-z|\mathrm d\tau ,\quad
\end{align*}
thus Gronwall's inequality guarantees that
\begin{align*}
  |y-z|e^{-\int_0^t||\nabla v_i\|_{L^\infty}\mathrm d\tau} \leq |X_{i,t}(y)-X_{i,t}(z)| \leq |y-z|e^{\int_0^t||\nabla u_i\|_{L^\infty}\mathrm d\tau} .
\end{align*}
Hence, by taking  $T_1>0$ small enough, we have that for any $t\leq T_1$,
\begin{equation}\label{can.delta.sip}
  \frac{3}{4}\leq \frac{|y-z|} {|X_{i,t}(y)-X_{i,t}(z)|}\leq \frac{4}{3},\quad\,\forall y\neq z\in \RR^2 ,
\end{equation}
and equivalently,
\begin{equation}\label{can.delta.sip2}
  \frac{3}{4}\leq \frac{|X_{i,t}^{-1}(y)- X_{i,t}^{-1}(z)|} {| y - z|}\leq \frac{4}{3},\quad\,\forall y\neq z\in \RR^2 .
\end{equation}
Thus it follows from \eqref{eq.N1} and \eqref{can.delta.sip} that
\begin{align*}
  N_1 & \geq \frac{1}{2(3/4)^{4\alpha} C_1}
  \int_{\RR^2} \int_{\RR^2} \frac{|\nabla \delta v(t,y) -\nabla \delta v(t,z)|^2}{|y- z|^{4\alpha}} \dd y \dd z \\
  & \geq \frac{1}{8C_1} \|\Lambda^{2\alpha-1}\nabla \delta v\|_{L^2}^2 \geq \frac{1}{8 C_1} \|\Lambda^{2\alpha} \delta v\|_{L^2}^2 ,
\end{align*}
and then for every $t\in [0,T_1]$,
\begin{align}\label{es.N1-2}
  \|N_1\|_{L^1([0,t])}\geq \frac{1}{8 C_1} \|\Lambda^{2\alpha} \delta v\|_{L^2_t (L^2)}^2 .
\end{align}
Concerning $N_2$, by applying Lemmas \ref{lem:frc-leib} and \ref{lem:f.mu}, and using \eqref{eqA}, \eqref{est.uv-lip} we find
\begin{align}\label{es.N2}
  & \|N_2\|_{L^1([0,T_1])} \nonumber \\
  \leq & C \|\nabla \delta v(t, X_{1,t}^{-1}(x))\|^2_{L^2_{T_1}(\dot H^{2\alpha -1})}
  \|\Id -A_1(t,X_{1,t}^{-1}(x)) \|_{L^\infty_{T_1}(L^\infty)}^2 \nonumber \\
  &\, + C \|\nabla \delta v (t, X_{1,t}^{-1}(x))\|_{L^2_{T_1}(L^{\frac{2p}{p-2}})}^2 \|\Id - A_1(t,X_{1,t}^{-1}(x)) \|_{L^\infty_{T_1}(\dot W^{2\alpha-1,p})}^2\nonumber \\
  \leq & C \|\nabla \delta v\|^2_{L^2_{T_1}(\dot H^{2\alpha -1})} \|\Id -A_1 \|^2_{L^\infty_{T_1}(L^\infty)}
  + C \|\nabla \delta v\|^2_{L^2_{T_1}(L^{\frac{2p}{p-2}})} \|\Id - A_1\|^2_{L^\infty_{T_1}(\dot B^{2\alpha-1}_{p,1})} ,
\end{align}
where in the last line we have used the embedding $\dot B^{2\alpha-1}_{p,1}\hookrightarrow \dot W^{2\alpha-1,p}$.
Recalling that under the condition \eqref{est.uv-lip},
\begin{align}\label{eq.A_i}
    A_i(t,y) = (\nabla X_{i,t}(y))^{-1}=(\Id+B_i(t,y))^{-1}=\Id + \sum_{k=1}^\infty (- B_i(t,y))^k,\quad i=1,2,
\end{align}
with $B_i(t,y):= \nabla X_{i,t}(y) -\Id = \int_0^t\nabla v_i(\tau,y)\,\mathrm d\tau $,
it is easy to see that
\begin{align}\label{es.Id-A.infty}
  \big\|\Id-A_i\big\|_{L^\infty_{T_1}(L^\infty)} \leq \sum_{k=1}^\infty \big\|B_i\big\|_{L^\infty_{T_1}( L^\infty)}^k
  \leq \sum_{k=1}^\infty \| \nabla v_i\|^k_{L^1_{T_1}(L^\infty)}
  \leq 2 \|\nabla v_i\|_{L^1_{T_1}(L^\infty)} \leq 2 C T_1^{\frac{1}{2}}.
\end{align}
While for the term $\|\Id-A_i\|_{L^\infty_T(\dot B^{2\alpha-1}_{p,1})}$,
due to $2\alpha-1- \frac{2}{p}>0$, the nonhomogeneous space $B^{2\alpha-1}_{p,1}(\RR^2)\hookrightarrow L^\infty(\RR^2)$ is a Banach algebra,
thus together with Proposition \ref{pro.propri.Lag} and letting $T_1>0$ small enough, it yields that for $i=1,2,$
\begin{align}\label{es.Id-A.wp}
  \big\|\Id-A_i\big\|_{L^\infty_{T_1}(\dot B^{2\alpha-1}_{p,1})}
  & \leq \sum_{k=1}^\infty\|B_i(t,y)\|_{L^\infty_{T_1}( B^{2\alpha-1}_{p,1})}^k \nonumber\\
  &\leq \sum_{k=1}^\infty \Big(C T_1^{\frac{1}{2}}\|\nabla v_i\|_{L^2_{T_1}(B^{2\alpha-1}_{p,1})}\Big)^k \nonumber\\
  &\leq \sum_{k=1}^\infty  \Big(C T_1^{\frac{1}{2}} \| v_i\|_{L^2_{T_1}( \dot B^{2\alpha}_{p,1} \cap \dot H^\alpha)} \Big)^k\nonumber\\
  &\leq C \sqrt{T_1} \| v_i\|_{L^2_{T_1}( \dot B^{2\alpha}_{p,1} \cap \dot H^\alpha)} \leq C \sqrt{T_1}.
\end{align}
Using the Sobolev embedding $\dot H^\alpha \cap \dot H^{2\alpha}(\RR^2) \hookrightarrow \dot W^{1,\frac{2p}{p-2}}(\RR^2)$ (due to $p>\frac 2{2\alpha-1}$),
 we also get
\begin{align}\label{es.intp}
  \big\|\nabla \delta v\big\|_{L^2_{T_1}(L^{\frac{2p}{p-2}})} \leq
  C  \|\delta v\|_{L^2_{T_1} (\dot H^\alpha)}^{\frac{(2\alpha-1)p-2}{\alpha p}} \|\delta v\|_{L^2_{T_1} (\dot H^{2\alpha})}^{\frac{(1-\alpha)p+2}{\alpha p}}
  \leq C \|\delta v\|_{L^2_{T_1}(\dot H^\alpha\cap \dot H^{2\alpha})} .
\end{align}
Collecting estimates \eqref{es.N2} and \eqref{es.Id-A.infty}--\eqref{es.intp} yields
\begin{align}\label{es.N2-2}
  \|N_2\|_{L^1([0,T_1])} \leq C T_1 \|\delta v\|^2_{L^2_{T_1}(\dot H^{2\alpha})} + C T_1 \|\delta v\|^2_{L^2_{T_1}(\dot H^\alpha)}.
\end{align}

We then consider the first term on the left-hand side of \eqref{est.sta.L-2}. In light of \eqref{Lam-v1-v} and \eqref{eq.diss-exp}, we see that
\begin{align*}
  &\quad  \int_{\RR^2} \partial_t \delta v(t,y)\, \Lambda^{2\alpha}_{v_1} \delta v(t,y) \dd y \\
  & = c_\alpha \int_{\RR^2} \int_{\RR^2} \partial_t \delta v(t,y)\,\frac{(X_{1,t}(y)- X_{1,t}(z))  \cdot (A_1^t(t,y) \nabla \delta v(t,y) - A_1^t(t,z) \nabla \delta v(t,z)) }{|X_{1,t}(y)-X_{1,t}(z)|^{2+2\alpha}} \dd z \dd y \\
  & = \int_{\RR^2} \partial_t \delta v\big(t, X_{1,t}^{-1}(x)) \, \Lambda^{2\alpha}_x\big(\delta v(t, X_{1,t}^{-1}(x))\big) \dd x  \\
  & = \int_{\RR^2} \partial_t \big(\delta v(t, X_{1,t}^{-1}(x)) \big) \, \Lambda^{2\alpha}_x\big(\delta v(t, X_{1,t}^{-1}(x))\big) \dd x\\
  &\quad - \int_{\RR^2} \partial_t X_{1,t}^{-1}(x)\,\cdot\nabla \delta v(t, X_{1,t}^{-1}(x)) \,
  \Lambda^{2\alpha}_x \big(\delta v(t, X_{1,t}^{-1}(x)) \big) \dd x \\
  & = : \Psi_1 + \Psi_2 .
\end{align*}
For $\Psi_1$, noticing that
\begin{align*}
  \Psi_1 & = \frac{1}{2} \frac{\dd }{\dd t} \| \big(\delta v(t, X_{1,t}^{-1}(x)) \big) \|_{\dot H^\alpha_x}^2  \\
  & = \frac{1}{2} \frac{\dd }{\dd t} \int_{\RR^2} \int_{\RR^2} \frac{|\delta v(t, X_{1,t}^{-1}(x)) - \delta v(t, X_{1,t}^{-1}(z))|^2}{|x-z|^{2+2\alpha}} \dd x \dd z \\
  & = \frac{1}{2} \frac{\dd }{\dd t} \int_{\RR^2} \int_{\RR^2} \frac{|\delta v(t, y) - \delta v(t, \tilde{z})|^2}{|X_{1,t}(y)-X_{1,t}(\tilde{z})|^{2+2\alpha}} \dd y \dd \tilde{z} ,
\end{align*}
thus by letting $T_1>0$ small enough so that \eqref{can.delta.sip} holds,
we integrate in the time variable and then use \eqref{can.delta.sip} to have that for every $t\in [0,T_1]$,
\begin{align}\label{es.Psi}
  \int_0^t \Psi_1 \dd \tau \geq \frac{1}{2\cdot (4/3)^{2+2\alpha}} \|\delta v(t,\cdot)\|_{\dot H^\alpha}^2
  \geq \frac{1}{8} \|\delta v(t,\cdot)\|_{\dot H^\alpha}^2.
\end{align}
Noting that (from \eqref{eq:X-invers})
\begin{align*}
  \partial_t X_{i,t}^{-1}(x) = u_0(X_{i,t}^{-1}(x)) + \int_0^t \partial_t u_i(t-\tau, X_{i,\tau}^{-1}(x)) \dd \tau ,
\end{align*}
and using the equality \eqref{eq.diss-exp} and interpolation inequality \eqref{es.intp}, the term $\Psi_2$ can be estimated as
\begin{align}\label{es.Psi2}
  \int_0^t |\Psi_2| \dd \tau  & \leq   \int_0^t \|\partial_\tau X_{1,\tau}^{-1}(x)\|_{L^p_x}
  \| \nabla \delta v(\tau, X_{1,\tau}^{-1}(x))\|_{L^{\frac{2p}{p-2}}_x}
  \|\delta v(\tau, X_{1,\tau}^{-1}(x))\|_{\dot H^{2\alpha}_x} \dd \tau \nonumber\\
  &\leq C \|\partial_\tau X_{1,\tau}^{-1}\|_{L^\infty_t (L^p)} \|\nabla \delta v\|_{L^2_t (L^{\frac{2p}{p-2}})} \|\delta v(\tau, X_{1,\tau}^{-1}(x))\|_{L^2_t(\dot H^{2\alpha})} \nonumber\\
  & \leq C \big(\|u_0\|_{L^p}^2 + \|\partial_\tau u_1\|_{L^1_t(L^p)}^2\big) \| \delta v\|_{L^2_t(\dot H^\alpha)}^{\frac{2(2\alpha-1)p-4}{\alpha p}}
  \|\delta v\|_{L^2_t(\dot H^{2\alpha})}^{\frac{2(1-\alpha)p +4}{\alpha p}} + \frac{1}{16}\|\Lambda^{2\alpha}_{v_1} \delta v(\tau,y) \|_{L^2_t(L^2)}^2 \nonumber\\
  &\leq C \big(\|u_0\|_{L^p}^2 + t \|\partial_\tau u_1\|_{L^2_t (L^p)}^2 \big)^{\frac{\alpha p}{(2\alpha-1)p-2}} \|\delta v\|_{L^2_t(\dot H^\alpha)}^2
  + \frac{\|\delta v\|_{L^2_t (\dot H^{2\alpha})}^2}{128 C_1} + \frac{1}{16}\|\Lambda^{2\alpha}_{v_1}\delta v\|_{L^2_t(L^2)}^2 \nonumber \\
  & \leq C t \|\delta v\|_{L^\infty_t(\dot H^\alpha)}^2 + \frac{1}{128 C_1}\|\delta v\|_{L^2_t(\dot H^{2\alpha})}^2 + \frac{1}{16}\|\Lambda^{2\alpha}_{v_1}\delta v\|_{L^2_t(L^2)}^2  .
\end{align}
Integrating on the time interval $[0,t]$ yields that for every $t\in [0,T_1]$,
\begin{align}\label{eq:term1}
  \int_0^t \int_{\RR^2} \partial_\tau \delta v\, \Lambda^{2\alpha}_{v_1} \delta v \dd y \dd \tau \geq \frac{\|\delta v(t)\|_{\dot H^\alpha}^2 }{8}
  - C t \|\delta v\|_{L^\infty_t(\dot H^\alpha)}^2 - \frac{\|\delta v\|_{L^2_t(\dot H^{2\alpha})}^2}{128 C_1} - \frac{\|\Lambda^{2\alpha}_{v_1}\delta v\|_{L^2_t(L^2)}^2}{16}.
\end{align}

Now we consider $\|\PP(\Lambda_{v_1}^{2\alpha} \delta v)\|_{L^2}$ (recalling that $\PP := \nabla \Delta^{-1} \Div $). By arguing as \eqref{eq.diss-exp},
and using \eqref{Lam-v1-v} and the change of variables,
we find
\begin{align*}
  & \quad \|\PP (\Lambda_{v_1}^{2\alpha} \delta v)\|_{L^2}^2 \leq \|\Div(\Lambda_{v_1}^{2\alpha}\delta v)\|_{\dot H^{-1}}^2 \\
  & \lesssim \int_{\RR^2} \Big| \int_{\RR^2} \frac{\Div(\Lambda_{v_1}^{2\alpha} \delta v)(y)}{|x-y|} \dd y\Big|^2 \dd x \\
  & \lesssim \int_{\RR^2} \Big| \int_{\RR^2} \frac{1}{|x-y|} \Div_y\Big( \int_{\RR^2} \frac{(X_{1,t}(y)- X_{1,t}(z))
  \cdot (A_1^t(t,y) \nabla \delta v(t,y) - A_1^t(t,z) \nabla \delta v(t,z)) }{|X_{1,t}(y)-X_{1,t}(z)|^{2+2\alpha}} \dd z\Big)\dd y \Big|^2 \dd x \\
  & \lesssim \int_{\RR^2} \Big| \int_{\RR^2} \frac{1}{|X_{1,t}^{-1}(\tilde{x})- X_{1,t}^{-1}(\tilde{y})|} \Div_{\tilde{y}}\Big(\big(\nabla X_{1,t}^{-1}(\tilde{y})\big)^{-1}\cdot \\
  & \qquad\qquad \int_{\RR^2} \frac{( \tilde{y}- \tilde{z})
  \cdot \big( \nabla [\delta v(t,X_{1,t}^{-1}(\tilde{y}))] - \nabla [\delta v(t,X_{1,t}^{-1}(\tilde{z}))]\big) }{|\tilde{y}- \tilde{z}|^{2+2\alpha}} \dd \tilde{z}\Big)\dd \tilde{y} \Big|^2 \dd \tilde{x} \\
  & \lesssim \int_{\RR^2} \Big| \int_{\RR^2} \frac{1}{|X_{1,t}^{-1}(\tilde{x})- X_{1,t}^{-1}(\tilde{y})|} \Div_{\tilde{y}}\Big(\big(\nabla X_{1,t}^{-1}(\tilde{y})\big)^{-1}
  \Lambda^{2\alpha} \big(\delta v(t,X_{1,t}^{-1}(\tilde{y})) \big)\Big)\dd \tilde{y} \Big|^2 \dd \tilde{x}  \\
  & \lesssim \int_{\RR^2} \Big| \int_{\RR^2} \frac{1}{|X_{1,t}^{-1}(\tilde{x})- X_{1,t}^{-1}(\tilde{y})|} \Div_{\tilde{y}}\Big(
  \Lambda^{2\alpha} \big(\delta v(t,X_{1,t}^{-1}(\tilde{y})) \big)\Big)\dd \tilde{y} \Big|^2 \dd \tilde{x} \\
  & \quad + \int_{\RR^2} \Big| \int_{\RR^2} \frac{1}{|X_{1,t}^{-1}(\tilde{x})- X_{1,t}^{-1}(\tilde{y})|} \Div_{\tilde{y}}\Big(\big(\Id-\big(\nabla X_{1,t}^{-1}(\tilde{y})\big)^{-1}\big)
  \Lambda^{2\alpha} \big(\delta v(t,X_{1,t}^{-1}(\tilde{y})) \big)\Big)\dd \tilde{y} \Big|^2 \dd \tilde{x} \\
  & = : \Upsilon_1 + \Upsilon_2.
\end{align*}
For $\Upsilon_1$, by using the relation $\eqref{eq.delta.v}_2$ and the change of variables again,
it follows that
\begin{align*}
  \Upsilon_1 & \lesssim \int_{\RR^2} \Big| \int_{\RR^2} \frac{1}{|X_{1,t}^{-1}(\tilde{x})- X_{1,t}^{-1}(\tilde{y})|}
  \Big( \Lambda^{2\alpha} \Div \big(\delta v(t,X_{1,t}^{-1}(\cdot)) \big)(\tilde{y})\Big)\dd \tilde{y} \Big|^2 \dd \tilde{x} \\
  & \lesssim  \int_{\RR^2} \Big| \int_{\RR^2} \frac{1}{|X_{1,t}^{-1}(\tilde{x}) - X_{1,t}^{-1}(\tilde{y})|}
  \Big( \Lambda^{2\alpha} \big(\nabla X_{1,t}^{-1}(\cdot) : \nabla \delta v(t,X_{1,t}^{-1}(\cdot)) \big)(\tilde{y})\Big)\dd \tilde{y} \Big|^2 \dd \tilde{x} \\
  & \lesssim  \int_{\RR^2} \Big| \int_{\RR^2} \frac{1}{|X_{1,t}^{-1}(\tilde{x}) - X_{1,t}^{-1}(\tilde{y})|}
  \Big( \Lambda^{2\alpha} \big(\Div \delta v(t,X_{1,t}^{-1}(\cdot)) \big)(\tilde{y})\Big)\dd \tilde{y} \Big|^2 \dd \tilde{x} \\
  & \quad + \int_{\RR^2} \Big| \int_{\RR^2} \frac{1}{|X_{1,t}^{-1}(\tilde{x}) - X_{1,t}^{-1}(\tilde{y})|}
  \Big( \Lambda^{2\alpha} \big( (\Id -\nabla X_{1,t}^{-1}(\cdot)) : \nabla \delta v(t,X_{1,t}^{-1}(\cdot)) \big)(\tilde{y})\Big)\dd \tilde{y} \Big|^2 \dd \tilde{x} \\
  & \lesssim  \int_{\RR^2} \Big| \int_{\RR^2} \frac{1}{|X_{1,t}^{-1}(\tilde{x})- X_{1,t}^{-1}(\tilde{y})|}
  \Div_{\tilde{y}} \Big( \nabla\Lambda^{2\alpha-2} \big( \Div\delta g(t,X_{1,t}^{-1}(\cdot)) \big)(\tilde{y})\Big)\dd \tilde{y} \Big|^2 \dd \tilde{x} \\
  & \quad + \int_{\RR^2} \Big| \int_{\RR^2} \frac{1}{|X_{1,t}^{-1}(\tilde{x}) - X_{1,t}^{-1}(\tilde{y})|}
  \Div_{\tilde{y}} \Big( \nabla\Lambda^{2\alpha-2} \big( (\Id -\nabla X_{1,t}^{-1}(\cdot)) : \nabla \delta v(t,X_{1,t}^{-1}(\cdot)) \big)(\tilde{y})\Big)\dd \tilde{y} \Big|^2 \dd \tilde{x} \\
  & \lesssim \int_{\RR^2} \Big| \int_{\RR^2} \frac{1}{|x- y|}
  \,\Div_y \Big( A_1(t,y) \big(\big[\, \nabla \Lambda^{2\alpha-2} \big( \Div\delta g(t,X_{1,t}^{-1}) \big)\big]\circ X_{1,t}(y)\big)\Big) \dd y \Big|^2 \dd x \\
  & \quad + \int_{\RR^2} \Big| \int_{\RR^2} \frac{1}{|x - y|} \Div_y \Big(A_1(t,y)\, \big(\big[\nabla\Lambda^{2\alpha-2} \big( (\Id -\nabla X_{1,t}^{-1}) : \nabla \delta v(t,X_{1,t}^{-1}) \big)\big]\circ X_{1,t}(y)\big)
  \Big)\dd y \Big|^2 \dd \tilde{x} \\
  & \lesssim \Big\|\Lambda^{-1}\Div \Big(A_1(t,\cdot) \,\big[\nabla \Lambda^{2\alpha-2}\big(\Div \delta g(t,X_{1,t}^{-1}) \big)\big]\circ X_{1,t}(\cdot) \Big) \Big\|_{L^2}^2 \\
  & \quad + \Big \| \Lambda^{-1}\Div \Big(A_1(t,\cdot) \,\big[\nabla\Lambda^{2\alpha-2} \big( (\Id -\nabla X_{1,t}^{-1}) : \nabla \delta v(t,X_{1,t}^{-1}) \big)\big] \circ X_{1,t}(\cdot) \Big)\Big\|_{L^2}^2 \\
  & \leq C  \|A_1(t,y)\|_{L^\infty_y}^2 \|\nabla \Lambda^{2\alpha-2} \big(\Div\delta g(t,X_{1,t}^{-1}(\cdot)) \big)(x)\|_{L^2_x}^2 \\
  & \quad + C \|A_1(t,y)\|_{L^\infty_y}^2 \|\nabla \Lambda^{2\alpha-2} \big((\Id -\nabla X_{1,t}^{-1}(\cdot)): \nabla \delta v(t, X_{1,t}^{-1}(\cdot)) \big)(x)\|_{L^2_x}^2.
\end{align*}
In a similar way as estimating \eqref{es.N2}, by applying Lemma \ref{lem:f.mu}, Proposition \ref{pro.propri.Lag} and letting $T_1>0$ small enough, we infer that
\begin{align}\label{Ups1}
  \|\Upsilon_1\|_{L^1_{T_1}}
  & \lesssim \|\Div\delta g(t,X_{1,t}^{-1}(\cdot)) \|_{L^2_{T_1} (\dot  H^{2\alpha-1})}^2
  + \|\big(\Id - A_1(t, X_{1,t}^{-1}(x))\big): \nabla \delta v(t, X_{1,t}^{-1}(x))\|_{L^2_{T_1}(\dot H^{2\alpha-1})}^2 \nonumber\\
  & \leq C \|\Div \delta g \|_{L^2_{T_1}(\dot H^{2\alpha-1})}^2 + C \|\Id - A_1\|_{L^\infty_{T_1}(L^\infty)}^2 \|\nabla \delta v\|_{L^2_{T_1}(\dot H^{2\alpha-1})}^2 \nonumber\\
  & \quad + C \|\Id -A_1 \|_{L^\infty_{T_1}(\dot B^{2\alpha-1}_{p,1})}^2 \|\nabla \delta v\|_{L^2_{T_1}(L^{\frac{2p}{p-2}})}^2 \nonumber \\
  & \leq C \|\Div \delta g\|_{L^2_{T_1}(\dot H^{2\alpha-1})}^2 + C T_1 \|\delta v\|^2_{L^2_{T_1}(\dot H^{2\alpha})} + C T_1 \|\delta v\|^2_{L^2_{T_1}(\dot H^\alpha)} .
 \end{align}

For $\Upsilon_2$, observing that $(\nabla X_{1,t}^{-1}(\tilde{y}))^{-1} = \nabla X_{1,t}(y) =(\nabla X_{1,t})\circ X_{1,t}^{-1}(\tilde{y}) $,
and by letting $T_1>0$ small enough (so that \eqref{can.delta.sip2} and the last inequality in \eqref{Ups2} hold true), we obtain that for every $t\in [0,T_1]$,
\begin{align}\label{Ups2}
  \|\Upsilon_2\|_{L^1_t } & \lesssim  \Big\| \int_{\RR^2} \frac{1}{|\tilde{x}-\tilde{y}|} \Div_{\tilde{y}}
  \Big( \big(\Id -  (\nabla X_{1,\tau})\circ X_{1,\tau}^{-1}(\tilde{y}) \big)\,\big[\Lambda^{2\alpha}(\delta v(\tau, X_{1,\tau}^{-1}))\big]\circ X_{1,\tau}(\tilde{y}) \Big) \dd \tilde{y} \Big\|_{L^2_t(L^2_{\tilde{x}})}^2 \nonumber\\
  & \lesssim \Big\| \Lambda^{-1} \Div \Big(  \big(\Id - (\nabla X_{1,\tau})\circ X_{1,\tau}^{-1}(\cdot) \big)\, \big[\Lambda^{2\alpha}(\delta v(\tau, X_{1,\tau}^{-1})) \big]\circ X_{1,\tau}(\cdot)\Big) \Big\|_{L^2_t (L^2_{\tilde{x}})}^2 \nonumber\\
  & \lesssim \|\Id - \nabla X_{1,\tau}\|_{L^\infty_t(L^\infty)}^2
  \big\| \Lambda^{2\alpha} \big(\delta v(\tau, X_{1,\tau}^{-1}(\cdot))\big)\big\|_{L^2_t(L^2)}^2 \nonumber\\
  & \leq C  \|\nabla v_1\|_{L^1_t (L^\infty)}^2 \|\Lambda_{v_1}^{2\alpha} \delta v(\tau,x)\|_{L^2_t (L^2_x)}^2 \nonumber \\
  & \leq C  T_1 \|\Lambda_{v_1}^{2\alpha} \delta v(\tau,x)\|_{L^2_t (L^2)}^2 \leq \frac{1}{16} \|\Lambda_{v_1}^{2\alpha} \delta v(\tau,x)\|_{L^2_t (L^2)}^2,
\end{align}
where in the last line we used \eqref{eq.diss-exp}.

Gathering \eqref{est.sta.L-3} and \eqref{eq.diss-exp}, \eqref{es.N1-2}, \eqref{es.N2-2}, \eqref{eq:term1}--\eqref{Ups2}, we conclude that for every $t\in [0,T_1]$,
\begin{align}
  & \quad \frac{\|\delta v(t)\|_{\dot H^\alpha}^2 }{8}
  + \frac{1}{128 C_1} \|\Lambda^{2\alpha} \delta v\|_{L^2_t (L^2)}^2  + \frac{1}{16} \|\partial_\tau\delta v\|_{L^2_t(L^2)}^2
  + \frac{1}{8} \|\nabla \delta \Pi\|_{L^2_t(L^2)}^2 \nonumber \\
  & \leq C  T_1 \big(\|\delta v\|_{L^\infty_{T_1}(\dot H^\alpha)}^2 +  \|\delta v\|^2_{L^2_{T_1}(\dot H^{2\alpha})}\big)
  + C \|\Div \delta g\|_{L^2_{T_1}(\dot H^{2\alpha-1})}^2 + C \|(\delta f_1,\delta f_2,\partial_t g)\|_{L^2_{T_1}(L^2)}^2.
\end{align}
Hence by taking the supremum over $[0,T_1]$ and then by letting $T_1>0$ small enough, we can conclude the desired estimate \eqref{est.sta}.
\end{proof}

Now in order to get the uniqueness result in Theorem \ref{Main.thm}, we are in a position to check the right-hand terms of estimate \eqref{est.sta}.
As regards to $\|\delta f_1\|_{L^2_{T_1}(L^2)}$, with $\delta f_1=- (\Id-A^t_1)\nabla \delta\Pi + \delta A^t\nabla \Pi_2$,
recalling \eqref{eq.A_i}--\eqref{es.Id-A.infty} and noting that
\begin{align}\label{eq.del-Ai}
    \delta A(t,y) = \sum_{k=1}^\infty \Big((-B_1(t,y))^k - (-B_2(t,y))^k\Big)
    = \int_0^t\nabla \delta v(\tau,y)\,\dd\tau \Big(\sum_{k=1}^\infty\sum_{j=0}^{k-1}(-1)^kB_1^jB_2^{k-1-j}\Big),
\end{align}
with $B_i(t,y)= \int_0^t \nabla v_i(\tau,y)\dd \tau$, by letting $T_1>0$ small enough it can be controlled as follows
\begin{align}\label{est.f11}
  \|\delta f_1 \|_{L^2_{T_1}(L^2)} & \leq \|\Id-A_1\|_{L^\infty_{T_1}(L^\infty)} \|\nabla \delta\Pi\|_{L^2_{T_1}(L^2)}
  + \|\delta A\|_{L^\infty_{T_1}(L^{\frac {2p}{p-2}})} \|\nabla \Pi_2\|_{L^2_{T_1}(L^p)} \nonumber \\
  & \leq C T_1^{\frac{1}{2}} \|\nabla \delta\Pi\|_{L^2_{T_1}(L^2)}
  + C T_1^{\frac{1}{2}} \|\nabla \delta v\|_{L^2_{T_1}(L^{\frac {2p}{p-2}})} \nonumber \\
  & \leq C T_1^{\frac{1}{2}} \|\nabla \delta\Pi\|_{L^2_{T_1}(L^2)} + C T_1^{\frac{1}{2}} \|\delta v\|_{L^2_{T_1}(\dot H^{2\alpha})}
  + C T_1 \|\delta v\|_{L^\infty_{T_1}(\dot H^\alpha)},
\end{align}
where in the last line we also used \eqref{es.intp}.

Next, let us treat the term $\delta f_2$ given by \eqref{eq:f2}.
Observe that
\begin{align*}
  \delta f_2 & = c_\alpha\,  \int_{\RR^2} \frac{\big(X_{1,t}(y) -X_{1,t}(z)\big)\cdot \big(A^t_1(t,y) -A^t_2(t,y) - (A^t_1(t,z) - A^t_2(t,z)) \big) \nabla v_2(t,y)}{|X_{1,t}(y)-X_{1,t}(z)|^{2+2\alpha}} \dd z \\
  & \, + c_\alpha\,  \int_{\RR^2} \frac{\big(X_{1,t}(y) -X_{1,t}(z)\big)\cdot \big(A^t_1(t,z) - A^t_2(t,z)\big) \big(\nabla v_2(t,y)- \nabla v_2(t,z)\big)}{|X_{1,t}(y)-X_{1,t}(z)|^{2+2\alpha}} \dd z \\
  & \,+ c_\alpha\, \int_{\RR^2}\Big [\frac{X_{1,t}(y) -X_{1,t}(z)}{|X_{1,t}(y)-X_{1,t}(z)|^{2+2\alpha}} - \frac{X_{2,t}(y) - X_{2,t}(z)}{|X_{2,t}(y)-X_{2,t}(z)|^{2+\alpha}}\Big]\cdot
  \big(A^t_2(t,y) - A^t_2(t,z)\big)\nabla v_2(t,y) \dd z \\
  & \,  + c_\alpha\, \int_{\RR^2}\Big[\frac{X_{1,t}(y) -X_{1,t}(z)}{|X_{1,t}(y)-X_{1,t}(z)|^{2+2\alpha}} - \frac{X_{2,t}(y) - X_{2,t}(z)}{|X_{2,t}(y)-X_{2,t}(z)|^{2+\alpha}}\Big]\cdot
  A^t_2(t,z)(\nabla v_2(t,y) - \nabla v_2(t,z)) \dd z \\
  & :=  \delta f_2^1 + \delta f_2^2 + \delta f_2^3 + \delta f_2^4 .
\end{align*}
For $\delta f_2^1$, by changing of variables and using Lemma \ref{lem:f.mu}, we get
\begin{align*}
  \|\delta f_2^1\|_{L^2_y}^2 &= c^2_\alpha \int_{\RR^2} \Big| \int_{\RR^2} \frac{\big(X_{1,t}(y) -X_{1,t}(z)\big)\cdot
  \big(\delta A(t,y) - \delta A(t,z)) \big) }{|X_{1,t}(y)-X_{1,t}(z)|^{2+2\alpha}} \dd z \,\nabla v_2(t,y)\Big|^2 \dd y \\
  & = c^2_\alpha \int_{\RR^2} \Big| \int_{\RR^2} \frac{\big(x -\tilde{z}\big)\cdot
  \big( \delta A(t,X_{1,t}^{-1}(x)) - \delta A (t,X_{1,t}^{-1}(\tilde{z})) \big) }{|x -\tilde{z}|^{2+2\alpha}} \dd \tilde{z} \,\nabla v_2(t,X_{1,t}^{-1}(x))\Big|^2 \dd x  \\
  & = C \int_{\RR^2} \Big| \big[\Lambda^{2\alpha-2} \nabla \cdot
  (\delta A (t,X_{1,t}^{-1}(x))) \big] \,\nabla v_2(t,X_{1,t}^{-1}(x))\Big|^2 \dd x \\
  & \leq C \| \delta A(t,X_{1,t}^{-1}(x))\|_{\dot H^{2\alpha -1}}^2 \|\nabla v_2(t, X_{1,t}^{-1}(x))\|_{L^\infty}^2 \\
  & \leq C \|\delta A(t,y)\|_{\dot H^{2\alpha-1}_y}^2 \|\nabla v_2(t)\|_{L^\infty}^2.
\end{align*}
In view of \eqref{eq.del-Ai} and Lemma \ref{lem:frc-leib}, using estimates \eqref{es.Id-A.infty}--\eqref{es.intp}, we have
\begin{align}\label{es:dA-H}
  \big\|\delta A\big\|_{L^\infty_{T_1}(\dot H^{2\alpha-1})} & \leq C  \big\|\nabla \delta v\big\|_{L^1_{T_1} (\dot H^{2\alpha-1})}
  \sum_{k= 1}^\infty \sum_{j=0}^{k-1}\big\|B_1^j B_2^{k-1-j} \big\|_{L^\infty_{T_1} (L^{\infty})} \nonumber \\
  & \quad  + C \big\|\nabla \delta v\big\|_{L^1_{T_1} (L^{\frac{2p}{p-2}})} \sum_{k=1}^\infty \sum_{j=0}^{k-1} \big\|B_1^j B_2^{k-1-j} \big\|_{L^\infty_{T_1} (W^{2\alpha-1,p})} \nonumber \\
  & \leq C T^{\frac{1}{2}} \|\delta v\|_{L^2_{T_1}(\dot H^{2\alpha})} \sum_{k=1}^\infty k \Big( T_1^{\frac{1}{2}} \|(\nabla v_1,\nabla v_2)\|_{L^2_{T_1}(L^\infty)}\Big)^k \nonumber \\
  & \quad + C T_1^{\frac{1}{2}} \|\nabla \delta\|_{L^2_{T_1}(L^{\frac{2p}{p-2}})} \sum_{k=1}^\infty k \Big(C T_1^{\frac{1}{2}} \|(v_1,v_2)\|_{L^2_{T_1}(\dot B^{2\alpha}_{p,1}\cap \dot H^\alpha)} \Big)^k \nonumber \\
  & \leq C T^{\frac{1}{2}} \|\delta v\|_{L^2_{T_1}(\dot H^{2\alpha})} + T_1 \|\delta v\|_{L^\infty_{T_1}(\dot H^\alpha)},
\end{align}
where $T_1>0$ is small enough so that $C T^{\frac{1}{2}} \|(v_1,v_2)\|_{L^2_{T_1}(\dot B^{2\alpha}_{p,1}\cap \dot H^\alpha)}\leq \frac{1}{2}$.
Thus combining the above two estimates yields that
\begin{align}\label{es.df21}
  \|\delta f_2^1\|_{L^2_{T_1} (L^2)}  \leq  \|\delta A\|_{L^\infty_{T_1}(\dot H^{2\alpha-1})} \|\nabla v_2\|_{L^2_{T_1} (L^\infty)}
  \leq C T_1^{\frac{1}{2}} \|\delta v\|_{L^2_{T_1}(\dot H^{2\alpha})} + T_1 \|\delta v\|_{L^\infty_{T_1}(\dot H^\alpha)} .
\end{align}
Let us check the estimation of term $\delta f^4_2$. By denoting
\begin{equation*}
  h(t,y,z,\theta) := \theta \big(X_{1,t}(y)-X_{1,t}(z)\big)+(1-\theta)\big(X_{2,t}(y)-X_{2,t}(z)\big),
\end{equation*}
and using  Newton-Leibniz's formula, we find
\begin{align}\label{eq:X-form}
  &\frac{X_{1,t}(y) - X_{1,t}(z)}{|X_{1,t}(y) - X_{1,t}(z)|^{2+2\alpha}} - \frac{X_{2,t}(y) - X_{2,t}(z)}{|X_{2,t}(y)-X_{2,t}(z)|^{2+2\alpha}} \nonumber \\
  =&\int_0^1\frac {\mathrm d}{\mathrm d\theta} \frac{h(t,x,y,\theta)}{|h(t,x,y,\theta)|^{2+2\alpha}}\mathrm d\theta \nonumber\\
  =&-(1+2\alpha)\int_0^1\frac{1}{|h(x,y,t,\theta)|^{2 + 2\alpha} } \mathrm d\theta \cdot
  \Big( X_{1,t}(y) - X_{2,t}(y)  - \big(X_{1,t}(z)-X_{2,t}(z)\big)\Big) \nonumber  \\
  = &-(1+2\alpha)\int_0^1\frac{1}{|h(x,y,t,\theta)|^{2 + 2\alpha} } \mathrm d\theta \cdot
  \int_0^t\Big(\delta v(\tau,y)-\delta v(\tau,z)\Big) \mathrm d\tau .
\end{align}
This gives that
\begin{align}\label{eq:df24-exp}
  \delta f_2^4 = & - c_\alpha(1+2\alpha) \int_0^1 \int_0^t\int_{\RR^2}\frac {\big(\delta v(\tau,y)-\delta v(\tau,z)\big)
  \cdot A^t_2(t,z)(\nabla v_2(t,y) - \nabla v_2(t,z)) }{|h(t,x,y,\theta)|^{2+2\alpha}} \mathrm d z\mathrm d\tau \mathrm d\theta \nonumber \\
  = & - c_\alpha(1+2\alpha) \int_0^1 \int_0^1 \int_0^t\int_{\RR^2}\frac {(y-z)\cdot
  \nabla \delta v(\tau,\tilde\theta y+ (1-\tilde\theta)z)}{|h(t,x,y,\theta)|^{2+2\alpha}}\cdot \nonumber  \\
  & \qquad\qquad\qquad\qquad \qquad \qquad
  \cdot A_2(t,z)(\nabla v_2(t,y) - \nabla v_2(t,z)) \mathrm d z\mathrm d\tau \mathrm d\theta \dd \tilde{\theta} .
\end{align}
Noticing that
\begin{align*}
  |h(t,y,z,\theta)-(y-z)| & \leq \big(\theta \|\nabla v_1\|_{L^1_t(L^\infty)} + (1-\theta) \|\nabla v_2\|_{L^1_t (L^\infty)}\big) |y-z| \\
  & \leq T^{\frac{1}{2}}\|(\nabla v_1,\nabla v_2)\|_{L^2_t(L^\infty)} |y-z| ,
\end{align*}
by letting $T_1>0$ small enough, we get $h(t,y,z,\theta) \approx |y-z|$ for every $y\neq z$ and $t\leq T_1$.
Then taking advantage of Minkowski's inequality, Lemma \ref{lem:Besov-fd} and estimates \eqref{est.0L-2}, \eqref{es.intp}, we infer that
\begin{align}\label{est.f24}
  & \|\delta f_2^4\|_{L^2_{T_1}(L^2)}  \nonumber \\
  \lesssim & \Big\|\int_0^1\int_0^t \int_{\RR^2} \frac{|\nabla \delta v(\tau,\tilde{\theta} y +(1-\tilde{\theta})z)|
  |\nabla v_2(t,y) -\nabla v_2(t,z)|}{|y-z|^{1+2\alpha}} |A_2(t,z)| \dd z \dd \tau \dd \tilde{\theta}\Big\|_{L^2_{T_1}(L^2_y)} \nonumber \\
  \lesssim  & \Big\|\int_0^1\int_0^t \int_{\RR^2} \frac{|\nabla \delta v(\tau, \tilde{\theta} y +(1-\tilde{\theta})z)|
  |\nabla v_2(t,y) -\nabla v_2(t,y+z)|}{|z|^{1+2\alpha}}  \dd z \dd \tau \dd \tilde{\theta}\Big\|_{L^2_{T_1}(L^2_y)} \|A_2\|_{L^\infty_{T_1}(L^\infty)} \nonumber \\
  \lesssim &  \Big\|\int_0^1\int_0^t \int_{\RR^2} \frac{\| \nabla \delta v(\tau, \tilde{\theta} y + (1-\tilde{\theta})z)\,|\nabla v_2 (t,y) -\nabla v_2(t,y+z)| \|_{L^2_y}}{|z|^{1+2\alpha}} \dd z \dd \tau \dd \tilde{\theta}\Big\|_{L^2_{T_1}}  \nonumber \\
  \lesssim & \,T_1^{\frac{1}{2}} \Big\|\int_{\RR^2}\frac{\|\nabla v_2(t,y)-\nabla v_2(t, y+z)\|_{L^p_y}}{|z|^{1+2\alpha}} \dd z \Big\|_{L^2_{T_1}} \|\nabla \delta v\|_{L^2_{T_1}(L^{\frac {2p}{p-2}})} \nonumber \\
  \leq & C T_1^{\frac{1}{2}} \|\nabla v_2\|_{L^2_{T_1}( \dot B^{2\alpha-1}_{p,1})}
  \|\nabla \delta v\|_{L^2_{T_1}(L^{\frac {2p}{p-2}})} \nonumber\\
  \leq & C T_1^{\frac{1}{2}} \|\delta v\|_{L^2_{T_1}(\dot H^{2\alpha})} + C T_1 \|\delta v\|_{L^\infty_{T_1}(\dot H^\alpha)} .
\end{align}
The terms $\delta f_2^2$ and $\delta f_2^3$ can be estimated along the similar way as deducing \eqref{est.f24}:
due to \eqref{can.delta.sip} and \eqref{est.f11}, we have the following bounds
\begin{align}\label{es.df22}
  \|\delta f_2^2\|_{L^2_{T_1}(L^2)} & \leq C
  \Big\|\int_{\RR^2} \frac{|\delta A^t(t,y+z)|\, |\nabla v_2(t,y)-\nabla v_2(t,y+z)| }{|z|^{1+2\alpha}}
  \dd z \Big\|_{L^2_{T_1}(L^2_y)} \nonumber \\
  & \leq C \Big\|\int_{\RR^2} \frac{\|\nabla v_2(t,y)-\nabla v_2(t, y+z)\|_{L^p_y}}{|z|^{1+2\alpha}} \dd z \Big\|_{L^2_{T_1}}
  \|\delta A\|_{L^\infty_{T_1}(L^{\frac{2p}{p-2}})} \nonumber\\
  & \leq C T_1 \|\nabla v_2\|_{L^2_{T_1}( \dot B^{2\alpha-1}_{p,1})} \| \nabla \delta v\|_{L^2_{T_1}(L^{\frac{2p}{p-2}})} \nonumber \\
  & \leq C T_1^{\frac{1}{2}}  \| \delta v\|_{L^2_{T_1}( \dot H^{2\alpha})} + C T_1 \|\delta v\|_{L^\infty_{T_1}(\dot H^\alpha)} ,
\end{align}
and 
\begin{align*}
  & \quad \|\delta f_2^3\|_{L^2_{T_1} (L^2)} \nonumber \\
  &\leq  C \Big\|\int_0^1\int_0^t \int_{\RR^2}
  \frac{|\nabla \delta(\tau,\tilde{\theta} y + (1-\tilde{\theta})z)|\, |A_2(t,y)-A_2(t,y+z)|}{|z|^{1+2\alpha}}
  |\nabla v_2(t,y)|\dd z \dd \tau \dd \tilde{\theta} \Big\|_{L^2_{T_1}(L^2)} \nonumber \\
  & \leq CT_1^{\frac{1}{2}} \|\nabla \delta v\|_{L^2_{T_1}(L^{\frac{2p}{p-2}} )}
  \Big\|\int_{\RR^2} \frac{\|(\Id -A_2(t,y)) - (\Id - A_2(t,y+z))\|_{L^p_y}}{|z|^{1+2\alpha}} \dd z \Big\|_{L^\infty_{T_1}}
  \|\nabla v_2\|_{L^2_{T_1}(L^\infty )} \nonumber  \\
  & \leq CT_1^{\frac{1}{2}} \|\nabla \delta v\|_{L^2_{T_1}(L^{\frac{2p}{p-2}} )}
  \|\Id - A_2\|_{L^\infty_{T_1}(\dot B^{2\alpha-1}_{p,1})} \|\nabla v_2\|_{L^2_{T_1}(L^\infty )} \nonumber  \\
  & \leq C\big( T_1^{\frac{1}{2}} \|\delta v \|_{L^2_{T_1}(\dot H^{2\alpha})} + T_1 \|\delta v\|_{L^\infty_{T_1}(\dot H^\alpha)}\big)
  \|\Id -A_2\|_{L^\infty_{T_1}(\dot B^{2\alpha-1}_{p,1})}.
\end{align*}
By means of \eqref{est.prod2} and the embedding $B^{2\alpha-1}_{p,1}(\RR^2)\hookrightarrow L^\infty(\RR^2)$, and similarly as estimating \eqref{es.Id-A.wp},
we see that for $i=1,2$,
\begin{align*}
  \|\Id-A_i\|_{L^\infty_T(\dot B^{2\alpha-1}_{p,1})} & \leq \sum_{k=1}^\infty \Big(C T_1^{\frac{1}{2}}\|\nabla v_i\|_{L^2_{T_1}(B^{2\alpha-1}_{p,1})}\Big)^k
  \leq \sum_{k=1}^\infty \Big(C T_1^{\frac{1}{2}}\| v_i\|_{L^2_{T_1}(\dot B^{2\alpha}_{p,1}\cap \dot H^\alpha)}\Big)^k \leq C T^{\frac{1}{2}} ,
\end{align*}
so this implies
\begin{align}\label{es.df23-2}
  \|\delta f_2^3\|_{L^2_{T_1}(L^2)} \leq   C T_1^{\frac{1}{2}} \|\delta v \|_{L^2_{T_1}(\dot H^{2\alpha})} + C T_1 \|\delta v\|_{L^\infty_{T_1}(\dot H^\alpha)} .
\end{align}
Collecting the above estimates on $f_2^1$\,-\,$f_2^4$ yields that for $T_1>0$ small enough,
\begin{align}\label{es.df2}
  \|\delta f_2\|_{L^2_{T_1}(L^2)} \leq C T_1^{\frac{1}{2}} \|\delta v \|_{L^2_{T_1}(\dot H^{2\alpha})} + C T_1 \|\delta v\|_{L^\infty_{T_1}(\dot H^\alpha)} .
\end{align}

Next we consider the estimation related to $\delta g = (\Id - A_1)\delta v - (\delta A)\, v_2$ given by \eqref{eq:g}.
The algebraic relation \eqref{est.notation-Lagr} implies that
\begin{equation*}\label{eq.divg}
  \Div\delta g = (\Id-A_1^t):\nabla\delta v-(\delta A^t):\nabla v_2,
\end{equation*}
thus by using Lemma \ref{lem:frc-leib}, along with \eqref{es.Id-A.infty}--\eqref{es.intp} and \eqref{est.f11}--\eqref{es:dA-H}, we deduce that
\begin{align}\label{est.divg}
  \|\Div\delta g\|_{L^2_{T_1}(\dot H^{2\alpha-1})} & \leq \|(\Id-A_1^t):\nabla \delta v\|_{L^2_{T_1}(\dot H^{2\alpha-1})}
  + \|\delta A^t : \nabla v_2\|_{L^2_{T_1}(\dot H^{2\alpha -1})} \nonumber\\
  & \leq C \|\Id-A_1\|_{L^\infty_{T_1}(\dot W^{2\alpha-1,p})} \|\nabla \delta v\|_{L^2_{T_1}(L^{\frac {2p}{p-2}})}
  + C \|\Id-A_1\|_{L^\infty_{T_1}(L^\infty)} \|\delta v\|_{L^2_{T_1}(\dot H^{2\alpha})} \nonumber\\
  & \quad + C \|\delta A\|_{L^\infty_{T_1}(\dot H^{2\alpha-1})} \|\nabla v_2\|_{L^2_{T_1}(L^\infty)}
  + C \| \delta A\|_{L^\infty_{T_1}(L^{\frac {2p}{p-2}})} \| \nabla v_2\|_{L^2_{T_1}(\dot W^{2\alpha-1,p})} \nonumber \\
  & \leq C T_1^{1/2} \| \delta v\|_{L^2_{T_1}(\dot H^{2\alpha})} + T_1 \| \delta v\|_{L^\infty_{T_1}(\dot H^\alpha)} .
\end{align}

We split $\partial_t\delta g$ into the following four terms:
\begin{align*}
  \partial_t\delta g&=\partial_t[(\Id-A_1)\delta v- (\delta A) \,v_2]\\
  &= -(\partial_t A_1)\,\delta v + (\Id-A_1)\,\partial_t\delta v- (\partial_t\delta A)\, v_2- (\delta A)\,\partial_tv_2.
\end{align*}
Noting that for $i=1,2$ (from \eqref{eq.A_i})
\begin{align}\label{eq:par-t-A}
  \partial_t A_i(t,y) =\sum_{k=1}^\infty (-1)^k k \big(B_i(t,y)\big)^{k-1}\,\nabla v_i(t,y) ,
  \quad B_i(t,y)=\int_0^t \nabla v_i(\tau,y)\dd \tau ,
\end{align}
and
\begin{align}\label{eq:dv-L2}
  \|\delta v(t,y)\|_{L^2_y} =  \|\delta v(t,y) -\delta v(0,y)\|_{L^2_y} \leq \int_0^t \|\partial_\tau\delta v(\tau,y)\|_{L^2_y} \dd \tau
  \leq t^{\frac{1}{2}} \|\partial_\tau \delta v\|_{L^2_t (L^2)},
\end{align}
by letting $T_1>0$ small enough we find
\begin{align}\label{est.g1_t}
  \|\partial_t A_1\,\delta v \|_{L^2_{T_1}(L^2)} & \leq \|\partial_t A_1\|_{L^2_{T_1}(L^\infty)}
  \|\delta v\|_{L^\infty_{T_1}(L^2)} \nonumber \\
  & \leq C T_1^{1/2} \|\nabla v_1\|_{L^2_{T_1}(L^\infty)} \|\partial_t \delta v\|_{L^2_{T_1}(L^2)}
  \leq C T_1^{1/2} \|\partial_t \delta v\|_{L^2_{T_1}(L^2)}.
\end{align}
Thanks to \eqref{es.Id-A.infty} and \eqref{eq.del-Ai}, we immediately get
\begin{equation}\label{est.g2_t}
  \|(\Id -A_1)\, \partial_t \delta v\|_{L^2_{T_1}(L^2)} \leq \|\Id-A_1\|_{L^\infty_{T_1}(L^\infty)} \|\partial_t\delta v\|_{L^2_{T_1}(L^2)}
  \leq C T_1^{1/2} \|\partial_t\delta v\|_{L^2_{T_1}(L^2)} ,
\end{equation}
and
\begin{align}\label{est.g4_t1}
  \|(\delta A)\,\partial_t v_2\|_{L^2_{T_1}(L^2)} & \leq \|\delta A\|_{L^\infty_{T_1}(L^{\frac{2p}{p-2}})} \|\partial_t v_2\|_{L^2_{T_1}(L^p)} \nonumber \\
  & \leq C T_1^{1/2}\|\nabla \delta v\|_{L^2_T(L^{\frac {2p}{p-2}})} \leq C T_1^{1/2} \|\delta v\|_{L^2_{T_1}(\dot H^{2\alpha})}
  + C T_1 \|\delta v\|_{L^\infty_{T_1}(\dot H^\alpha)}.
\end{align}
In view of the following formula (from \eqref{eq:par-t-A})
\begin{align}
  \partial_t \delta A(t,y) = & - \nabla \delta v(t,y)\, + \sum_{k=2}^\infty (-1)^k k B_2^{k-1} \,\nabla \delta v(t,y) \nonumber \\
  & +\sum_{k=2}^\infty\sum_{ j=0}^{k-2} (-1)^k k B_1^{j-1}B_2^{k-1-j}\,\Big(\int_0^t\nabla \delta v(\tau,y)\,\dd\tau\Big)\, \nabla v_1(t,y).\nonumber
\end{align}
and using the Gagliardo-Nirenberg inequality
$\|\nabla \delta v\|_{L^2}\leq C \| \delta v\|^{2-\frac 1\alpha }_{\dot H^\alpha}\|\delta v\|_{\dot H^{2\alpha}}^{\frac 1\alpha -1}$,
we infer that
\begin{align}\label{est.g3_t2}
  \|(\partial_t \delta A)\, v_2\|_{L^2_{T_1}(L^2)}
  & \leq \|\partial_t \delta A\|_{L^2_{T_1}(L^2)} \|v_2\|_{L^\infty_{T_1}(L^\infty)} \nonumber \\
  & \leq C \|v_2\|_{L^\infty_{T_1}(L^\infty)} \big(\|\nabla \delta v\|_{L^2_{T_1}(L^2)}
  + \|\nabla \delta v\|_{L^1_{T_1}(L^2)} \|\nabla v_1\|_{L^2_{T_1}(L^\infty)}\big) \nonumber \\
  & \leq C T_1^{1-\frac{1}{2\alpha}} \big(\|\delta v\|_{L^\infty_{T_1}(\dot H^\alpha)}
  + \|\delta v\|_{L^2_{T_1}(\dot H^{2\alpha})}\big).
\end{align}
Noticing that $T_1\leq 1$ and $\alpha\in (\frac{1}{2},1)$, we collect estimates \eqref{est.g1_t}--\eqref{est.g3_t2} to obtain
\begin{align}\label{es.dg}
  \|\partial_t \delta g\|_{L^2_{T_1}(L^2)} \leq C T_1^{\frac{1}{2}} \|\partial_t\delta v\|_{L^2_{T_1}(L^2)}
  + C  T_1^{1-\frac{1}{2\alpha}} \big( \|\delta v\|_{L^2_{T_1}(\dot H^{2\alpha})}
  + \|\delta v\|_{L^\infty_{T_1}(\dot H^\alpha)} \big).
\end{align}

Therefore, plugging inequalities \eqref{est.f11}, \eqref{es.df2}, \eqref{est.divg} and \eqref{es.dg} into \eqref{est.sta},  
we find that for $T_1\in(0,1]$ small enough,
\begin{align}\label{est.dE'}
  \delta E(T_1) & \leq C T_1^{\frac{1}{2}} \|(\partial_t\delta v, \nabla \delta\Pi)\|_{L^2_{T_1}(L^2)}
  + C T_1^{1-\frac{1}{2\alpha}} \big( \|\delta v\|_{L^2_{T_1}(\dot H^{2\alpha})}
  + \|\delta v\|_{L^\infty_{T_1}(\dot H^\alpha)} \big) \nonumber \\
  & \leq C  T_1^{1-\frac{1}{2\alpha}} \delta E(T_1).
\end{align}
By letting $T_1>0$ be a even smaller constant (if necessarily) so that $C T_1^{1-\frac{1}{2\alpha}} \leq \frac{1}{2}$,
we conclude that $\delta E(t)\equiv 0$ on $[0,T_1]$.
The Sobolev inequality $\dot H^\alpha(\RR^2) \hookrightarrow L^{\frac{2}{1-\alpha}}(\RR^2)$ or estimate \eqref{eq:dv-L2}
further implies that $\delta v\equiv 0$ for a.e. $\RR^2\times [0,T_1]$. By using \eqref{eq:Xit} and coming back to the Eulerian coordinates,
we also get $X_{1,t}(y)\equiv X_{2,t}(y)$ and $u_1(t,x)\equiv u_2(t,x)$ on a.e. $\RR^2\times [0,T_1]$.

Repeating the above procedure and arguing as the corresponding part in \cite{MXZ19},
we can further prove $u_1 = u_2$ on $\RR^2\times [T_1,2T_1]$, $\RR^2\times [2T_1,3T_1]$, $\cdots$,
where $T_1>0$ is a small constant depending only on $\alpha,p,s$ and the norms of $(u_i,\pi_i)$ in Propositions \ref{pro.propri} and \ref{prop:uni.M.bet}.
Hence the uniqueness part of Theorem \ref{Main.thm} is proved.

\vskip0.2cm

\textbf{Acknowledgments.}
YL is partly supported by the NNSF of China (No. 12071043).
QM is partly supported by the NNSF of China (No. 12001041).
LX is partly supported by the NNSF of China (No. 11771043) and National Key Research and Development Program of China (No. 2020YFA0712900).

\bibliographystyle{plain}

\begin{thebibliography}{100}


%



\bibitem{Apple09}
D.~Applebaum,
{\em L\'{e}vy processes and stochastic calculus}, volume 116 of
{\em Cambridge Studies in Advanced Mathematics}.
Cambridge University Press, Cambridge, second edition, (2009).

\bibitem{BCD11}
H.~Bahouri, J.-Y. Chemin and R.~Danchin,
{\em Fourier Analysis and Nonlinear Partial Differential Equations},
Grundlehren der Mathematischen Wissenschaften.
Springer, Heidelberg, (2011).


\bibitem{BVV80} H. Beir\~ao da Veiga and A. Valli, Existence of $C^\infty$ solutions of the Euler equations for nonhomogeneous fluids.
\emph{Communications in Partial Differential Equations}, 5 (1980), no. 2, 95--107.



\bibitem{CV10} L. Caffarelli and  V. Vasseur, Drift diffusion equations with fractional diffusion and the quasi-geostrophic equations.
  \emph{Annals of Math.}, 171 (2010), no. 3, 1903--1930.


\bibitem{CaoChLai19}
X.~Cao, Q.~Chen and B.~Lai,
Properties of the linear non-local {S}tokes operator and its
  application.
{\em Nonlinearity}, 32 (2019), no. 7, 2633--2666.

%



\bibitem{CTT20} L. Chen, C. Tan and L. Tong,
On the global classical solution to compressible Euler system with singular velocity alignment.
ArXiv:2007.08356v1 [math.AP].


\bibitem{ChenLi20}
Q.~Chen and Y.~Li.
Global regularity of density patch for the 3{D} inhomogeneous
  {N}avier-{S}tokes equations.
{\em Z. Angew. Math. Phys.}, 71 (2020), no. 2, Art. 40.

\bibitem{CK03} H. J. Choe and H. Kim, Strong solutions of the Navier-Stokes equations for nonhomogeneous incompressible fluids.
\emph{Comm. Partial Differential Equations}, 28 (2003), no. 5-6, 1183--1201.

\bibitem{CDLM20} M. Colombo, C. De Lellis and A. Massaccesi, The generalized Caffarelli-Kohn-Nirenberg theorem for the hyperdissipative Navier-Stokes system.
\emph{Comm. Pure Appl. Math.}, 73 (2020), no. 3, 609--663.


\bibitem{Con15} P. Constantin, Lagrangian-Eulerian methods for uniqueness in hydrodynamic systems.
\emph{Adv. Math.}, 278 (2015), 67--102.

\bibitem{CJ19} P. Constantin and L. Joonhyun, Note on Lagrangian-Eulerian methods for uniqueness in hydrodynamic systems.
\emph{Adv. Math.}, 345 (2019), 27--52.



\bibitem{CDR20} P. Constantin, T. D. Drivas and R. Shvydkoy,
Entropy Hierarchies for equations of compressible fluids and self-organized dynamics.
\emph{SIAM J. Math. Anal.}, 52 (2020), no. 3, 3073--3092.


\bibitem{Danch03}
R.~Danchin,
Density-dependent incompressible viscous fluids in critical spaces.
{\em Proc. Roy. Soc. Edinburgh Sect. A}, 133 (2003), no. 6, 1311--1334.


\bibitem{Dan10}
R.~Danchin,
On the well-posedness of the incompressible density-dependent {E}uler
  equations in the {$L^p$} framework.
{\em J. Differential Equations}, 248 (2010), no. 8, 2130--2170.


\bibitem{DanFra11}
R.~Danchin and F.~Fanelli,
The well-posedness issue for the density-dependent {E}uler equations
  in endpoint {B}esov spaces.
{\em J. Math. Pures Appl.}, 96 (2011), no. 3, 253--278.

\bibitem{DanchMucha09}
R.~Danchin and P. Mucha,
A critical functional framework for the inhomogeneous {N}avier-{S}tokes equations in the half-space.
{\em J. Funct. Anal.}, 256 (2009), no. 3, 881--927.


\bibitem{DanchMucha12}
R.~Danchin and P. Mucha,
A {L}agrangian approach for the incompressible {N}avier-{S}tokes  equations with variable density.
{\em Comm. Pure Appl. Math.}, 65 (2012), no. 10, 1458--1480.

\bibitem{DanM13}
R.~Danchin and P. Mucha,
Incompressible flows with piecewise constant density.
{\em Arch. Ration. Mech. Anal.}, 207 (2013), no. 3, 991--1023.


\bibitem{DanchMuch19}
R.~Danchin and P. Mucha,
The incompressible Navier-Stokes equations in vacuum.
{\em Comm. Pure Appl. Math.}, 72 (2019), no. 7, 1351--1385.



\bibitem{DanchZhxin17}
R.~Danchin and X.~Zhang,
On the persistence of {H}\"{o}lder regular patches of density for the inhomogeneous {N}avier-{S}tokes equations.
{\em J. \'{E}c. polytech. Math.}, 4 (2017), 781--811.



\bibitem{FangZi13}
D.~Fang and R.~Zi,
On the well-posedness of inhomogeneous hyperdissipative {N}avier-{S}tokes equations.
{\em Discrete Contin. Dyn. Syst.}, 33 (2013), no. 8, 3517--3541.

%




\bibitem{Gancedo18}
F.~Gancedo and E.~Garc\'{i}a-Ju\'{a}rez,
Global regularity of 2{D} density patches for inhomogeneous {N}avier-{S}tokes.
{\em Arch. Ration. Mech. Anal.}, 229 (2018), no. 1, 339--360.

\bibitem{Giga85} Y. Giga, Domains of fractional powers of the Stokes operator in Lr spaces.
\emph{Arch. Rational Mech. Anal.}, 89 (1985), no. 3, 251--265.
%

\bibitem{GigS91} Y. Giga and H. Sohr, Abstract $L_p$ estimates for the Cauchy problem with applications to the Navier-Stokes equations in exterior domains.
\emph{J. Funct. Anal.}, 102 (1991), no. 1, 72--94.

%


\bibitem{Grafakos14}
L.~Grafakos,
{\em Modern {F}ourier analysis.} Third edition.
Springer, New York, New York, (2014).


\bibitem{HuangPaicuZhang13}
J.~Huang, M.~Paicu, and P.~Zhang,
Global well-posedness of incompressible inhomogeneous fluid systems with bounded density or non-{L}ipschitz velocity.
{\em Arch. Ration. Mech. Anal.}, 209 (2013), no. 2, 631--682.

\bibitem{HuangPZh13}
J.~Huang, M.~Paicu, and P.~Zhang.
Global solutions to 2-{D} inhomogeneous {N}avier-{S}tokes system with general velocity.
{\em J. Math. Pures Appl.}, 100 (2013), no. 6, 806--831.

\bibitem{Iwas89} H. Iwashita, $L_q$-$L_r$ estimates for solutions of the nonstationary Stokes equations in an exterior domain and the Navier-Stokes initial value problems in $L_q$ spaces.
\emph{Math. Ann.}, 285 (1989), 265--288.

\bibitem{KMT15} T. Karper, A. Mellet and K. Trivisa, Hydrodynamic limit of the kinetic Cucker-Smale flocking model.
\emph{Math. Mod. Meth. Appl. Sci.}, 25 (2015), 131--163.

\bibitem{Lady75}
O.~Lady\v{z}enskaja and V.~Solonnikov,
The unique solvability of an initial-boundary value problem for viscous incompressible inhomogeneous fluids.
{\em Zap. Nau\v{c}n. Sem. Leningrad. Otdel. Mat. Inst. Steklov. (LOMI)}, 52 (1975), 218--219, 52--109.

\bibitem{LMZ19} B. Lai, C. Miao and X. Zheng, Forward self-similar solutions of the fractional Navier-Stokes equations.
\emph{Adv. Math.}, 352 (2019), 981--1043.

\bibitem{Laskin00}
N.~Laskin,
Fractional quantum mechanics and {L}\'{e}vy path integrals.
{\em Phys. Lett. A}, 268 (2000), no. 4-6, 298--305.


\bibitem{LRie02}
P.~G. Lemari\'{e}-Rieusset,
{\em Recent developments in the {N}avier-{S}tokes problem}, volume
  431 of {\em Chapman $\&$ Hall/CRC Research Notes in Mathematics}.
Chapman $\&$ Hall/CRC, Boca Raton, FL, (2002).


\bibitem{LiaoLiu12}
X.~Liao and Y.~Liu,
Global regularity of three-dimensional density patches for inhomogeneous incompressible viscous flow.
{\em Sci China Mat}, 62 (2019), 1749--1764.

\bibitem{LiaoZhang18}
X.~Liao and P.~Zhang,
Global regularity of 2{D} density patches for viscous inhomogeneous incompressible flow with general density: low regularity case.
{\em Comm. Pure Appl. Math.}, 72 (2019), no. 4, 835--884.

\bibitem{LiaoZhang19}
X.~Liao and P.~Zhang,
Global regularity of 2-{D} density patches for viscous inhomogeneous incompressible flow with general density: high regularity case.
{\em Anal. Theory Appl.}, 35 (2019), no. 2, 163--191.

\bibitem{PLions96}
P.~Lions,
\emph{Mathematical topics in fluid mechanics. Vol. 1. Incompressible models}.
Oxford Lecture Series in Mathematics and its Applications, 3. Oxford Science Publications. The Clarendon Press, Oxford University Press, New York, (1996).


\bibitem{MGSIG} J. M. Mercado, E. P. Guido, A. J. S\'anchez-Sesma, M. {\'I}{\~n}iguez and A. Gonz\'alez,
Analysis of the Blasius formula and the Navier-Stokes fractional equation. In: J. Klapp, et al. (eds.) Fluid Dynamics in Physics,
Engineering and Environmental Applications, pp. 475--480. Springer, Berlin (2013).

\bibitem{MetKla00}
R.~Metzler and J.~Klafter,
The random walk's guide to anomalous diffusion: a fractional dynamics approach.
{\em Phys. Rep.}, 339 (2000), no. 1, 1--77.

\bibitem{MiaoYZh08}
C.~Miao, B.~Yuan, and B.~Zhang,
Well-posedness of the Cauchy problem for the fractional power dissipative equations.
{\em Nonlinear Anal.}, 68 (2008), no.3, 461--484.

\bibitem{Mucha01}
P.~Mucha,
Stability of nontrivial solutions of the {N}avier-{S}tokes system on the three dimensional torus.
{\em J. Differential Equations}, 172 (2001), no. 2, 359--375.

\bibitem{Mucha08}
P.~Mucha,
Stability of 2{D} incompressible flows in {$\mathbb{R}^3$}.
{\em J. Differential Equations}, 245 (2008), no. 9, 2355--2367.

\bibitem{MXZ19}
P.~Mucha, L.~Xue, and X.~Zheng,
Between homogeneous and inhomogeneous {N}avier-{S}tokes systems: the issue of stability.
{\em J. Differential Equations}, 267 (2019), no. 1, 307--363.


%

\bibitem{PaicuZhZh13}
M.~Paicu, P.~Zhang, and Z.~Zhang,
Global unique solvability of inhomogeneous {N}avier-{S}tokes equations with bounded density.
{\em Comm. Partial Differential Equations}, 38 (2013), no. 7, 1208--1234.






\bibitem{Simon90}
J.~Simon,
Nonhomogeneous viscous incompressible fluids: existence of velocity, density, and pressure.
{\em SIAM J. Math. Anal.}, 21 (1990), no. 5, 1093--1117.

\bibitem{SZF95} M. F. Shlesinger, G. M. Zaslavsky, U. Frisch, (eds.), L\'evy filghts and related topics in physics.
Lect. Notes in Physics, Vol. 450, Berlin: Springer-Verlag, (1995).

\bibitem{Stein70}
E.~Stein,
{\em Singular integrals and differentiability properties of functions}.
Princeton Mathematical Series, No. 30. Princeton University Press, Princeton, N.J., (1970).

\bibitem{WangYe18}
 D.~Wang and  Z.~Ye,
Global existence and exponential decay of strong solutions for the inhomogeneous incompressible Navier-Stokes equations with vacuum.
ArXiv:1806.04464v1 [math.AP], 2018

\bibitem{Wo01}
W.~Woyczy\'{n}ski,
L\'{e}vy processes in the physical sciences.
In {\em L\'{e}vy processes}, pages 241--266. Birkh\"{a}user Boston, Boston, MA, (2001).

\bibitem{WuYuan08}
G.~Wu and J.~Yuan,
Well-posedness of the {C}auchy problem for the fractional power dissipative equation in critical {B}esov spaces.
{\em J. Math. Anal. Appl.}, 340(2008), no. 2, 1326--1335.

\bibitem{Wu05} J. Wu, Lower bounds for an integral involving fractional Laplacians and the generalized Navier-Stokes equations in Besov spaces.
\emph{Commun. Math. Phys.}, 263 (2005), 803--831.


\bibitem{Zhang12} X. Zhang, Stochastic Lagrangian particle approach to fractal Navier-Stokes equations.
\emph{Comm. Math. Phys.}, 311 (2012), 133--155.
\end{thebibliography}

\end{document}